\documentclass[12pt,a4paper]{article}

\usepackage{amsmath,amssymb}
\usepackage{latexsym}
\usepackage{amsmath}
\usepackage{amssymb}
\usepackage{amsthm}
\usepackage{amscd}
\usepackage{mathrsfs}
\usepackage[all]{xy}
\usepackage{graphicx}

\usepackage{bm}
\usepackage{comment}

\usepackage{mathrsfs}

\newtheorem{theorem}{Theorem}[section]
\newtheorem{lemma}[theorem]{Lemma}
\newtheorem{corollary}[theorem]{Corollary}
\newtheorem{proposition}[theorem]{Proposition}

\theoremstyle{definition}

\newtheorem{remark}[theorem]{Remark}
\newtheorem{definition}[theorem]{Definition}

\theoremstyle{remark}
\newtheorem*{notation*}{\bf Notation}
 \newtheorem*{acknowledgment*}{ \bf Acknowledgment}

\makeatletter

\@addtoreset{equation}{section}
\makeatother

\DeclareMathOperator{\Spec}{Spec}

\DeclareMathOperator{\Gal}{Gal}

\DeclareMathOperator{\Spa}{Spa}
\DeclareMathOperator{\Spf}{Spf} 
\DeclareMathOperator{\Sp}{Sp}

\DeclareMathOperator{\Ind}{Ind} 
\DeclareMathOperator{\Tr}{Tr}
\DeclareMathOperator{\Nr}{Nr}
\DeclareMathOperator{\Trd}{Trd}
\DeclareMathOperator{\Nrd}{Nrd}

\DeclareMathOperator{\Image}{Im}
\DeclareMathOperator{\cInd}{\mathrm{c\mathchar`-Ind}}

\def\La{\Lambda}

\newcommand{\C}{{\bf C}}
\newcommand{\Z}{{\bf Z}}
\newcommand{\Y}{{\bf Y}}
\newcommand{\X}{{\bf X}}
\newcommand{\W}{{\bf W}}

\newcommand{\cK}{{\cal K}}

\renewcommand{\bigskip}{\vspace{0.2cm}}
\newcommand{\ch}{\mathrm{char}\ }

\usepackage{amsmath}
\usepackage{graphicx}

\makeatletter
\def\widebreve#1{\mathop{\vbox{\m@th\ialign{##\crcr\noalign{\kern\p@}%
  \brevefill\crcr\noalign{\kern0.1\p@\nointerlineskip}%
  $\hfil\displaystyle{#1}\hfil$\crcr}}}\limits}

\def\brevefill{$\m@th \setbox\z@\hbox{}%
 \hfill\scalebox{0.7}{\rotatebox[origin=c]{90}{(}} \kern4pt $}
\makeatletter

\usepackage{amsmath}
\usepackage{graphicx}
\usepackage{bm}
\newcommand{\bom}[1]{\mbox{\boldmath $#1$}}

\begin{document}
\title{On non-abelian Lubin-Tate theory 
for $\mathrm{GL}(2)$ \\ 
in the odd equal characteristic case}
\author{Takahiro Tsushima}
\date{}
\maketitle

\footnotetext{2010 \textit{Mathematics Subject Classification}. 
 Primary: 11G25; Secondary: 11F80.}

\begin{abstract} 
In this paper, 
we define a
 family of affinoids in the tubular 
 neighborhoods
 of CM points in 
 the Lubin-Tate curve with 
 suitable level structures, and  
 compute the reductions of 
them 
in the equal characteristic case. 
By using \'etale cohomology theory of 
adic spaces due to 
Huber, we show that 
the cohomology of the reductions 
contributes to the cohomology of the Lubin-Tate curve.
As an application, with the help of 
explicit descriptions of the local Langlands 
correspondence and the local Jacquet-Langlands 
correspondence 
due to Bushnell-Henniart via the theory of types, 
we prove the non-abelian Lubin-Tate theory 
for $\mathrm{GL}(2)$ in 
the odd equal characteristic case in a purely 
local manner. 
Conversely, 
if we admit the non-abelian Lubin-Tate theory, 
we can recover the explicit descriptions of the 
two correspondences geometrically. 
\end{abstract}
\section{Introduction}
It is known that the cohomology of 
Lubin-Tate spaces simultaneously realizes 
the local Langlands correspondence (LLC, shortly)
and the local Jacquet-Langlands correspondence 
(LJLC, shortly)
for general linear groups over non-archimedean local
fields. This is conjectured 
in \cite{Ca}. 
This is called the Deligne-Carayol conjecture or 
the non-abelian Lubin-Tate theory (NALT, for shortly). 
This has been proved in \cite{Bo} in 
the equal characteristic case 
and \cite{HT} 
in the mixed characteristic case respectively. 
The proofs depend on global automorphic 
representations, and 
Shimura variety or Drinfeld 
modular variety. 
In \cite{St2}, \cite{Mi0} and \cite{Mi2}, 
by a purely local and geometric method, 
the LJLC is proved to be realized in 
the cohomology of the Lubin-Tate space.  
A purely local and geometric proof of the 
corresponding assertion for the LLC 
is not known. 
Since the NALT is a generalization of the 
Lubin-Tate theory, it 
is hoped that it is proved in a local approach. 

To obtain such a proof, 
it is necessary 
to understand the cohomology 
of the Lubin-Tate space 
with respect to the cohomology 
of the reductions of some 
affinoid subdomains in it (cf.\ \cite{Ha}).  
In this paper, in the odd equal 
characteristic case, 
we define a 
family of affinoids near 
CM points in the Lubin-Tate tower, 
determine the reductions of them, and analyze the middle cohomology of them. 
By relating such analysis to 
the cohomology of the Lubin-Tate curve
using \'etale cohomology theory of rigid analytic varieties given in \cite{Huet}, 
we obtain a new geometric proof of 
the NALT for $\mathrm{GL}(2)$ without depending on 
global automorphic representations and geometry of 
Drinfeld modular curves.

In \cite{WeSemi}, 
Weinstein defines a family of affinoids
in the Lubin-Tate perfectoid curve, 
and determines the reductions of them. 
 As a result, 
 with the help of the NALT, 
 he classifies irreducible components 
 into four types 
 in the stable reductions of the Lubin-Tate curves
 with Drinfeld level structures  
 up to purely inseparability.   

In this paper, 
without using perfectoid space, 
we explicitly 
compute the reductions of affinoids
in the Lubin-Tate curves with finite level structures
of two types 
in the equal characteristic case. 
In the unramified case, we use 
the Lubin-Tate curve 
with Drinfeld level structures. 
This curve has a nice simple formal
model \eqref{sss}, which is very compatible 
with the points which have complex 
multiplication by the ring of integers in the 
quadratic unramified extension.   
In the ramified case, 
we use the Lubin-Tate curve with 
Iwahori level structures whose 
formal model is given in 
\cite{FGL}. 
Its natural formal
model \eqref{in} is very compatible 
with the points which have complex 
multiplication by the rings of integers in 
quadratic ramified extensions.  
In 
the equal characteristic case, 
the formal model \eqref{in} is 
described 
through the theory of coordinate modules, 
which is developed in \cite{Ge}. 
A group action on this model is very 
explicitly described in \cite{FGL}. 
This model is originally used to explicitly compare  
Lubin-Tate tower 
with Drinfeld tower. 
Also for our purpose, 
this model plays a very nice role. 

The image of the affinoid by 
the level lowering 
map between Lubin-Tate curves 
equals a closed disk. 
This fact implies that 
the reductions are 
considered to be 
``new components''
in each level. 
It seems difficult  to determine 
all ``old components'' in the stable reduction
of Lubin-Tate tower without cohomological 
understanding of Lubin-Tate tower. 
In our proof of the NALT for $\mathrm{GL}(2)$, 
it is unnecessary to compute all irreducible 
components
in the stable reduction
of Lubin-Tate tower. This is a key point 
in our proof. 

In \S \ref{2}, we define affinoids 
 near 
CM points and 
compute 
the reductions of the affinoids. The reductions are classified into three
types (cf.\ Lemma \ref{xy}, Proposition \ref{rx}
and Proposition \ref{rz}).  
This classification fits into 
the one given in \cite[Theorem 1.0.1]{WeSemi}.
The other one except for the three is a projective 
line. Since the middle cohomology 
of it is trivial, we will not consider. 
As for several previous works 
on types of irreducible components 
in the stable reductions of modular curves or 
Lubin-Tate curves, 
see \cite[Introduction]{WeSemi}.   
The idea to study 
affinoids near CM points in order  
to understand stable coverings of 
modular curve or Lubin-Tate curve appears in 
\cite[Corollary 4.2]{CMc2} first in the literature. 
This idea is used also in 
\cite{WeSemi} at infinity level in a general setting. 
Originally, 
CM points on Lubin-Tate spaces are studied in 
\cite{G} and \cite{GH} (cf.\ \cite{Fa}).

In \S \ref{4}, we calculate 
the group action on the reductions in \S \ref{2}, which 
is induced by the 
group action on the Lubin-Tate tower. 
Similar descriptions 
are predicted in \cite{WeLT}. 
In \cite{WeLT}, 
he constructs a stable curve over a finite field with
appropriate group action, 
whose middle cohomology 
realizes the NALT for $\mathrm{GL}(2)$.  

In \S \ref{33}, 
we collect some known facts 
on the first cohomology of some curves. 
To relate the cohomology of the reductions of 
the affinoids to the cohomology of the Lubin-Tate curve, Lemma \ref{top} plays a key role. 
Let $K$ be a non-archimedean local field. 
Let $p$ be the residue characteristic of $K$. 
By this lemma, for a rigid analytic variety 
$\X$ over $K$ and 
its affinoid subdomain $\W$, 
if $\W$ has good reduction, under 
some condition on the cohomology of the canonical reduction 
$\overline{\W}$ of $\W$, 
we can relate the cohomology of the reduction 
$\overline{\W}$ to  
the cohomology of the total space $\X$. 
This fact is an immediate consequence of 
a comparison theorem 
between formal nearby cycles and 
usual nearby cycles 
in 
\cite[Theorem 0.7.7]{Huet} (cf.\ \eqref{formal1/2}).
Note that Lemma \ref{top} works for
any dimensional case. 
It could be possible to prove it 
also by using Berkovich's results in \cite{Be}
and \cite{Be2}.  
Let $\ell \neq p$ be a prime number. 
For 
the reductions 
$\overline{\W}$ 
of the affinoids $\W$ 
in the Lubin-Tate curve
in \S \ref{2}
 except for the 
level zero case, 
the canonical map 
$f \colon H_{\rm c}^1(\overline{\W}_{\overline{\mathbb{F}}_p},\overline{\mathbb{Q}}_{\ell}) \to 
H^1(\overline{\W}_{\overline{\mathbb{F}}_p},\overline{\mathbb{Q}}_{\ell})
$ is an isomorphism.
This follows from the fact that 
the reduction is isomorphic to 
a curve of Artin-Schreier 
type associated to a monomial 
(cf.\ Corollary \ref{lc1}). 
For the level zero case, the map 
$f$ becomes an injection  
on the cuspidal part  of $H_{\rm c}^1(\overline{\W}_{\overline{\mathbb{F}}_p},\overline{\mathbb{Q}}_{\ell})$ by Lemma \ref{DL}.1.
The cuspidal part is well-understood 
through a small part of the Deligne-Lusztig theory in \cite{DL}.  
To apply  Lemma \ref{top}, 
we use these properties of the reductions. 
The depth zero case, in any dimensional 
case, is studied in \cite{Yo}. 


Let $W_K$ denote the Weil group of $K$, and 
let $D$ denote the quaternion 
division algebra over $K$. 
In the following, we assume that $K$ has 
odd characteristic. 
In Theorem \ref{Mainc}.2, on the basis of
the theory of types for $\mathrm{GL}(2)$ and 
results in \cite{Mi0} and 
\cite{St},  
we give explicit and geometric 
one-to-one correspondences between 
the following three sets: 
\begin{itemize}
\item  $\mathcal{G}^0(K)$: the set of isomorphism 
classes of two-dimensional 
irreducible smooth representations of 
$W_K$, 
\item 
$\mathcal{A}_1^0(D)$: the set of isomorphism classes of irreducible smooth representations of $D^{\times}$ of degree $>1$, and 
\item $\mathcal{A}^0(K)$: the set of isomorphism classes of 
irreducible cuspidal representations of 
$\mathrm{GL}_2(K)$, 
\end{itemize}
 in the first cohomology of the Lubin-Tate curve. 
This is our main result in 
this paper. 
Theorem \ref{Mainc} is reduced to the assertion for special cases (Proposition \ref{Main}). 
To do so, we need to understand 
group action on the set of geometrically connected 
components of Lubin-Tate tower.  
In \S \ref{gc}, this is done by using results 
on the Lubin-Tate side in \cite[V.5]{FGL}.
Roughly speaking, the result, which we use, 
asserts 
that the action of $G=\mathrm{GL}_2(K) \times D^{\times} \times 
W_K$ on the set of geometrically connected components 
of the Lubin-Tate tower 
realizes the Lubin-Tate theory (cf.\ the proof of Corollary \ref{stst}). 
Note that the action of an 
open compact subgroup of $G$ 
on $\pi_0$ 
of the Lubin-Tate space with Drinfeld level 
structures is studied in 
\cite{St2} in a purely local manner.  
Since we need an information on the action of a larger subgroup of $G$ on the set, we need to 
use determinant morphisms explicitly 
constructed in \cite{FGL}.  
It is known that the cohomology 
of Lubin-Tate curve realizes 
the LJLC with multiplicity two 
by \cite{Mi0} and \cite{St} (cf.\ the proof
of Theorem \ref{MC}). Their proofs do not 
depend on any global method. 
To prove Theorem \ref{Mainc}.2, 
their results
play an important role. 
On the basis of 
the analysis in 
\S \ref{2}, \S \ref{4} and \S \ref{33}, 
Proposition \ref{Main} will be 
proved in a purely local manner in \S \ref{66}. 
In \S \ref{66}, 
we describe representations 
appearing in the middle cohomology of 
the reductions. 
To describe them in the unramified case, 
we use linking orders studied in 
\cite{WeJL} and \cite{WeSemi}. 

By the work of Bushnell-Henniart in 
\cite{BH}, the three sets $\mathcal{G}^0(K)$
$\mathcal{A}^0_1(D)$, 
and $\mathcal{A}^0(K)$
 are simply parametrized by 
admissible pairs $(L/K,\chi)$, 
where $L/K$ is a quadratic separable 
extension and an appropriate character $\chi$
of $L^{\times}$ 
(cf.\ \eqref{bins}). 
For example, starting 
from an admissible pair $(L/K,\chi)$, 
we can construct a smooth representation 
$\pi_{\chi}$
of an open compact-mod-center subgroup 
$J_{\chi}$
of $\mathrm{GL}_2(K)$ 
through representation theory
of a finite Heisenberg group.
Then, the compact induction of $\pi_{\chi}$
from $J_{\chi}$ to $\mathrm{GL}_2(K)$ 
is an irreducible 
cuspidal representation of 
$\mathrm{GL}_2(K)$ (cf.\ \S \ref{exBH}).
As above, 
the theory of types gives a recipe to construct 
 irreducible cuspidal representations. 
An explicit description of the 
LLC and the LJLC via the theory of types is  
given in \cite{BH}. 
We call the description 
the explicit LLC and the LJLC.  
See Theorem \ref{exp} for precise statements of them. 
This theory is established in a 
purely local and representation-theoretic method
without geometry. 
See \cite{Hen} for 
more developments in this direction.

In \S \ref{mimp}, we 
introduce a direct consequence 
of Theorem \ref{Mainc}. 
In 
Theorem \ref{NALexp}, we show that under Theorem \ref{Mainc}, 
the explicit LLC and the LJLC is equivalent to 
the NALT for $\mathrm{GL}(2)$.  
In this sense, a new proof
of the NALT for $\mathrm{GL}(2)$ is obtained
in a local approach.

We emphasize that,  
in the proof of our main theorem, 
it is unnecessary 
to understand a whole shape of 
the stable reduction of
the Lubin-Tate curve with each finite level structure. 
To prove it, it is enough to understand the first cohomology 
of the affinoids in this paper.  
To justify this, 
the \'etale cohomology theory 
in \cite{Huet} is needed. 
To apply this theory, 
we need to work at finite levels.  
It has been an anxious problem for us to 
relate the cohomology of 
the reductions of affinoids to 
 the one of Lubin-Tate tower. 
As explained above, 
this is settled by understanding 
the shape of the reductions of affinoids
and by  just applying Huber's theory.  
It makes possible for us to  
obtain a geometric proof of 
NALT for $\mathrm{GL}(2)$ without depending 
on global automorphic representations.  
Our approach will be applied to 
higher dimensional case in a subsequent paper. 
We note that the analysis given in 
\S \ref{2} and \ref{4} 
is elementary and explicit. 

In the case where the residual characteristic
equals two, 
it seems unknown to define a family of 
affinoids in a systematic or conceptual way. 
An example of a semi-stable model of some
Lubin-Tate curve 
in the residual characteristic two case
is found in 
\cite{IT}.  
In \cite{IT2}, 
the LLC for primitive representations 
of conductor 
three over dyadic fields is proved. 

\begin{acknowledgment*}
The author would like to thank 
Y.~Mieda for helpful discussions 
on \'{e}tale cohomology of rigid analytic 
varieties,  
for kindly pointing out 
that the NALT could imply the 
explicit LLC and LJLC in some 
setting, and for drawing his attention 
to the paper \cite{Mi}.  
He would like to thank 
T.~Ito and Y.~Mieda for 
drawing his attention to the formal models
in \cite{FGL}.  
The author would like to 
thank N.~Imai for many discussions on this 
topic through several joint works with him. 
He would like to 
thank N.~Otsubo and S.~Saito 
for their interests 
to this work and encouragements
in Niseko. 
The author would like to 
thank his former adviser T.~Saito 
for having asked him whether 
one can prove the NALT for $\mathrm{GL}(2)$
without full-understanding of the 
stable reduction of 
Lubin-Tate curve several years ago. 
He hopes that this paper answers the question.    
This work was supported by JSPS KAKENHI Grant Number 15K17506.
\end{acknowledgment*}
\begin{notation*}
For a non-archimedean valued field $K$, 
let $\mathcal{O}_K$ denote the valuation ring
of $K$, and 
let $\mathfrak{p}_K$ denote the maximal ideal of 
$\mathcal{O}_K$. 
We set $\mathbb{F}_K=\mathcal{O}_K/\mathfrak{p}_K$. 
For a non-archimedean local field $K$,
let $U_K^0=\mathcal{O}_K^{\times}$ 
and, for a positive integer $n \geq 1$, 
let $U_K^n=1+\mathfrak{p}_K^n$. 
Let $v_K(\cdot)$ denote the normalized 
valuation of $K$. 
For a prime number $p$ 
and a positive integer $r \geq 1$,
let $\mathbb{F}_{p^r}$ denote the finite field of cardinality $p^r$. 
\end{notation*}
\section{Reductions of affinoids}\label{2}
In this section, 
we define affinoids of two types  
in the Lubin-Tate curve
and determine 
the reductions of them over some finite 
extension of a base field. 
One is contained in tubular neighborhoods of points 
which have complex multiplication 
by the ring of integer in the 
unramified quadratic extension of the base field.
The other is contained in 
tubular neighborhoods of points 
which have complex multiplication 
by the rings of integers 
in ramified quadratic extensions.
Except for the depth zero case, 
their reductions are isomorphic to 
some Artin-Schreier curves as in Proposition 
\ref{rx} and Proposition \ref{rz}.
In the depth zero case,   
the reduction of the affinoid is isomorphic to  
the Deligne-Lusztig curve for 
a special linear group of degree two
over finite fields. 
\subsection{Preliminary on the canonical 
reduction}\label{canonical}
In this subsection, 
we recall 
several known facts on the canonical reduction of 
an affinoid.  
Let $K$ be 
a complete non-archimedean 
valued field of height one. 
Let $\X=\Sp A$ 
be a reduced affinoid variety over $K$. 
Let $|\cdot|_{\rm sup}$ be the supremum
norm on $A$. 
We set 
\begin{align*}
A^{\circ}&=\{x \in A \mid |x| _{\rm sup} \leq 1\}: 
\textrm{the set of all power-bounded elements}, \\
A^{\circ \circ}&=\{x \in A \mid |x|_{\rm sup} <1\}:
\textrm{the set of all topologically nilpotent elements} 
\end{align*}
(cf.\ \cite[Propositions 1 and 2 in \S 6.2.3]{BGR}). 
Then, $A^{\circ}$ is a subring of $A$, 
and $A^{\circ \circ }$ is an ideal of $A^{\circ}$. 
Then, 
we set 
\[
\overline{A}=A^{\circ}/A^{\circ \circ}. 
\]
This is called the 
canonical reduction of $A$ 
(cf.\ \cite[\S6.3]{BGR} and \cite[\S1]{BL}). 
We write $\overline{\X}$
for 
$\Spec \overline{A}$, which we simply 
call
the reduction of $\X$.  This is reduced, because $|\cdot|_{\rm sup}$
is a power-multiplicative norm. 
Let 
\begin{align*}
T_{n,K} =K \langle X_1, \ldots, X_n\rangle
\end{align*}
be the free Tate algebra in $n$ indeterminates 
over $K$ (cf.\ \cite[\S 5.1.1]{BGR}). 
Then, we have 
\[
T_{n,K}^{\circ}=\mathcal{O}_K\langle X_1,\ldots, X_n\rangle, 
\]
where the right hand side denotes the 
$\mathfrak{p}_K$-adic completion of 
$\mathcal{O}_K[X_1,\ldots,X_n]$. 
We take 
a surjective morphism of $K$-affinoid algebras
$\alpha \colon T_{n,K} \twoheadrightarrow
A$.
Let $|\cdot |_{\alpha}$ be the residue norm 
on $A$ associated to $\alpha$. 
We write $A_{\alpha}$ for the image of 
$T_{n,K}^{\circ}$
by $\alpha$. 
Then, 
we have 
\[
A_{\alpha}=\{x \in A \mid |x|_{\alpha} \leq 1\}. 
\]
This is a subring of $A^{\circ}$, because 
we have $|x|_{\rm sup} \leq |x|_{\alpha}$ for any $x \in 
A$. 
 
We keep the following 
lemma in mind 
whenever we compute the reductions of affinoids 
in the proceeding sections. 
\begin{lemma}\label{fund}
We assume 
that $A_{\alpha} \otimes_{\mathcal{O}_K}\mathbb{F}_K$ is reduced.  Then, we have 
$|\cdot|_{\rm sup}=|\cdot|_{\alpha}$ on $A$. 
Furthermore, we have 
\begin{align*}
A^{\circ} & =A_{\alpha} \supset A^{\circ \circ}=\mathfrak{p}_K A_{\alpha}, \\
\overline{A} &=A_{\alpha} \otimes_{\mathcal{O}_K}\mathbb{F}_K = A^{\circ} \otimes_{\mathcal{O}_K} \mathbb{F}_K. 
\end{align*}
\end{lemma}
\begin{proof}
By \cite[Proposition 1.1]{BLR}, 
we obtain $|\cdot|_{\rm sup}=|\cdot|_{\alpha}$ on $A$,
and hence $A^{\circ}=A_{\alpha}$. 
By 
\cite[Proposition 3 (i) in \S 6.4.3]{BGR}, 
we have $A^{\circ}=\alpha (T_{n,K}^{\circ})$
and 
$A^{\circ \circ}=\alpha(T_{n,K}^{\circ \circ})$.
Hence, we obtain $A^{\circ \circ}=
\alpha(\mathfrak{p}_K T_{n,K}^{\circ})=\mathfrak{p}_K A_{\alpha}$.  
Hence, the claims follow. 
\end{proof}
\begin{remark}
By the reduced fiber theorem in 
\cite[Theorem 1.3]{BLR}, for any 
geometrically reduced affinoid $K$-algebra $A$, 
there exist a finite separable extension $K'$ over 
$K$, and  
an epimorphism $\alpha \colon 
T_{n,K'} \to A \otimes_K K'$ such that  
$A_{\alpha} \otimes_{\mathcal{O}_{K'}} \mathbb{F}_{K'}
$ is reduced. 
\end{remark}
\subsection{Morphisms between formal schemes}\label{smooth}
In this subsection, we 
fix some terminology on formal geometry
(cf.\ \cite{Be} and \cite{Be2}). 

Let $K$ be a complete non-archimedean 
valued field. 
A morphism of affine formal schemes 
$\Spf A \to  \mathcal{S}=\Spf \mathcal{O}_K$ is topologically finitely generated 
if $A$ is $\mathcal{O}_K$-isomorphic to 
$T_{n,K}^{\circ}/J$, where $J$ is 
a finitely generated ideal of 
$T_{n,K}^{\circ}$. 
A morphism of formal schemes $\mathcal{X} \to 
 \mathcal{S}$ is 
locally finitely presented 
if it is locally 
isomorphic to $\Spf A$, where 
$\Spf A \to \mathcal{S}$ is topologically finitely presented. 

Let $\mathcal{O}_K$-$\mathcal{F}sch$ denote the 
category of formal schemes which are locally 
finitely presented over $\mathcal{S}$.
Assume that the valuation of $K$ is non-trivial.  
Let $a \in \mathfrak{p}_K \setminus \{0\}$. 
For an object $\mathcal{X} \in \mathcal{O}_K$-$\mathcal{F}sch$ and a positive 
integer $n$, let $\mathcal{X}_n$ denote the scheme 
$\left(\mathcal{X},\mathcal{O}_{\mathcal{X}}/a^n \mathcal{O}_{\mathcal{X}}\right)$, which 
is locally finitely presented 
over $\Spec (\mathcal{O}_K/a^n)$. 
\begin{definition}\label{des1}
 A morphism 
$\mathcal{Y} \to \mathcal{X}$ in $\mathcal{O}_K$-$\mathcal{F}sch$ is said to be 
\'etale if for any positive integer $m$, 
the induced morphism of schemes 
 $\mathcal{Y}_m \to \mathcal{X}_m$ is \'etale. 
\end{definition}
Clearly, this notion is independent of the choice of $a$. 

Let $\widehat{\mathbb{A}}^n_{\mathcal{S}}$ be 
an $n$-dimensional formal 
affine space 
$\Spf T_{n,K}^{\circ}$. 
\begin{definition}\label{des2}
A morphism 
$\mathcal{Y} \to \mathcal{X}$ in 
$\mathcal{O}_K$-$\mathcal{F}sch$ is said to be 
smooth if locally the map factors through an \'etale morphism 
$\mathcal{Y} \to 
\widehat{\mathbb{A}}^n_{\mathcal{S}} 
\widehat{\times}_{\mathcal{S}} \mathcal{X}$. 
\end{definition}

Assume that $K$ has a discrete valuation. 
\begin{definition}\label{des3}
Let $\mathcal{X} \in \mathcal{O}_K$-$\mathcal{F}sch$. 
We say that 
$\mathcal{X} \to \Spf \mathcal{O}_K$ is smoothly 
algebraizable if there exists  a scheme 
$X$ which is smooth, separated and of finite 
type over $\Spec \mathcal{O}_K$, and 
whose formal completion 
along the special fiber $X_s=X \otimes_{\mathcal{O}_K} \mathbb{F}_K$ 
is isomorphic to $\mathcal{X}$ over 
$\mathcal{S}$. 
\end{definition}
A smoothly algebraizable formal 
scheme 
$\mathcal{X} \to \Spf \mathcal{O}_K$
is smooth 
in the sense of Definition \ref{des2}.  

\subsection{Unramified components}\label{f00}
In the following, 
we fix a non-archimedean local field $F$. 
We simply write 
$\mathfrak{p}$ for $\mathfrak{p}_F$. 
We write $q$ for the cardinality of $\mathbb{F}_F$.  
Assume that the characteristic of 
$F$ equals $p$.
We fix a separable algebraic closure of $F$, for 
which we write $\overline{F}$. 
Let $F^{\mathrm{ur}}$ be the maximal unramified 
extension of $F$ in $\overline{F}$. 
We write 
$\widehat{F}^{\rm ur}$ for the completion of 
$F^{\rm ur}$. 
We write $\mathbb{F}$ for the residue field
of $\mathcal{O}_{\widehat{F}^{\rm ur}}$. 
Every formal module considered in this paper  
is assumed to be one-dimensional. 
Let $\mathscr{F}_0$ be the 
formal $\mathcal{O}_F$-module over 
$\mathbb{F}$ of height two, which exists 
uniquely up to isomorphism. 
We take a uniformizer $\varpi$ of $F$. 
We choose a coordinate of $\mathscr{F}_0$
such that 
\begin{equation}\label{f0}
[\varpi]_{\mathscr{F}_0}(X)=X^{q^2}, \quad 
X+_{\mathscr{F}_0} Y=X+Y, \quad 
[a]_{\mathscr{F}_0}(X)=a X\ \textrm{for $a \in \mathbb{F}_q$}. 
\end{equation}
Let $\mathscr{C}$ be the category 
of complete noetherian local
$\mathcal{O}_{\widehat{F}^{\rm ur}}$-algebras 
$A$ with residue field $\mathbb{F}$. 
Let $\mathfrak{p}_A$ denote 
the maximal ideal of $A$. 
For $n \geq 0$, 
let $\mathcal{R}(\mathfrak{p}^n)$ denote
the functor which associates to 
$A \in \mathscr{C}$ 
the set of isomorphism classes 
of triples 
$(\mathscr{F}_A,\rho,\phi)$, where 
$\mathscr{F}_A$ is a 
formal $\mathcal{O}_F$-module over $A$ with 
an isomorphism of $\mathcal{O}_F$-modules $\rho \colon 
\mathscr{F}_0 \xrightarrow{\sim} \mathscr{F}_A\otimes_A \mathbb{F}$ and the 
$\mathcal{O}_F$-homomorphism 
$\phi \colon 
\left(\mathcal{O}_F/\mathfrak{p}^n\right)^2
\to \mathscr{F}_A[\mathfrak{p}^n](A) \subset 
\mathfrak{p}_A$ 
is a Drinfeld level 
$\mathfrak{p}^n$-structure in the sense of 
\cite[Definition in p.\ 572]{Dr}.
This means that the polynomial 
$\prod_{x \in \left(\mathcal{O}_F/\mathfrak{p}^n\right)^2}(X-\phi(x))$ divides 
$[\varpi'^n]_{\mathscr{F}_A}(X)$
in $A[[X]]$ for any prime element $\varpi'$ 
of $F$. 
The functor 
$\mathcal{R}(\mathfrak{p}^n)$
is representable by a regular local ring 
$R(\mathfrak{p}^n)$ by \cite[Proposition 4.3]{Dr}.   
 
We can choose an isomorphism 
$\mathcal{O}_{\widehat{F}^{\rm ur}}[[u]] \simeq R(1)$
such that the universal formal 
$\mathcal{O}_F$-module 
 $\mathscr{F}^{\rm univ}$ over $\Spf R(1)$
has the form: 
\begin{align*}
[\varpi]_{\mathscr{F}^{\mathrm{univ}}}(X)& =X^{q^2}+uX^q+\varpi X, \\
X+_{\mathscr{F}^{\mathrm{univ}}}Y&=X+Y, \quad 
[a]_{\mathscr{F}^{\mathrm{univ}}}(X)=aX\ 
\textrm{for $a \in \mathbb{F}_q$}
\end{align*}
(cf.\ \cite[Proposition 5.1.1 (ii)]{St} and \cite[(2.2.1)]{WeG}).
We simply 
write $[a]_{\mathrm{u}}$ for $[a]_{\mathscr{F}^{\mathrm{univ}}}$ for $a \in \mathcal{O}_F$. 
We set $\mu_1(S,T)=S^qT-ST^q 
\in \mathbb{Z}[S,T]$. 

For a formal scheme $\mathcal{M}$, 
we write $\mathcal{M}^{\rm rig}$ for the rigid analytic 
variety associated to $\mathcal{M}$.
The rigid analytic variety $\mathcal{M}^{\rm rig}$
is called the generic fiber of $\mathcal{M}$. 
For an integer $n \geq 0$, 
we write $\X(\mathfrak{p}^n)$
for 
$\Spf R(\mathfrak{p}^n)^{\rm rig}$ (cf.\ \cite[\S1.2]{Ca}). 
For a positive integer 
$n \geq 1$, we have the natural explicit description of 
$R(\mathfrak{p}^n)$: 
\begin{equation}\label{sss}
R(\mathfrak{p}^n)=\mathcal{O}_{\widehat{F}^{\rm ur}}[[u,X_n,Y_n]]/
\mathcal{I}_n, 
\end{equation}
where $\mathcal{I}_n$ is generated by 
\begin{equation}\label{aa00}
[\varpi^n]_{\mathrm{u}}(X_n),
\quad  
[\varpi^n]_{\mathrm{u}}(Y_n), \quad 
\mu_1\left([\varpi^{n-1}]_{\mathrm{u}}(X_n),
[\varpi^{n-1}]_{\mathrm{u}}(Y_n)\right)^{q-1}-\varpi.
\end{equation}
The parameters $X_n$ and $Y_n$ 
are regarded as sections of 
$\mathscr{F}^{\rm univ}[\mathfrak{p}^n]$.
The final relation in \eqref{aa00} comes from 
the condition that 
$(X_n,Y_n)$ makes a Drinfeld basis of 
$\mathscr{F}^{\rm univ}[\mathfrak{p}^n]$. 
We simply write $X_1$ and $Y_1$ for 
$[\varpi^{n-1}]_{\mathrm{u}}(X_n)$ and 
$[\varpi^{n-1}]_{\mathrm{u}}(Y_n)$ respectively. 
Then, we have 
\[
\mu_1\left(X_1,
Y_1\right)^q-\varpi
\mu_1\left(X_1,Y_1\right) \in 
\left([\varpi]_{\mathrm{u}}(X_1),[\varpi]_{\mathrm{u}}(Y_1)\right)
=
\left([\varpi^n]_{\mathrm{u}}(X_n),
[\varpi^n]_{\mathrm{u}}(Y_n)\right)
\]
in $\mathcal{O}_{\widehat{F}^{\mathrm{ur}}}[[u,X_n,Y_n]]$. Hence, 
we have an isomorphism 
\[
R(\mathfrak{p}^n) \simeq 
\mathcal{O}_{\widehat{F}^{\rm ur}}
\left[\left[u,X_n,Y_n\right]\right][\mu_1(X_1,Y_1)^{-1}]/
\left([\varpi^n]_{\mathrm{u}}(X_n),
[\varpi^n]_{\mathrm{u}}(Y_n)\right). 
\]
The formal model \eqref{sss} is used 
in \cite{WeG}. In \cite[\S3.3]{WeG}, 
the formal model is described 
via an explicit description
 of determinant
of higher level structure.

Note that \eqref{sss} is the base change of
$\Spf \mathcal{O}_F[[u,X_n,Y_n]]/\mathcal{I}_n$
to $\Spf \mathcal{O}_{\widehat{F}^{\rm ur}}$. 
For $1 \leq i \leq n$, 
we set 
\begin{equation} \label{dfn}
[\varpi]_{\mathrm{u}}(X_i)=X_{i-1}, 
\quad 
[\varpi]_{\mathrm{u}}(Y_i)=Y_{i-1}, 
\end{equation}
where we set $X_0=Y_0=0$. 

Let $F_2$ denote the quadratic unramified extension of $F$. 
Let $\mathscr{F}$ be the formal 
$\mathcal{O}_{F_2}$-module over $\mathcal{O}_{F_2}$
such that 
\begin{equation}\label{fg2}
[\varpi]_{\mathscr{F}}(X)=X^{q^2}+\varpi X, 
\quad X+_{\mathscr{F}} Y =X+Y, \quad  
[a]_{\mathscr{F}}(X)=a X\ \textrm{for $a
\in \mathbb{F}_{q^2}$} 
\end{equation}
(cf.\ \cite{G} or \cite[\S13]{GH}). 
Note that 
the isomorphism class of the pair 
$\left(\mathscr{F} \widehat{\otimes}_{\mathcal{O}_{F_2}} \mathcal{O}_{\widehat{F}^{\rm ur}},\rho\right)$ 
with the natural isomorphism 
$\rho \colon \mathscr{F}_0 
\xrightarrow{\sim} \mathscr{F} \otimes_{\mathcal{O}_{F_2}} \mathbb{F}$ corresponds 
to 
the $\mathcal{O}_{\widehat{F}^{\rm ur}}$-valued point 
$\Spf \mathcal{O}_{\widehat{F}^{\rm ur}} \hookrightarrow 
\Spf \mathcal{O}_{\widehat{F}^{\rm ur}}[[u]]$ defined 
by $u \mapsto 0$. 
We consider the set 
of primitive $\varpi^n$-torsion points of $\mathscr{F}$:  
\begin{equation}\label{lt}
\mathscr{F}[\mathfrak{p}_{F_2}^n]_{\mathrm{prim}}=\left\{\varpi_n \in \overline{F} \mid [\varpi^n]_{\mathscr{F}}(\varpi_n)=0, \ 
[\varpi^{n-1}]_{\mathscr{F}}(\varpi_n) \neq 0\right\}. 
\end{equation}
Let $\varpi_n \in \mathscr{F}[\mathfrak{p}_{F_2}^n]_{\mathrm{prim}}$. 
We set $\varpi_i=[\varpi^{n-i}]_{\mathscr{F}}(\varpi_n)$
for $1 \leq i \leq n-1$. 
Let $\zeta \in \mathbb{F}_{q^2} 
\setminus \mathbb{F}_q$. 
For $1 \leq i \leq n$, 
we set 
\begin{equation}\label{ffo}
X_i= \varpi_i +S_{\varpi_i}, \quad 
Y_i=\zeta \varpi_i+T_{\varpi_i}. 
\end{equation}
The parameters 
$S_{\varpi_i}$ and $T_{\varpi_i}$
depend on $\varpi_i$. 
However, we simply write 
$S_i$ and $T_i$ for $S_{\varpi_i}$ and $T_{\varpi_i}$ respectively.  
By \eqref{dfn} and \eqref{lt}, for $1 \leq i \leq n$, 
we obtain 
\begin{gather}\label{z1}
\begin{aligned}
& S_i^{q^2}+\varpi S_i+u(\varpi_i^q +S_i^q)=S_{i-1}, \\
& T_i^{q^2}+\varpi T_i+u((\zeta \varpi_i)^q+ T_i^q)=T_{i-1}, 
\end{aligned}
\end{gather}
where $S_0=T_0=0$. 
We put $\zeta_1=\zeta^q-\zeta$. 
We have $\zeta_1^q+\zeta_1=0$. 
By $\zeta \notin \mathbb{F}_q$, 
we have $\zeta_1 \neq 0$ and 
$\zeta_1^{q-1}=-1$. 
We set  
\begin{equation}\label{duo}
U_{\varpi_i}=\zeta^q S_i-T_i \quad 
\textrm{for any 
$1 \leq i \leq n$}. 
\end{equation}
We simply write $U_i$ for $U_{\varpi_i}$. 
By using \eqref{z1}, for $1 \leq i \leq n$, 
 we acquire 
 \begin{align}\label{ss2}
&  S_i^{q^2}+\varpi S_i+u(\varpi_i^q+ S_i^q)=S_{i-1}, 
  \\ \label{u}
& U_i^{q^2}+\varpi U_i+
 u \left(\zeta_1 S_i^q+U_i^q\right)=U_{i-1},  
 \end{align}
where $U_0=0$. 

Let $k$ be a positive integer. 
We set $m=\left[\frac{k+1}{2}\right]$ and 
 $h=(q^2-1)^{-1}$. 
 Let $\C$ be the completion of 
 $\overline{F}$. 
 We simply write $v$ for the normalized valuation 
 $v_F$ on $F$.
 We write also 
 $v$ for the unique extension of $v$ to $\C$.  
Let $\X_{n,k,\zeta, \varpi_n} 
\subset \X(\mathfrak{p}^n)$ 
be the affinoid defined by 
\begin{equation}\label{dad}
\begin{cases}
v(u) \geq m, \  v(S_k) \geq h/q^{k-1}, \ 
\ v(U_k) \geq h & \textrm{if $k$ is even}, \\
v(u) \geq m-\frac{1}{q+1}, \  
v(S_k) \geq h/q^{k-1}, \  
v(U_k) \geq h & \textrm{if $k$ is odd}.
\end{cases}
\end{equation} 
\begin{lemma}
On $\X_{n,k,\zeta, \varpi_n}$, we have 
\begin{gather}\label{dr}
\begin{aligned}
& v(u) \geq m, \quad 
 v(S_{m+i}) \geq h/q^{2i-1} \quad 
\textrm{for $1 \leq i \leq m$},  
\quad v(S_{m-i}) \geq i+hq 
\quad \textrm{for $0 \leq i \leq m-1$}, \\
&  \quad v(U_{k-i}) \geq i+h 
\quad \textrm{for $0 \leq i \leq m$}, \quad 
v(U_i)=v(uS_i^q/\varpi) \quad \textrm{for $1 \leq i \leq m$} 
\end{aligned}
\end{gather}
 if $k$ is even, and 
\begin{gather}\label{dry}
\begin{aligned}
& v(u) \geq m-\frac{1}{q+1}, 
\quad 
v(S_{m-i}),\ v(U_{k-i}) \geq i+h 
\quad \textrm{for $0 \leq i \leq m-1$}, \\
& v(U_i)=v(uS_i^q/\varpi), \quad
v(S_{m+i}) \geq h/q^{2i} \quad \textrm{for 
$1 \leq i \leq m-1$} 
\end{aligned}
\end{gather}
if $k$ is odd. 
\end{lemma}
\begin{proof}
These assertions immediately follow from 
\eqref{ss2} and 
\eqref{u}. 
\end{proof}
In the following, 
we focus on the case where $n=k$. 
\begin{lemma}\label{rot}
Let $k \geq 1$ be a positive integer. \\
{\rm 1}.\ Assume that $k=1$. 
The affinoid $\X_{1,1,\zeta,\varpi_1}$
is independent of $\zeta$ and $\varpi_1$.  \\[0.2cm]
Assume that $k=2$. 
Let $\varpi_k, \varpi'_k \in 
\mathscr{F}[\mathfrak{p}_{F_2}^k]_{\mathrm{prim}}$. \\
{\rm 2}.\ 
If $v(\varpi_k-\varpi'_k)<h$, 
we have $\X_{k,k,\zeta, \varpi_k}
\cap \X_{k,k,\zeta, \varpi'_k}=\emptyset$
in $\X(\mathfrak{p}^k)$.  \\
{\rm 3}.\ If $v(\varpi_k-\varpi'_k) \geq h$, 
we have $\X_{k,k,\zeta, \varpi_k}=
\X_{k,k,\zeta, \varpi'_k}$. 
\end{lemma}
\begin{proof}
We prove the first assertion. 
The affinoid $\X_{1,1,\zeta,\varpi_1}$ is defined 
only by $v(u) \geq \frac{q}{q+1}$. 
Hence, the required assertion follows. 

We set $\varpi'_i=
[\varpi^{k-i}]_{\mathscr{F}}(\varpi'_k)$ for $1 \leq i \leq k-1$. 
By \eqref{duo}, we have 
$U_{\varpi_i}=-\zeta_1 \varpi_i+\zeta^qX_i-Y_i$
for $1 \leq i \leq k$. 
Hence, in particular,  
we have 
\begin{equation}\label{fff}
U_{\varpi_k}-U_{\varpi'_k}=-\zeta_1
(\varpi_k-\varpi'_k). 
\end{equation}
On $\X_{k,k,\zeta, \varpi_k}
\cap \X_{k,k,\zeta, \varpi'_k}$, we have 
$v(U_{\varpi_k}-U_{\varpi'_k}) \geq h$. 
Hence, 
the second assertion follows from 
\eqref{fff}. 

Assume that $v(\varpi_k-\varpi'_k) \geq h$. 
This implies that $\varpi_i=\varpi'_i$ for $1 \leq i \leq k-1$.  
Hence, the third assertion follows 
from the assumption, \eqref{ffo}, \eqref{dad} and \eqref{fff}.  
\end{proof}
\begin{remark}\label{xc}
Let $\varpi_k, \varpi'_k \in \mathscr{F}[\mathfrak{p}_{F_2}^k]_{\mathrm{prim}}$. 
There exists 
a unique element $a \in U_{F_2}^0/U_{F_2}^k$ such that $\varpi'_k=[a]_{\mathscr{F}}(\varpi_k)$. 
Note that 
\[
v(\varpi_k-[a]_{\mathscr{F}}(\varpi_k)) \geq 
h \iff a \in U_{F_2}^{k-1}/U_{F_2}^{k}. 
\]
\end{remark}
For $k \geq 1$, we put $K_k=\widehat{F}^{\mathrm{ur}}(\varpi_k)$. 
By Lemma \ref{rot}.1, we simply write 
$\X_{1,1}$ for $\X_{1,1,\zeta,\varpi_1}$. 
For a field extension 
$L/\widehat{F}^{\rm ur}$ in $\C$
and a rigid analytic variety $\X$ over 
$\widehat{F}^{\rm ur}$, 
we write $\X_L$ for the base change of 
$\X$ to $L$. 
\begin{lemma}\label{xy}  
We set 
\[
A_1=\Gamma\left(\X_{1,1,K_1},\mathcal{O}_{\X_{1,1,K_1}}\right).
\] 
Then, the formal scheme 
$\Spf 
A_1^{\circ} \to \Spf \mathcal{O}_{K_1}$ is smoothly 
algebraizable in the sense of Definition \ref{des3}.   
Furthermore, we have isomorphisms  
\[
\overline{\X}_{1,1,K_1} \simeq \Spec 
\left(A_1^{\circ} \otimes_{\mathcal{O}_{K_1}} \mathbb{F}\right)
\simeq 
\Spec \mathbb{F}[X,Y]
\Big/ \left((X^qY-XY^q)^{q-1}+1\right). 
\] 
\end{lemma}
\begin{proof}
In the sequel, we discuss on 
$\X_{1,1,K_1}$. 
By setting $u=\varpi_1^{q(q-1)} u_0$, $X_1=\varpi_1 X$ and 
$Y_1=\varpi_1 Y$ with $v(u_0), v(X), v(Y) \geq 0$, 
we have 
\[
(X^qY-XY^q)^{q-1}=-1
\]
by $\mu_1(X_1,Y_1)^{q-1}=\varpi$. 
Hence, we obtain $v(X)=v(Y)=0$. 
Furthermore, we have 
\[
u_0=X^{-(q-1)}-X^{q(q-1)}=
Y^{-(q-1)}-Y^{q(q-1)}
\]
by $[\varpi]_{\mathrm{u}}(X_1)=
[\varpi]_{\mathrm{u}}(Y_1)=0$.  
We set $f=(X^qY-XY^q)^{q-1}+1$. 
Let 
\begin{align*}
A'_1& =\mathcal{O}_{K_1} \langle u_0,X^{\pm 1},
Y^{\pm 1}\rangle\big/
\left(f,u_0-X^{-(q-1)}+X^{q(q-1)},
u_0-Y^{-(q-1)}+Y^{q(q-1)}\right) \\
& \xleftarrow{\sim} \mathcal{O}_{K_1} \langle X^{\pm 1},Y^{\pm 1} \rangle\big/\left(f\right). 
\end{align*}
Then, $A'_1 \otimes_{\mathcal{O}_{K_1}} \mathbb{F}$ is reduced and 
$A_1=A'_1 \otimes_{\mathcal{O}_{K_1}} K_1$. 
Hence, by Lemma \ref{fund}, we obtain 
\begin{gather}\label{A}
\begin{aligned}
A_1^{\circ}& =A'_1, \\  
\overline{A_1} & \simeq A_1^{\circ} \otimes_{\mathcal{O}_{K_1}} \mathbb{F} \simeq 
\mathbb{F}[X^{\pm 1},Y^{\pm 1}]/(f). 
\end{aligned}
\end{gather} 
We set 
\[
X=\Spec \mathcal{O}_{K_1}[X^{\pm 1},Y^{\pm 1}]\big/
\left(f\right).  
\]
Then, $X$ is smooth over 
$\Spec \mathcal{O}_{K_1}$, and 
$\Spf A_1^{\circ}$ is isomorphic to 
the formal completion of $X$ along the special fiber 
$X_s=X\otimes_{\mathcal{O}_{K_1}} \mathbb{F}_{K_1}$. Hence, 
the required assertions follow from \eqref{A}. 
\end{proof}
\begin{proposition}\label{rx} 
Assume that $k \geq 2$.  
We set 
\[
A_k=\Gamma\left(\X_{k,k,\zeta,\varpi_k,K_k}, 
\mathcal{O}_{\X_{k,k,\zeta,\varpi_k,K_k}}\right).
\]
Then, the formal scheme 
$\Spf A_k^{\circ} \to \Spf \mathcal{O}_{K_k}$ is 
smoothly algebraizable. 
Furthermore, we have isomorphisms  
\begin{equation}\label{fol}
\overline{\X}_{k,k,\zeta,\varpi_k,K_k} \simeq 
\Spec \left(A_k^{\circ} \otimes_{\mathcal{O}_{K_k}} \mathbb{F}\right) \simeq 
\Spec \mathbb{F}[X,Y]\Big/\left(X^{q^2}-X-\left(Y^{q(q+1)}-Y^{q+1}\right)^{q^{k-1}}\right). 
\end{equation}
\end{proposition}
\begin{proof}
For any $\alpha \in \mathbb{Q}_{\geq 0}$, 
we write 
$f \equiv g \mod \alpha+$ if $v(f-g)>\alpha$. 
In the following, 
we always consider on $\X_{k,k,\varpi_k,K_k}$.    
 Assume that $k$ is even. 
 We write $k=2m$.
By \eqref{ss2}, \eqref{u} and \eqref{dr}, if $m=1$, 
we have 
\begin{gather}\label{vpp}
\begin{aligned}
& U_2^{q^2}+\varpi U_2 +u \zeta_1 S_2^q \equiv U_1 \mod (1+h)+, \\
& \varpi U_1+u \zeta_1 S_1^q \equiv 0 \mod 
(2+h)+, \\
& S_2^{q^2} \equiv S_1 \mod h q+, \\
& \varpi S_1+\varpi_1^q u  \equiv 0 
\mod (1+h q)+,   
\end{aligned}
\end{gather}
and, if $m>1$, 
\begin{gather}\label{upp}
\begin{aligned}
& U_k^{q^2}+\varpi U_k \equiv 
U_{k-1} \mod (1+h)+, \\
& \varpi U_{k-i} \equiv U_{k-i-1} \mod 
\left(i+1+h\right)+ \quad \textrm{for $1 \leq i \leq m-2$}, \\
& \varpi U_{m+1}+
u \zeta_1 S_{m+1}^q \equiv U_m \mod 
\left(m+h\right)+, \\
& \varpi U_i+u \zeta_1 S_i^q \equiv
0 \mod v(\varpi U_i)+ 
\quad \textrm{for $1 \leq i \leq m$}, \\
&  S_{m+i}^{q^{2i}} \equiv S_m \mod hq+ \quad \textrm{for $1 \leq i \leq m$}, \\
& \varpi^m S_m+\varpi_1^q u \equiv 0 
\mod \left(m+hq\right)+. 
\end{aligned}
\end{gather}
By \eqref{vpp} and \eqref{upp}, 
we obtain 
\begin{gather}\label{e1}
\begin{aligned}
U_k^{q^2}+\varpi U_k & \equiv 
\begin{cases}
U_1-u \zeta_1 S_2^q & \textrm{if $m=1$}, \\[0.1cm] 
U_{k-1} \equiv U_{m+1}/\varpi^{m-2} & \textrm{if $m>1$}
\end{cases} \\
& \equiv 
\frac{-u \zeta_1 S_{m+1}^q+U_{m}}{\varpi^{m-1}} \equiv 
-\frac{u \zeta_1}{\varpi^m}
\left(\varpi S_{m+1}^q+S_{m}^q\right) \\
&  \equiv \frac{\zeta_1}{\varpi_1^q} \left(\varpi
S_{m+1}^{q(q+1)}+S_{m}^{q+1} \right)
\equiv \frac{\zeta_1}{\varpi_1^q}\left(
\varpi S_k^{q^{k-1}(q+1)}+S_k^{q^{k}(q+1)}\right)\mod (1+h)+. 
\end{aligned}
\end{gather}
We set 
\begin{gather}\label{rd}
\begin{aligned}
u &=\varpi^m u_0 \quad 
\textrm{with $v(u_0) \geq 0$}, \\
U_k &=\zeta_1 \varpi_1 X, 
\quad S_k=\varpi_{m+1}^q Y \quad 
\textrm{with $v(X), v(Y) \geq 0$}. 
\end{aligned}
\end{gather}
By \eqref{vpp} and \eqref{upp}, we have 
\[
u_0 \equiv -\frac{S_m}{\varpi_1^q} \equiv 
-\frac{S_k^{q^k}}{\varpi_1^q} \equiv 
-Y^{q^k} \mod 0+,  
\]
where we use $\varpi_{m+1}^{q^k} \equiv 
\varpi_1 \mod h+$ at the third congruence. 
Hence, by \eqref{rd}, we obtain 
\begin{equation}\label{qqq}
f_{1,k}(u_0,X,Y)
=u_0+Y^{q^k}-F_k(u_0,X,Y)=0
\end{equation}
with some polynomial $F_k(u_0,X,Y) \in 
\mathcal{O}_{K_k} [ u_0,X,Y ]$
such that $v(F_k(u_0,X,Y))>0$. 
By substituting \eqref{rd} 
to \eqref{e1} and  
dividing it by
$\varpi_1^{q^2}$, 
we obtain 
\begin{equation}\label{qqq1}
f_{2,k}(u_0,X,Y)=X^{q^2}-X -\left(Y^{q(q+1)}-Y^{q+1}\right)^{q^{k-1}}-G_k(u_0,X,Y)=0
\end{equation}
with some polynomial  
$G_k(u_0,X,Y) \in \mathcal{O}_{K_k} [ u_0,X,Y ]$ such that $v(G_k(u_0,X,Y))>0$.  
We consider the subring in $A_k$:
\begin{equation}\label{qqq2}
A'_k=\mathcal{O}_{K_k} \langle u_0,X,Y \rangle \big/
\left(f_{1,k}(u_0,X,Y), f_{2,k}(u_0,X,Y)\right).  
\end{equation}
Then, $A'_k \otimes_{\mathcal{O}_{K_k}} \mathbb{F}$ is reduced, and 
$A_k=A'_k \otimes_{\mathcal{O}_{K_k}} K_k$. 
Hence, by Lemma \ref{fund}, 
we obtain 
\begin{gather}
\begin{aligned}\label{qqq4}
A_k^{\circ}& =A'_k, \\ 
\overline{A_k}  
& \simeq A_k^{\circ} 
\otimes_{\mathcal{O}_{K_k}} 
\mathbb{F} \simeq 
\mathbb{F}[X,Y] \Big/\left(X^{q^2}-X-\left(Y^{q(q+1)}-Y^{q+1}\right)^{q^{k-1}}\right). 
\end{aligned}
\end{gather}
Assume that $k$ is odd. 
We write $k=2m-1$. 
Then, by \eqref{ss2}, \eqref{u} and \eqref{dry}, 
we have 
\begin{gather}\label{off}
\begin{aligned}
& U_k^{q^2}+\varpi U_k \equiv U_{k-1} \mod (1+h)+, \\
& \varpi^i U_{k-1} \equiv U_{k-i-1} 
\mod \left(i+1+h\right)+ \quad 
\textrm{for $0 \leq i \leq m-2$}, \\
& \varpi U_i+u \zeta_1 S_i^q \equiv 0 
\mod v(\varpi U_i)+\quad 
\textrm{for $1 \leq i \leq m$}, \\
& S_{m+i}^{q^{2i}} \equiv S_m \mod h+ \quad \textrm{for $1 \leq i \leq m-1$}, \\ 
& S_m^{q^2}+\varpi S_m \equiv S_{m-1}
\mod (1+h)+, \\
& \varpi^i S_{m-1} \equiv S_{m-i-1} \mod 
\left(i+1+h\right)+ \quad 
\textrm{for $0 \leq i \leq m-2$}, \\
& \varpi S_1+\varpi_1^q u \equiv 0 \mod 
\left(m+h\right)+. 
\end{aligned}
\end{gather}
Hence, 
we obtain 
\begin{gather}\label{ed}
\begin{aligned}
U_k^{q^2}+\varpi U_k & \equiv 
\frac{U_m}{\varpi^{m-2}} \equiv 
-\frac{u \zeta_1 S_m^q }{\varpi^{m-1}}
\equiv \frac{\zeta_1 S_{m-1}S_m^q}{\varpi_1^q} \\
&\equiv \frac{\zeta_1}{\varpi_1^q}
\left(S_m^{q^2}+\varpi S_m\right) S_m^q
\equiv 
\frac{\zeta_1}{\varpi_1^q}
\left(S_k^{q^k(q+1)}+\varpi S_k^{q^{k-1}(q+1)}\right)
 \mod (1+h)+. 
\end{aligned}
\end{gather}

We set 
\begin{gather}\label{fs}
\begin{aligned}
u &= \varpi^{m-1} \varpi_1^{q(q-1)} u_0 \quad 
\textrm{with $v(u_0) \geq 0$}, \\
U_k&=\zeta_1 \varpi_1 X, \quad 
S_k=\varpi_m Y \quad 
\textrm{with $v(X), v(Y) \geq 0$}. 
\end{aligned}
\end{gather}
By \eqref{off} and \eqref{ed}, 
we have 
\begin{align*}
f_{1,k}(u_0,X,Y)& =u_0+\bigl(Y^{q^2}-Y\bigr)^{q^{k-1}}-F_k(u_0,X,Y)=0, \\
f_{2,k}(u_0,X,Y)& =X^{q^2}-X-\left(Y^{q(q+1)}-Y^{q+1}\right)^{q^{k-1}}-G_k(u_0,X,Y)=0
\end{align*}
with some elements $F_k(u_0,X,Y), G_k(u_0,X,Y) \in 
\mathcal{O}_{K_k} [ u_0,X,Y ]$ such that 
\[
v(F_k(u_0,X,Y)),\ v(G_k(u_0,X,Y))>0.
\] 
Let 
\begin{equation}\label{qqq3}
A'_k=\mathcal{O}_{K_k} \langle u_0,X,Y\rangle \big/
\left(f_{1,k}(u_0,X,Y), 
f_{2,k}(u_0,X,Y)\right) \subset A_k. 
\end{equation}
Since $A'_k \otimes_{\mathcal{O}_{K_k}} \mathbb{F}$  is reduced, and 
$A'_k \otimes_{\mathcal{O}_{K_k}} K_k=A_k$, 
we obtain 
\begin{gather}\label{qqq5}
\begin{aligned}
A'_k & =A_k^{\circ}, \\ 
\overline{A_k} & =A_k^{\circ} \otimes_{\mathcal{O}_{K_k}} \mathbb{F} \simeq 
\mathbb{F}[X,Y] \Big/\left(X^{q^2}-X-\left(Y^{q(q+1)}-Y^{q+1}\right)^{q^{k-1}}\right)
\end{aligned}
\end{gather}
 by Lemma \ref{fund}. 
We set 
\[
\Delta_k=
\begin{vmatrix}
\frac{\partial f_{1,k}}{\partial u_0} & \frac{\partial f_{1,k}}{\partial X} \\[0.2cm]
\frac{\partial f_{2,k}}{\partial u_0} & \frac{\partial f_{2,k}}{\partial X} 
\end{vmatrix} \in 
\mathcal{O}_{K_k}[u_0,X,Y] 
\] and 
\[
V_k=\Spec 
\mathcal{O}_{K_k}\left[u_0,X,Y,\Delta_k^{-1}\right]
\big/\left(f_{1,k}, f_{2,k}\right). 
\]
Note that $\Delta_k \equiv -1 \mod \mathfrak{p}_{K_k}$. 
Then, $V_k \to \Spec \mathcal{O}_{K_k}[Y]$ is \'etale
and hence $V_k \to \Spec \mathcal{O}_{K_k}$
is smooth.
By \eqref{qqq2} and \eqref{qqq3}, 
the formal scheme $\Spf A_k^{\circ}$ is isomorphic 
to the formal completion of $V_k$ along 
$(V_k)_s$ 
over $\Spf \mathcal{O}_{K_k}$. Hence, 
the required assertions follow from
\eqref{qqq4} and \eqref{qqq5}. 
\end{proof}
\subsection{Ramified components}\label{rcc}
In this subsection, we assume that $p \neq 2$.
Whenever we treat the ramified case, 
we always assume this. 
We simply write $\underline{2}$
for $\mathbb{Z}/2 \mathbb{Z}$. 
We consider the following 
formal scheme in \cite[II.2.1]{FGL}: 
\begin{equation}\label{in}
\Spf \mathcal{O}_{\widehat{F}^{\rm ur}}[[x_1,x_2]][(s_{i,j})_{i \in 
\underline{2},\ 0 \leq j \leq n}]/
\mathcal{J}_{n}, 
\end{equation}
where $\mathcal{J}_n$ is generated by 
\begin{gather}\label{s}
\begin{align}
x_1x_2-\varpi, \quad 
s_{i,0}^{q-1}-x_i, \quad 
s_{i,j}^q-x_{i-j} s_{i,j}-s_{i,j-1} \quad 
\textrm{for $i \in \underline{2},\ 1 \leq j \leq n$}. 
\end{align}
\end{gather}
The model \eqref{in} is the base change of 
$\Spf \mathcal{O}_F[[x_1,x_2]][(s_{i,j})_{i \in 
\underline{2},\ 0 \leq j \leq n}]/
\mathcal{J}_n$ to 
$\Spf \mathcal{O}_{\widehat{F}^{\rm ur}}$. 
We write $\Y(\mathfrak{p}^n)$
for the generic fiber of this formal scheme
\eqref{in}. 
If $n=2m-1$ is odd, the formal scheme \eqref{in}
equals 
${\widebreve{\mathcal{M}}}
_{\mathcal{LT},\mathfrak{B}_{m}^{\times}}$
in the notation of \cite[Remarque II.2.3]{FGL}.  
Hence, the rigid analytic curve 
$\Y(\mathfrak{p}^n)$
is the quotient of $\X(\mathfrak{p}^{m+1})$
by $\mathfrak{B}_m^{\times}=U_{\mathfrak{I}}^{n+1}$ in the notation of 
\S \ref{GL}. 
Let $((x_i)_{i \in \underline{2}}, 
(s_{i,j})_{i \in \underline{2},\ 0 \leq j \leq n}) \in 
\Y(\mathfrak{p}^n)$. 
We set  
\begin{equation}\label{so-1}
t_{i,0}=\frac{s_{i,0}}{s_{i-1,0}^q}, 
\quad 
t_{i,j}=\frac{s_{i,j}}{s_{i-j,0}} 
\quad \textrm{for $i \in 
\underline{2}$ and $1 \leq j \leq n$}. 
\end{equation}
Then, for $i \in \underline{2}$, 
we have
\begin{gather}\label{so} 
\begin{aligned}
t_{1,0} t_{2,0} &=(s_{1,0}s_{2,0})^{-(q-1)}=\varpi^{-1}, \\ 
t_{i,j}^q-t_{i,j} &=
\begin{cases}
t_{i,0}  & \textrm{if $j=1$}, \\ 
t_{i,j-1}t_{i-j+1,0} & \textrm{if $2 \leq j \leq n$}. 
\end{cases}
\end{aligned}
\end{gather}
We take a second 
root of $\varpi$, for
which we write $\varpi_E \in \overline{F}$. 
We set $E=F(\varpi_E)$. 
Let $\mathscr{G}$ be the formal 
$\mathcal{O}_E$-module over $\mathcal{O}_E$ such that
\begin{equation}\label{gf2}
[\varpi_E]_{\mathscr{G}}(X)=X^q+\varpi_E X, \quad   
X+_{\mathscr{G}}Y =X+Y, \quad [\xi]_{\mathscr{G}}(X)=\xi X\ \textrm{for $\xi \in \mathbb{F}_q$} 
\end{equation}
(cf.\ \cite{G}). 
For any positive integer $i$, 
we set 
\begin{equation}\label{enu}
\mathscr{G}[\mathfrak{p}_E^i]_{\mathrm{prim}}=
\left\{\varpi_{E, i} \in \overline{F} \mid 
[\varpi_E^i]_{\mathscr{G}}(\varpi_{E,i})=0, \  
[\varpi_E^{i-1}]_{\mathscr{G}}(\varpi_{E,i}) \neq 0\right\}. 
\end{equation}
Let $\varpi_{E,n+1} \in \mathscr{G}[\mathfrak{p}_E^{n+1}]_{\mathrm{prim}}$. 
We put  
\[
\theta_i=
\begin{cases}
\varpi_{E,1} & \textrm{if $i=1$}, \\
[\varpi_E^{n+1-i}]_{\mathscr{G}}(\varpi_{E,n+1})/\varpi_{E,1} & \textrm{if $2 \leq i \leq n$}. 
\end{cases}
\]
Then, we have 
\begin{gather}\label{ltr}
\begin{aligned} 
\theta_1^{q-1}&=-\varpi_E, \quad
\theta_2^q-\theta_2=-\varpi_E^{-1}  \\
\theta_i^q-\theta_i & =-\theta_{i-1} \varpi_E^{-1} 
\quad \textrm{for $3 \leq i \leq n+1$}. 
\end{aligned}
\end{gather}
The extension $E_{n+1}=E(\varpi_{E,n+1})
=E(\theta_1,\theta_{n+1})$
 is a 
Lubin-Tate extension of 
$E$ of degree 
$q^n(q-1)$. 
Assume that $s_{1,0}s_{2,0}=\theta_1^2$. 
Then, for $i \in \underline{2}$, 
we have  
\begin{equation}\label{st}
t_{i,0}=\frac{s_{i,0}^{q+1}}{\theta_1^{2q}}.  
\end{equation}

Let $k$ be a positive integer. 
For $i \in \underline{2}$, we set 
\begin{gather}\label{rt}
\begin{aligned}
s_{i,0} &=\theta_1+\theta_1^q u_i^0, \\ 
t_{i,j} &=
\begin{cases}
-\varpi_E^{-1}+u_i & \textrm{if $j=0$}, \\
\theta_{j+1}+u_{i,j} & \textrm{if $1 \leq j \leq n$}.  
\end{cases}
\end{aligned}
\end{gather}
By \eqref{so}, \eqref{ltr}, 
\eqref{st} and \eqref{rt}, for 
$i \in \underline{2}$,  
we have  
\begin{align}
\label{so0-1}
u_i& =u_i^0+\varpi_E^{q-1} 
(u_i^0)^q(1-\varpi_E u_i^0), \\
\label{so0}
u_1+u_2 &=\varpi_E u_1 u_2, \\ 
\label{so1}
u_{i,j}^q-u_{i,j} 
&=
\begin{cases}
u_i & \textrm{if $j=1$}, \\
-\varpi_E^{-1} u_{i,j-1}
+(\theta_j+u_{i,j-1}) u_{i-j+1} &
 \textrm{if $2 \leq j \leq n$}. 
\end{cases}
\end{align}
We put
\begin{gather}\label{ut}
\begin{aligned}
U_i &=
\begin{cases}
u_1+u_2 & \textrm{if $i=0$}, \\
u_{1,i}+u_{2,i} & \textrm{if $1 \leq i \leq n$}, 
\end{cases} \\
Q_i &=
\begin{cases}
\varpi_E u_1 u_2 & \textrm{if $i=1$}, \\ 
u_{1,i-1}u_i+u_{2,i-1}u_{i-1}  
&  \textrm{if $2 \leq i \leq n$}.  
\end{cases}
\end{aligned}
\end{gather}
By summing up equations 
\eqref{so1} through 
$i \in \underline{2}$
respectively, and 
using \eqref{so0} and 
 \eqref{ut}, we obtain 
\begin{equation}\label{UU}
U_i^q-U_i =
\begin{cases}
 U_0=Q_0 & \textrm{if $i=1$}, \\[0.2cm]
 -\varpi_E^{-1} U_{i-1}+\theta_i U_0+
Q_i & \textrm{if $2 \leq i \leq n$}. 
\end{cases}
\end{equation}
We define an affinoid subdomain 
$\Z_{n,k,\varpi_{E,k+1}}$ by   
\begin{align*}
v(u_1^0) &\geq \frac{k-2}{4}, \quad 
v\left(u_{1,\left[\frac{k}{2}\right]}\right) \geq 4^{-1}\left(k-2\left[\frac{k}{2}\right]\right), \\ 
v(u_{1,j}) &=q^{-1}v\left(u_{1,j-1}/\varpi_E\right) \quad 
\textrm{for $[k/2]+1 \leq j \leq k$}, \quad  
v(U_k) \geq 0. 
\end{align*}
By \eqref{st}--\eqref{UU}, on  
$\Z_{n,k,\varpi_{E,k+1}}$, 
we can check that 
\begin{itemize}
\item for each $i \in \underline{2}$, 
$v(x_i+\varpi_E) \geq (k+2)/4$, \quad $
v(u_i^0)=v(u_i) \geq (k-2)/4, \\[0.2cm]  
v(u_{i,j})  \geq (k-2j)/4 \quad 
\textrm{for $1 \leq j \leq \left[k/2\right]$},  
v(u_{i,j}) =q^{-1}v\left(u_{i,j-1}/\varpi_E\right) \quad 
\textrm{for $[k/2]+1 \leq j \leq k$} 
$, and 
\item $v(U_0)  \geq (k-1)/2, \quad  
v(U_i) \geq (k-i)/2 \quad 
\textrm{for $1 \leq i \leq k$}$,   
\end{itemize}
and 
\begin{equation}\label{l2}
v(Q_i)  \geq \frac{k-i}{2} \quad 
\textrm{for $1 \leq i \leq \left[\frac{k}{2}+1\right]$}, \quad 
v(Q_i)>\frac{k-i}{2} \quad 
\textrm{for $\left[\frac{k}{2}+1\right] < i \leq k$}. 
\end{equation}
In the following, we consider only the case where 
$n=k$.
\begin{lemma}\label{rot2}
Assume that $p \neq 2$. Let $\varpi_{E,k+1}, 
\varpi'_{E,k+1} \in \mathscr{G}[\mathfrak{p}_E^{k+1}]_{\mathrm{prim}}$. \\
{\rm 1}.\ If $v(\varpi_{E,k+1}-\varpi'_{E, k+1})<(q-1)^{-1}$, 
we have 
$\Z_{k,k,\varpi_{E,k+1}} \cap \Z_{k,k,\varpi'_{E,k+1}}=\emptyset$ in $\Y(\mathfrak{p}^k)$.  \\
{\rm 2}.\  If $v(\varpi_{E,k+1}-\varpi'_{E, k+1}) \geq (q-1)^{-1}$, 
we have 
$\Z_{k,k,\varpi_{E,k+1}}=\Z_{k,k,\varpi'_{E,k+1}}$.
\end{lemma}
\begin{proof}
The parameter $U_i$ on $\Z_{k,k,\varpi_{E,k+1}}$ actually 
depends on $\varpi_{E,i+1}$. 
So, 
in this proof, 
we write $U_{\varpi_{E,i+1}}$ for $U_i$. 
We have 
$U_{\varpi_{E,i+1}}=t_{1,i}+t_{2,i}-
2 \theta_{i+1}$. 
Hence,  we have
\begin{equation}\label{zzf}
U_{\varpi_{E,k+1}}-U_{\varpi'_{E,k+1}}
=-2(\theta_{k+1}-\theta'_{k+1}). 
\end{equation}
On $\Z_{k,k,\varpi_{E,k+1}} \cap \Z_{k,k,\varpi'_{E,k+1}}$, we have 
$v(U_{\varpi_{E,k+1}}-U_{\varpi'_{E,k+1}}) \geq 0$. 
Hence, the first assertion follows. 
Assume that 
$v(\varpi_{E,k+1}-\varpi'_{E, k+1}) \geq (q-1)^{-1}$. 
This implies that $\theta_i=\theta'_i$ for 
$1 \leq i \leq k$ and $v(\theta_{k+1}-\theta'_{k+1})
\geq 0$. Therefore,
the second assertion follows from the 
definition of $\Z_{k,k,\varpi_{E,k+1}}$ 
and \eqref{zzf}.  
\end{proof}
\begin{proposition}\label{rz}
Let $k=2m-1$ be an odd positive integer. We set 
$L_k=\widehat{F}^{\mathrm{ur}}(\varpi_{E,k+1})$. 
We set 
\[
B_k=\Gamma\left(\Z_{k,k,\varpi_{E,k+1},L_k},\mathcal{O}_{\Z_{k,k,\varpi_{E,k+1},L_k}}\right).
\]
Then, $\Spf B_k^{\circ} \to 
\Spf \mathcal{O}_{L_k}$ is smoothly algebraizable.
Furthermore, we have isomorphisms  
\[
\overline{\Z}_{k,k,\varpi_{E,k+1},L_k} \simeq 
\Spec \left(B_k^{\circ} \otimes_{\mathcal{O}_{L_k}} \mathbb{F}\right) \simeq 
\Spec \mathbb{F}[a,s] \big/\left(a^q-a-s^{2q^m}\right). 
\] 
\end{proposition}
\begin{proof}
In the following, 
we always consider on 
$\Z_{k,k,\varpi_{E,k+1},L_k}$.  
By \eqref{UU}, we have 
\begin{gather}\label{rrr0}
\begin{aligned}
U_1 & \equiv -U_0  \equiv -Q_1 \mod \frac{k-1}{2}+, \\
U_i & \equiv \frac{U_{i-1}}{\varpi_E}-Q_i \mod \frac{k-i}{2}+ \quad \textrm{for $2 \leq i \leq m$}, \\
U_i & \equiv \frac{U_{i-1}}{\varpi_E} \mod \frac{k-i}{2}+\quad \textrm{for $m+1 \leq i \leq k-1$}, \\
U_k^q-U_k & \equiv -\frac{U_{k-1}}{\varpi_E} \mod 
0+. 
\end{aligned}
\end{gather}
By \eqref{so0}, \eqref{so0-1} 
and 
\eqref{so1}, for $i \in \underline{2}$, we have 
\begin{gather}\label{rrr1}
\begin{aligned}
-u_{i,1} &\equiv u_i  \equiv u_i^0 \mod \frac{k-2}{4}+, \\
u_{i,j} & \equiv \frac{u_{i,j-1}}{\varpi_E} \mod \frac{k-2j}{4}+ \quad \textrm{for $2 \leq j \leq m-1$}, \\ 
u_{i,j}^q & \equiv -\frac{u_{i,j-1}}{\varpi_E} \mod 
qv(u_{i,j})+ \quad 
\textrm{for $m \leq j \leq k$}. 
\end{aligned}
\end{gather}
We choose a second root $(-1)^{(m-1)/2} \in \mathbb{F}_{q^2}^{\times}$ 
of $(-1)^{m-1}$. 
We set as follows: 
\begin{gather}\label{rrr2}
\begin{aligned}
U_0 &=\varpi_E^{k-1} U_{0,0}, \quad 
U_i  =\varpi_E^{k-i} U_{i,0} \quad \textrm{for $1 \leq i \leq k$}, \\
u_i & =(-1)^{\frac{m-1}{2}} \varpi_E^{m-2} \theta_1^{\frac{q-1}{2}} b_i, \quad u_i^0 =(-1)^{\frac{m-1}{2}} \varpi_E^{m-2} \theta_1^{\frac{q-1}{2}} b^0_i, \\
u_{i,j} & =(-1)^{\frac{m-1}{2}} \varpi_E^{m-1-j} \theta_1^{\frac{q-1}{2}} b_{i,j} \quad \textrm{for $1 \leq j \leq m-1$}, \\
u_{i,m} &= (-1)^{\frac{m-1}{2}} \varpi_{E,2}^{-\frac{q-1}{2}} b_{i,m}, \\
u_{i,m+j} &=
(-1)^{\frac{m-1}{2}} \left(\varpi_{E,j+2} \prod_{l=2}^{j+1} \varpi_{E,l}^2\right)^{-\frac{q-1}{2}} b_{i,m+j} \quad 
\textrm{for $1 \leq j \leq m-1$} 
\end{aligned}
\end{gather}
with $v(b^0_i)$, $v(b_i)$, $v(b_{i,j})$, $v(U_{i,0}) \geq 0$. 
By using \eqref{rrr1} and \eqref{rrr2}, 
we have 
\begin{align*}
Q_1& \equiv 
 \varpi_E^{k-2}
u_{1,m-1}u_{2,m-1}=(-1)^m \varpi_E^{k-1} b_{1,m-1}
b_{2,m-1} 
\mod \frac{k-1}{2}+, \\
Q_i& \equiv 
- \varpi_E^{k-i-1} \sum_{j=1}^2
u_{j,m-1}u_{j-i+1,m-1}=
(-1)^{m-1} \varpi_E^{k-i} \sum_{j=1}^2 b_{j,m-1} 
b_{j-i+1,m-1} \mod \frac{k-i}{2}+ 
\end{align*} 
for 
$2 \leq i \leq m$. 
Hence, 
by \eqref{rrr0}, \eqref{rrr1} and \eqref{rrr2}, 
we obtain 
\begin{gather}\label{rrr3}
\begin{aligned}
& h_{1}=U_{0,0} +U_{1,0}-G_0=0, \\
& h_{i+1}=U_{i,0}-(-1)^m \left(-b_{1,m-1} b_{2,m-1}
+\sum_{j=2}^i \sum_{l=1}^2 b_{l,m-1} b_{l-j+1,m-1}\right)-G_i=0 \quad \textrm{for $1 \leq i \leq m$}, \\
& h_{i+1}=U_{i,0} -U_{i-1,0}-G_i=0 \quad \textrm{for $m+1 \leq i \leq k-1$}, \\
& h_{k+1}=U_{k,0}^q-U_{k,0}+U_{k-1,0}-G_k=0,  \\
& h_{k+2}=
b^0_1-b_1-H_0=0, \quad h_{k+3}=b_{1,1}+b_1+H_1=0, \\
& h_{k+i+2}=b_{1,i}- b_{1,i-1}-H_i=0 \quad \textrm{for $2 \leq i \leq m-1$}, \\
& h_{k+i+2}=b_{1,i}^q-b_{1,i-1}-H_i=0 \quad \textrm{for $m \leq i \leq k$}, 
\end{aligned}  
\end{gather}
with some polynomials 
\[
G_i, H_i \in \mathcal{O}_{L_k} \left[(U_{i,0})_{0 \leq i \leq k}, b_1^0, b_1, (b_{1,j})_{1 \leq j \leq k}\right]
\]
such that 
$v(G_i), v(H_i) >0$. 
We consider the subring in $B_k$: 
\[
B'_k=\mathcal{O}_{L_k} \left\langle (U_{i,0})_{0 \leq i \leq k}, b_1^0, b_1, (b_{1,j})_{1 \leq j \leq k}\right\rangle/J_k, 
\]
where $J_k$ is generated by 
\eqref{rrr3}. 
Since $B'_k \otimes_{\mathcal{O}_{L_k}} \mathbb{F}$ is reduced, and 
$B'_k \otimes_{\mathcal{O}_{L_k}}{L_k} \simeq  B_k$,
we obtain 
\begin{equation}\label{rrr4}
B'_k  \simeq B_k^{\circ}, \quad  
\overline{B_k}  =B'_k \otimes_{\mathcal{O}_{L_k}} \mathbb{F}
\end{equation}
by Lemma \ref{fund}. 
We write $a$ and $s$ for 
$U_{k,0}$ and $b_{1,k}$ respectively. 
By $b_{1,m-1} \equiv -b_{2,m-1}  \equiv s^{q^m} \mod 0+$ and \eqref{rrr3}, 
we can check that  
\[
a^q-a \equiv -U_{m,0} \equiv 
(-1)^m \left(-b_{1,m-1} b_{2,m-1}
+\sum_{i=2}^m \sum_{j=1}^2 b_{j,m-1} b_{j-i+1,m-1}\right) 
\equiv s^{2q^m} \mod 0+.  
\]
Therefore, the reduction $B'_k \otimes_{\mathcal{O}_K} \mathbb{F}_K$ is isomorphic to 
the affine curve 
defined by $a^q-a=s^{2q^m}$ by \eqref{rrr3}.  
Hence, by \eqref{rrr4}, we obtain 
\begin{equation}\label{rrr6}
\overline{B_k}=
B_k^{\circ} \otimes_{\mathcal{O}_{L_k}} \mathbb{F} \simeq \mathbb{F}[a,s]\big/\left(a^q-a-s^{2q^m}\right). 
\end{equation}
We write $(z_i)_{1 \leq i \leq 2k+3}$ 
for 
$\left((U_{i,0})_{0 \leq i \leq k}, b_1^0,b_1,(b_{1,i})_{1 \leq i \leq k}\right)$. 
Let
\[
\Delta'_k
=\det 
\left(\frac{\partial h_i}{\partial z_j}\right)_{1 \leq i,j \leq 2k+2} \in 
\mathcal{O}_{L_k}\left[z_1,\ldots, z_{2k+3}\right]. 
\]
Note that $\Delta'_k \equiv -1 \mod 
\mathfrak{p}_{L_k}$. 
We put 
\[
Z_k=\Spec \mathcal{O}_{L_k}
\left[z_1,\ldots, z_{2k+3},{\Delta'}_k^{-1}\right]\big/\left(h_1,\ldots,h_{2k+2}\right). 
\]
Then, the natural map 
$Z_k \to \Spec \mathcal{O}_{L_k}[z_{2k+3}]$
is \'etale and hence $Z_k$ is smooth over 
$\Spec \mathcal{O}_{L_k}$. 
Furthermore, 
the formal completion of $Z_k$ along 
$(Z_k)_s$ is isomorphic to $\Spf B_k$ over 
$\Spf \mathcal{O}_{L_k}$.  
Hence, the required assertions follow from \eqref{rrr6}. 
\end{proof}
\begin{remark}
We can prove that, if $k$ is even, 
the affinoid
 $\Z_{k,k,\varpi_{E,k+1}, L_k}$
reduces to a disjoint union
of $q$ copies of an affine line. 
Since we will not use this fact later,  we omit 
the details of the proof. 
\end{remark}
\begin{remark}
It is difficult to compute the reductions of 
$\X_{n,k,\zeta, \varpi_n}$ 
and 
$\Z_{n,k,\varpi_{E,n+1}}$ for $k<n$
in general. 
In a representation-theoretic view point, 
the cohomology of 
$\X_{n,n,\zeta, \varpi_n}$ 
and $\Z_{n,n,\varpi_{E,n+1}}$
is most interesting among 
$\{\X_{n,k,\zeta, \varpi_n}\}_{k \leq n}$ 
and $\{\Z_{n,n,\varpi_{E,n+1}}\}_{k \leq n}$
respectively. 
These realize ``new parts'' in 
the cohomology of the Lubin-Tate curve. 
On the other hand, the cohomology of 
$\X_{n,k,\zeta, \varpi_n}$ 
and 
$\Z_{n,k,\varpi_{E,n+1}}$ for $k<n$
is contained in ``old parts''  in 
the cohomology of the Lubin-Tate curve.
These things will be understood in this paper. 
In \cite{T}, 
we have computed the stable 
reduction of $\X(\mathfrak{p}^2)$.
Actually, in the stable reduction of 
$\X(\mathfrak{p}^2)$,
many old irreducible components appear. 
Similar examples in the mixed characteristic case can be found in \cite{CMc} and 
\cite{T2}. 
For $n >k$, we expect that 
$\X_{n,k,\zeta, \varpi_n}$ has several connected components and, the restriction 
of the canonical level lowering map $\X(\mathfrak{p}^n) \to \X(\mathfrak{p}^k)$
to each connected component 
induces a purely inseparable map
between their reductions. 
We expect that similar things happen 
for  $\Z_{n,k,\varpi_{E,n+1}}$.  
In the tower of 
modular curves $\{X_0(p^n)\}_{n \geq 0}$, 
such a phenomenon 
in a special case is observed in \cite{T3}.
\end{remark}
\subsubsection{Relation with stable reduction of 
Lubin-Tate curve}
The following corollary is a direct consequence of 
the analysis in the previous subsections. 
By this, we obtain a partial information 
on the stable reduction of the Lubin-Tate curve. 
However, we will not use 
this fact later. 
\begin{corollary}
For each $n \geq 1$, $\varpi_n \in \mathscr{F}[\mathfrak{p}_{F_2}^n]_{\mathrm{prim}}$ and 
$\varpi_{E,n+1} \in \mathscr{G}[\mathfrak{p}_E^{n+1}]_{\mathrm{prim}}$, 
the reductions of 
$\mathbf{X}_{n,n,\zeta, \varpi_n}$ and 
$\mathbf{Z}_{n,n,\varpi_{E,n+1}}$
appear as open 
dense subschemes of  
irreducible components
in the stable reductions of $\mathbf{X}(\mathfrak{p}^n)$ and of $\Y(\mathfrak{p}^n)$ respectively. 
\end{corollary} 
\begin{proof}
By Propositions \ref{rx} and \ref{rz}, 
the reductions of 
$\mathbf{X}_{n,n,\zeta, \varpi_n,K_n}$ and 
$\mathbf{Z}_{n,n,\varpi_{E,n+1},L_n}$ are smooth, 
and their smooth compactifications 
have positive genera. Hence, 
by \cite[Proposition 7.11]{IT}, the required assertion 
follows. 
\end{proof}
\section{Group action on the reductions}\label{4}
In this section, 
we give an explicit 
description 
of some  
group action on the reductions of the affinoids in
\S \ref{2}, which is induced by the group action 
on the Lubin-Tate tower summarized in 
\S \ref{got}.  

Let $K$  be a non-archimedean local field, and 
$\overline{K}$ its algebraic closure.
Let $W_K$ denote the Weil group of $K$. 
Let $\bom{a}_K \colon W_K^{\rm ab} \xrightarrow{\sim} K^{\times}$ be the Artin 
reciprocity map 
normalized such that a geometric 
Frobenius 
is sent to a prime element in $K$. 
For  $\sigma \in W_K^{\rm ab}$, 
let $n_{\sigma} \in \mathbb{Z}$ be the image of $\sigma$ by the composite of $\bom{a}_K$
and the normalized valuation $v_K \colon 
K^{\times} \to \mathbb{Z}$. 
For any Galois field extension $L/K$, 
we write $\mathrm{Gal}(L/K)$ 
for the Galois group
of the extension. 
\subsection{Review on group action on
Lubin-Tate curve}\label{got}
We briefly recall a 
group action on 
Lubin-Tate curves. In the following, 
we always consider a right action on 
spaces. Our main references are 
\cite{Ca}, \cite{Da}, \cite{Fa}, \cite{FGL} and \cite{St}. 

For any integer $h \in \mathbb{Z}$, let 
$\mathcal{R}^{(h)}(\mathfrak{p}^n)$ denote the functor 
which associates to $A \in \mathscr{C}$ the set of 
isomorphism classes of 
triples $(\mathscr{F}_A,\rho,\phi)$, 
where 
$\mathscr{F}_A$ is a formal $\mathcal{O}_F$-module 
over $A$, 
$\rho \colon \mathscr{F}_0 \to \mathscr{F}_A \otimes_A \mathbb{F}$ is 
a quasi-isogeny of height $h$,
and $\phi \colon 
(\mathcal{O}_F/\mathfrak{p}^n)^{2} \to 
\mathscr{F}_A[\mathfrak{p}^n](A)$ 
is a Drinfeld level 
$\mathfrak{p}^n$-structure. 
This is representable by 
a regular local ring 
$R^{(h)}(\mathfrak{p}^n)$ by \cite[Proposition 4.3]{Dr}. 
Let $\X^{(h)}(\mathfrak{p}^n)$
denote the generic fiber of 
$\mathcal{M}^{(h)}(\mathfrak{p}^n)=\Spf R^{(h)}(\mathfrak{p}^n)$. 
We set 
\[
\mathrm{LT}(\mathfrak{p}^n)=\coprod_{h \in \mathbb{Z}} \X^{(h)}(\mathfrak{p}^n). 
\]

Let $\mathcal{O}_D=\mathrm{End}_{\mathbb{F}}(\mathscr{F}_0)$. 
Let $\varphi \in \mathcal{O}_D$ 
be the endomorphism
of $\mathscr{F}_0$
defined by $X \mapsto X^q$. 
Then, $D=\mathcal{O}_D[\varphi^{-1}]$
is the quaternion division algebra over $F$. 
Let $D^{\times}$ act 
 on $\mathrm{LT}(\mathfrak{p}^n)$
 by 
 \[
d \colon \X^{(h)}(\mathfrak{p}^n)
\to \X^{(h+v(\Nrd_{D/F}(d)))}(\mathfrak{p}^n);\ 
 (\mathscr{F}_A,\rho,\phi) \mapsto 
 (\mathscr{F}_A,\rho \circ d,\phi) 
 \]
 for any $d \in D^{\times}$. 
 Let $\mathrm{GL}_2(\mathcal{O}_F)$
 act on 
 $\mathrm{LT}(\mathfrak{p}^n)$
 by 
 \[
 g \colon \X^{(h)}(\mathfrak{p}^n) \to 
 \X^{(h)}(\mathfrak{p}^n);\ 
 (\mathscr{F}_A,\rho,\phi) \mapsto 
 (\mathscr{F}_A,\rho,\phi \circ g) 
 \]
 for any 
 $g \in \mathrm{GL}_2(\mathcal{O}_F)$. 
 This action extends to an action of 
 $\mathrm{GL}_2(F)$ on the projective system 
 $\{\mathrm{LT}(\mathfrak{p}^n)\}_{n \geq 0}$
 (cf.\ \cite[\S 3.2.4]{Da}, 
 \cite[II.2.2]{FGL} or \cite[\S2.2]{St}).
 Let 
$U_{\mathfrak{M}}^n=\ker (\mathrm{GL}_2(\mathcal{O}_F) \to 
 \mathrm{GL}_2(\mathcal{O}_F/\mathfrak{p}^n))$
 for $n \geq 0$. 
For an open compact subgroup 
$M \subset \mathrm{GL}_2(\mathcal{O}_F)$, 
 we take $n$ such that 
 $U_{\mathfrak{M}}^n \subset 
 M$. 
 We write $\X^{(h)}_M$ for the quotient 
 $\X^{(h)}(\mathfrak{p}^n)/M$. 
 Furthermore,
 we set 
 \[
 \mathrm{LT}_M=\coprod_{h \in \mathbb{Z}} 
 \X^{(h)}_M. 
 \]
 The extended action of 
 $g \in \mathrm{GL}_2(F)$ induces morphisms 
 \begin{align*}
 g &\colon \mathrm{LT}_M \to \mathrm{LT}_{g^{-1}Mg}, 
 \\
  g & \colon \X^{(h)}_M \to \X^{(h-v(\det(g)))}_{g^{-1}Mg}. 
 \end{align*}
 The subgroup 
 $\{(x,x,1) \mid x \in F^{\times}\}
 \subset \mathrm{GL}_2(F) \times D^{\times}$
acts trivially on $\mathrm{LT}(\mathfrak{p}^n)$.

Finally, we recall the action of 
$W_F$ on 
$\mathrm{LT}(\mathfrak{p}^n)_{\C}$  
(cf.\ \cite[\S3]{Da}, 
\cite[I.1.4]{Fa} 
and \cite[Theorem 3.49]{RZ}).   
We fix an isomorphism 
\begin{equation}\label{vv11}
\mathcal{M}^{(h)}(\mathfrak{p}^n) \xrightarrow{\sim}
\mathcal{M}^{(0)}(\mathfrak{p}^n);\ 
(\mathscr{F}_A,\rho,\phi) \mapsto
(\mathscr{F}_A,\rho \circ \varphi^{-h},\phi). 
\end{equation}
There exists a formal scheme 
$\mathcal{M}^{(0)}(\mathfrak{p}^n)^0$ over 
$\Spf \mathcal{O}_F$ such that 
\begin{equation}\label{vv2}
\mathcal{M}^{(0)}(\mathfrak{p}^n) \simeq 
\mathcal{M}^{(0)}(\mathfrak{p}^n)^0 \widehat{\otimes}_{\mathcal{O}_F} 
\mathcal{O}_{\widehat{F}^{\rm ur}}. 
\end{equation}
The formal scheme 
$\mathcal{M}^{(0)}(\mathfrak{p}^n)^0$ equals 
$\mathcal{M}_{\mathcal{LT},n}$
in the notation of 
\cite[p.\ 335]{FGL}.  
We write $\X(\mathfrak{p}^n)^0$
for 
$(\mathcal{M}^{(0)}(\mathfrak{p}^n)^0)^{\rm rig}$.
Let $\sigma \in W_F$. 
By \eqref{vv11} and \eqref{vv2}, 
we identify as follows:  
\[
\mathrm{LT}(\mathfrak{p}^n)_{\C}=\coprod_{h \in 
\mathbb{Z}} \X^{(h)}(\mathfrak{p}^n)_{\C}
\simeq 
\coprod_{h \in 
\mathbb{Z}} \X^{(0)}(\mathfrak{p}^n)_{\C}
\simeq 
\coprod_{h \in 
\mathbb{Z}} \X(\mathfrak{p}^n)_h^0 \times_{\Sp F} \Sp {\C},  
\]
where 
$\X(\mathfrak{p}^n)_h^0
=\X(\mathfrak{p}^n)^0$. 
Let $\sigma \in W_F$. 
This gives the automorphism 
$\sigma^{\ast} \colon \Sp \C \to \Sp \C$. 
Let $\sigma$ denote the automorphism of 
$\mathrm{LT}(\mathfrak{p}^n)_{\C}$ 
defined by 
\[
1 \times \sigma^{\ast} \colon 
\X(\mathfrak{p}^n)_h^0 \times_{\Sp F} \Sp {\C}
\to \X(\mathfrak{p}^n)_h^0 \times_{\Sp F} \Sp {\C} 
\]
for each $h \in \mathbb{Z}$. 
Then, let $\sigma \in W_F$ act on 
$\mathrm{LT}(\mathfrak{p}^n)_{\C}$
as the composite 
\[
\varphi^{n_{\sigma}} \circ \sigma \colon 
\mathrm{LT}(\mathfrak{p}^n)_{\C} \to 
\mathrm{LT}(\mathfrak{p}^n)_{\C}. 
\]

We set 
\[
G=\mathrm{GL}_2(F) \times D^{\times} \times W_F.
\]
Let 
\[
\delta \colon G \to \mathbb{Z};\ (g,d,\sigma) \mapsto 
v(\det(g)\Nrd_{D/F}(d)^{-1}\bom{a}_F(\sigma)^{-1}) 
\]
and $G^0=\ker \delta$. 
As above, 
the tower $\{\mathrm{LT}(\mathfrak{p}^n)_{\C}\}_{n \geq 0}$ admits a $G$-action.
The subgroup $G^0$ is the 
stabilizer of 
the tower $\{\X(\mathfrak{p}^n)=\X^{(0)}(\mathfrak{p}^n)\}_{n \geq 0}$. 
\subsection{Action of Weil group 
on the reductions}\label{Wee}
\subsubsection{Unramified case}
Let $n \geq 2$.
Let $\mathscr{F}$ be as in \eqref{fg2}.  
We choose an element 
$\varpi_n \in \mathscr{F}[\mathfrak{p}_{F_2}^n]_{\rm prim}$. 
By the Lubin-Tate theory, we have 
the isomorphism 
$\iota_{\varpi_n} \colon U_{F_2}^0/U_{F_2}^n \xrightarrow{\sim} 
\mathscr{F}[\mathfrak{p}_{F_2}^n]_{\rm prim};\ 
a \mapsto [a]_{\mathscr{F}}(\varpi_n)$ and 
the commutative diagram 
\[
\xymatrix{
U_{F_2}^0/U_{F_2}^n 
\ar[d]^{\rm can.}
\ar[r]^{\!\!\!\!\!\!\!\!\simeq}_{\!\!\!\!\iota_{\varpi_n}}
 & 
\mathscr{F}[\mathfrak{p}_{F_2}^n]_{\rm prim}  \ar[d]^{[\varpi]_{\mathscr{F}}}
\\
U_{F_2}^0/U_{F_2}^{n-1} \ar[r]^{\!\!\!\!\simeq}_{\iota_{\varpi_{n-1}}}& 
 \mathscr{F}[\mathfrak{p}_{F_2}^{n-1}]_{\rm prim}. 
}
\]
Let $\zeta \in \mathbb{F}_{q^2} \setminus \mathbb{F}_q$. 
For $a \in U_{F_2}^0/U_{F_2}^{n-1}$, 
we take a lift $\widetilde{a}$ 
in $ U_{F_2}^0/U_{F_2}^{n}$. 
Then, the affinoid 
$\X_{n,n,\zeta,\iota_{\varpi_n}(\widetilde{a})}$
is independent of the choice of the lift $\widetilde{a}$
by Lemma \ref{rot} and Remark \ref{xc}.
Hence, we write $\X_{n,n,\zeta,a,\varpi_e}$ for 
this affinoid. 
We consider the union of the affinoids 
\begin{equation}\label{ya}
\X_{n,\zeta}=\bigcup_{\varpi'_n \in 
\mathscr{F}[\mathfrak{p}_{F_2}^n]_{\mathrm{prim}}}
\X_{n,n,\zeta,\varpi'_n}=
\coprod_{a \in U_{F_2}^0/U_{F_2}^{n-1}}\X_{n,n,\zeta,a,\varpi_n} \subset \X(\mathfrak{p}^n). 
\end{equation} 
We write $\mathfrak{S}_n$ for the index set 
$U_{F_2}^0/U_{F_2}^{n-1}$ in \eqref{ya}. 
In the following, 
we compute 
the action of 
\[ 
W'_{F_2}=\{(1,\varpi^{-n_{\sigma}},\sigma) \in G \mid 
\sigma \in W_{F_2}\}
\]
on the reduction 
$\overline{\X}_{n,n,\zeta, \varpi_n}$.
We identify 
$W'_{F_2}$ with $W_{F_2}$ 
by 
$(1,\varpi^{-n_{\sigma}},\sigma) \mapsto 
\sigma$. 
For an integer $n \geq 1$, 
let $\varpi_n \in \mathscr{F}[\mathfrak{p}^n_{F_2}]_{\mathrm{prim}}$, and 
set $F_{2,0}=F_2$ and 
$F_{2,n}=F_2(\varpi_n)$ for $n \geq 1$.  
We consider the homomorphism 
\[
\bom{a}_{F_2,n} \colon \mathrm{Gal}(F_{2,n}/F_2)
\to U_{F_2}^0/U_{F_2}^n;\ 
\sigma \mapsto a_{\sigma}, 
\]
where $a_{\sigma}$ is characterized by 
$\sigma(\varpi_n)=[a_{\sigma}]_{\mathscr{F}}(\varpi_n)$. 
By the Lubin-Tate theory (cf.\ \cite{Iw}), 
the map 
$\bom{a}_{F_2,n}$ is an isomorphism. 
We have the commutative diagram 
\begin{equation}\label{c1}
\xymatrix{
W_{F_2} \ar[r] & W_{F_2}^{\rm ab} \ar[r]^{\!\!\!\!\!\!\!\!\!\!\!\!\!\!\!\! \rm res.} & \mathrm{Gal}(F_{2,n}/F_2) \ar[r]^{\bom{a}_{F_2,n} \!\!\!\!\!\!\!}_{\simeq\!\!\!\!\!\!\!\!} & U_{F_2}^0/U_{F_2}^n \\
W_{F_{2,n-1}} \ar[r]\ar[u]^{\mathrm{can}.}& W_{F_{2,n-1}}^{\rm ab} \ar[u]^{\rm can.}\ar[r]^{\!\!\!\!\!\!\!\!\!\!\!\!\!\!\!\! \rm res.} & 
\mathrm{Gal}(F_{2,n}/F_{2,n-1}) \ar@{}[u]|{\bigcup}\ar[r]^{\simeq\!\!\!\!\!\!\!\!} & U_{F_2}^{n-1}/U_{F_2}^n.  \ar@{}[u]|{\bigcup}
}
\end{equation}
Note that the composite of the homomorphisms 
on the upper line equals the composite of 
the canonical map $W_{F_2} \twoheadrightarrow W_{F_2}^{\rm ab}$, 
$\bom{a}_{F_2} \colon W_{F_2}^{\rm ab} \xrightarrow{\sim} F_2^{\times}$
and $F_2^{\times} \to 
U_{F_2}^0/U_{F_2}^n;\ 
x \mapsto x/\varpi^{v_{F_2}(x)}$ 
by the Lubin-Tate theory. 
We identify $U_{F_2}^{n-1}/U_{F_2}^n$
with $\mathbb{F}_{q^2}^{\times}$ by 
$x \mapsto \bar{x}$ for $x \in U_{F_2}^0$
if $n=1$, and 
$\mathbb{F}_{q^2}$ by $1+ \varpi^{n-1}x \mapsto 
\bar{x}$ for any $x \in \mathcal{O}_{F_2}$ if $n \geq 2$. 
We write $\bom{\pi}_{n-1}
$ for the composite 
of the three homomorphisms on the lower line 
in \eqref{c1}.

\begin{lemma}\label{ural}
{\rm 1}.\ 
 The group $W_{F_2}$ acts on 
 $\overline{\X}_{1,1}$ as follows:
 \[
 \sigma \colon \overline{\X}_{1,1} 
 \to \overline{\X}_{1,1};\
 (X,Y) \mapsto \left(\bom{\pi}_0(\sigma^{-1})X^{q^{2n_{\sigma}}},
\bom{\pi}_0(\sigma^{-1}) Y^{q^{2n_{\sigma}}}\right)
 \]
 for $\sigma \in W_{F_2}$. \\
{\rm 2}.\ Let $n \geq 2$ be a positive integer. 
The induced action of $W_{F_2}$ 
on the set $\mathfrak{S}_n$ is transitive. 
The stabilizer of $\X_{n,n,\zeta,\varpi_n}$
in $W_{F_2}$ equals the subgroup 
$W_{F_{2,n-1}}$. 
Let $\sigma \in W_{F_{2,n-1}}$.
Then, $\sigma$ acts on 
$\overline{\X}_{n,n,\zeta,\varpi_n}$ as follows: 
\[
\sigma \colon 
\overline{\X}_{n,n,\zeta, \varpi_n} \to 
\overline{\X}_{n,n,\zeta, \varpi_n};\ 
(X,Y) \mapsto 
\left(X^{q^{2n_{\sigma}}}+\bom{\pi}_{n-1}(\sigma^{-1}),Y^{q^{2n_{\sigma}}}\right).  
\]
\end{lemma}
\begin{proof}
The assertion $1$ is easily checked on the 
basis of  
the computations in the proof of 
Lemma \ref{xy}. 
We omit the details.  

We prove the assertion $2$. 
The first assertion follows from 
the Lubin-Tate theory over $F_2$. 
We have 
\[
v(\sigma^{-1}(\varpi_n)-\varpi_n) \geq h \iff 
\sigma \in W_{F_{2,n-1}}. 
\] 
The second assertion follows from this and Lemma \ref{rot}. 
Finally, we prove the third assertion. 
Let $\sigma \in W_{F_{2,n-1}}$ 
and $P \in \X_{n,n, \varpi_n}(\C)$. 
Then we have $\bom{\pi}_{n-1}(\sigma^{-1})=\sigma^{-1}(\varpi_n)-\varpi_n$ in 
$\mathbb{F}_{q^2}$. 
We have $X_n(P\sigma)=\sigma^{-1}(X_n(P))$
and $Y_n(P \sigma)=\sigma^{-1}(Y_n(P))$. 
Therefore, by \eqref{ffo},  
we have 
\begin{align*}
S_n(P \sigma) &=\bom{\pi}_{n-1}(\sigma^{-1})+\sigma^{-1}(S_n(P)), \\
T_n(P \sigma) &=\zeta \bom{\pi}_{n-1}(\sigma^{-1})+\sigma^{-1}(T_n(P)).
\end{align*}
Hence, by \eqref{duo}, we obtain 
\[
U_n(P \sigma)=\zeta_1\bom{\pi}_{n-1}(\sigma^{-1})+\sigma^{-1}(U_n(P)). 
\]
By \eqref{rd}, we have 
$X(P \sigma)=\zeta_1 
\bom{\pi}_{n-1}(\sigma^{-1}) +\sigma^{-1}(X(P)) \equiv \zeta_1 \bom{\pi}_{n-1}(\sigma^{-1}) +X(P)^{q^{2n_{\sigma}}} \mod 0+$. 
We can check that 
$Y(P \sigma) \equiv 
\sigma^{-1}(Y(P)) \equiv Y(P)^{q^{2n_{\sigma}}} \mod 0+$. 
\end{proof}
\subsubsection{Ramified case}
Let $n$ be an odd positive integer. 
As in \eqref{ya}, 
by using Lemma \ref{rot2} and choosing 
$\varpi_{E,n+1} \in 
\mathscr{G}[\mathfrak{p}^{n+1}_E]_{\rm prim}$, 
we can define $\Z_{n,n,a,\varpi_{E,n+1}}$ for any 
$a \in U_E^0/U_E^n$. 
We consider 
the affinoid 
\begin{equation}\label{za}
\Z_{\varpi_E,n}=\bigcup_{\varpi'_{E,n+1} \in 
\mathscr{G}[\mathfrak{p}_E^{n+1}]_{\mathrm{prim}}}
\Z_{n,n,\varpi'_{E,n+1}}=
\coprod_{a \in U_E^0/U_E^n} 
\Z_{n,n,a,\varpi_{E,n+1}}
\subset \Y(\mathfrak{p}^n). 
\end{equation}
We write $\mathfrak{T}_n$ for the index set 
$U_E^0/U_E^n$ in \eqref{za}.  
In the following, 
we determine
the action of 
\[
W_E=\{(1,\varphi^{-n_{\sigma}},\sigma) \in G \mid 
\sigma \in W_E\}
\]
on the reduction $\overline{\Z}_{\varpi_E,n}$. 
We identify $W'_E$ with $W_E$
by $(1,\varphi^{-n_{\sigma}},\sigma) \mapsto 
\sigma$. 

Let $\mathscr{G}$ be as in \eqref{gf2}. 
For an integer $n \geq 1$, 
let $\varpi_{E,n} \in 
\mathscr{G}[\mathfrak{p}_E^n]_{\mathrm{prim}}$  
and set $E_n=E(\varpi_{E,n})$. 
We consider the homomorphism 
\[
\bom{a}_{E,n} \colon \mathrm{Gal}(E_n/E) \to 
U_E^0/U_E^n;\ 
\sigma \mapsto a_{\sigma}, 
\]
where $a_{\sigma}$ is characterized by
$\sigma(\varpi_{E,n})=
[a_{\sigma}]_{\mathscr{G}}(\varpi_{E,n})$. 
By the Lubin-Tate theory, 
this is an isomorphism. 
For $n \geq 2$, we consider the commutative diagram 
\begin{equation}\label{c2}
\xymatrix{
W_E \ar[r] & W_E^{\rm ab} \ar[r]^{\!\!\!\!\!\!\!\!\!\!\!\!\!\!\rm res.} & \mathrm{Gal}(E_n/E) \ar[r]_{\simeq\!\!\!\!\!}^{\bom{a}_{E,n}\!\!\!\!\!} &
U_E^0/U_E^n \\
W_{E_{n-1}} \ar[r]\ar[u]^{\rm can.}& W_{E_{n-1}}^{\rm ab} \ar[u]^{\rm can.}\ar[r]^{\!\!\!\!\!\!\!\!\!\!\!\!\!\!\rm res.} & 
 \mathrm{Gal}(E_n/E_{n-1}) \ar[r]^{\simeq\!\!\!\!\!}\ar@{}[u]|{\bigcup} & 
U_E^{n-1}/U_E^n \ar@{}[u]|{\bigcup}. 
}
\end{equation}
We identify $U_E^{n-1}/U_E^n$
with $\mathbb{F}_q$ 
by $1+ \varpi_E^{n-1}x  \mapsto x$
for any $x \in \mathbb{F}_q$. 
For any $n \geq 2$, 
we write $\bom{\pi}_{E,n-1}$ for the composite of
the three homomorphisms on the lower line in \eqref{c2}.

As before, 
we set 
\[
\theta_1=\varpi_{E,1}, \quad
\theta_n=\varpi_{E,n}/\varpi_{E,1} \quad 
\textrm{for $n \geq 2$}.  
\]
Then, we have the equality 
$\bom{\pi}_{E,n-1}(\sigma)=
\sigma(\theta_n)-\theta_n$
in $\mathbb{F}_q$
for $n \geq 2$ and $\sigma \in W_{E_{n-1}}^{\rm ab}$. 
\begin{lemma}\label{ral}
Let $n=2m-1 \geq 1$ be an odd integer. 
The induced action of $W_E$ 
on the set $\mathfrak{T}_n$ is transitive. 
The stabilizer of $\Z_{n,n,\varpi_{E,n+1}}$
in $W_E$ equals $W_{E_n}$. 
Let $\sigma \in W_{E_n}$.
Then, $\sigma$ acts on 
$\overline{\Z}_{n,n,\varpi_{E,n+1}}$ as follows: 
\[
\sigma \colon 
\overline{\Z}_{n,n,\varpi_{E,n+1}} \to 
\overline{\Z}_{n,n,\varpi_{E,n+1}};\ 
(a,s) \mapsto 
\left(a^{q^{n_{\sigma}}}+2 \bom{\pi}_{E,n}(\sigma^{-1}),\left(\frac{-1}{\mathbb{F}_q}\right)^{(m-1)n_{\sigma}}s^{q^{n_{\sigma}}}\right).
\]
\end{lemma}
\begin{proof}
The first assertion follows from 
the Lubin-Tate theory over $E$. 
We have 
\[
v(\sigma^{-1}(\theta_{n+1})-\theta_{n+1}) \geq 0 
\iff \sigma \in W_{E_n}. 
\]
The second assertion follows from this
and Lemma \ref{rot2}.2 and 3. 
We prove the third assertion. 
Let $P \in \Z_{n,n,\varpi_{E,n+1}}(\C)$. 
For $i \in \underline{2}$, we have 
\begin{align*}
u_{i,n}(P \sigma)=\bom{\pi}_{E, n}(\sigma^{-1})+\sigma^{-1}(u_{i,n}(P)). 
\end{align*}
Hence, 
by $U_n=u_{1,n}+u_{2,n}$, 
we have 
\[
U_n(P \sigma)=2 \bom{\pi}_{E, n}(\sigma^{-1})+\sigma^{-1}(U_n(P)) \equiv 
2 \bom{\pi}_{E, n}(\sigma^{-1})+
U_n(P)^{q^{n_{\sigma}}}
\mod 0+.
\]
By \eqref{rrr2}, we can check that 
\[
s(P \sigma) \equiv 
\frac{\sigma^{-1}((-1)^{\frac{m-1}{2}})}{
(-1)^{\frac{m-1}{2}}}\sigma^{-1}(s(P))
\equiv \left(\frac{-1}{\mathbb{F}_q}\right)^{(m-1)n_{\sigma}}s(P)^{q^{n_{\sigma}}} \mod 0+.\]
Hence, the required assertion follows. 
\end{proof}
\subsection{Action of $\mathrm{GL}(2)$ on unramified 
components}
In the following, 
we describe the action 
of some elements, which stabilizes the affinoid $\X_{n,n,\zeta,\varpi_n}$. 

Let $n \geq 2$ be a positive integer 
 until the end of 
\S \ref{diaag1}. 
Let $\zeta \in \mathbb{F}_{q^2} \setminus \mathbb{F}_q$. 
We consider the $F$-embedding 
\begin{equation}\label{aii}
\iota_{\zeta} 
\colon F_2 \hookrightarrow \mathrm{M}_2(F);\ a+b \zeta 
\mapsto 
\begin{pmatrix}
a &  -b \zeta^{q+1} \\
b & a+b(\zeta+\zeta^q)
\end{pmatrix}
\end{equation}
 for $a,b \in F$. 
 We consider the non-degenerate 
 symmetric $F$-bilinear form 
 $(\cdot, \cdot)$ on $\mathrm{M}_2(F)$ defined by 
 $(g_1,g_2)=\Tr(g_1g_2)$ for $g_1,g_2 \in \mathrm{M}_2(F)$. 
We set 
\[
C_1=\{g \in \mathrm{M}_2(F) \mid 
(x,g)=0\ \textrm{for $x \in F_2$}\}, 
\quad 
\mathfrak{C}_1
=C_1 \cap \mathrm{M}_2(\mathcal{O}_F). 
\]
Then $C_1$ and $\mathfrak{C}_1$ are stable 
under the left and right actions of $F_2$ and $\mathcal{O}_{F_2}$ respectively. 
Explicitly, we have 
\[
C_1=\left\{
g(a,b)=
\begin{pmatrix}
a & a(\zeta+\zeta^q) + b \zeta^{q+1} \\
b & -a 
\end{pmatrix} \in \mathrm{M}_2(F)\ 
\Big|\ a,b \in F \right\}. 
\]
We write 
$\mathfrak{M}$ for $\mathrm{M}_2(\mathcal{O}_F)$.
Then, we have decompositions
\begin{equation}\label{dec0}
\mathrm{M}_2(F)=F_2 \oplus C_1,\quad
\mathfrak{M}=
\mathcal{O}_{F_2} \oplus \mathfrak{C}_1
\end{equation} 
as $F_2$-vector spaces and $\mathcal{O}_{F_2}$-modules 
respectively.  
We can check
that 
$\mathfrak{C}_1 \mathfrak{C}_1
=\mathcal{O}_{F_2}$ (cf.\ \cite[Lemma 4.1]{WeJL}). 
For $m \geq 1$, 
we set $U_{\mathfrak{M}}^m=1+\mathfrak{p}^m\mathfrak{M}$ and  
\begin{equation}\label{decc1}
H_{\zeta}^n=1+\mathfrak{p}_{F_2}^{n-1}+
\mathfrak{p}^{[\frac{n}{2}]} \mathfrak{C}_1 \subset 
U_{\mathfrak{M}}^{[\frac{n}{2}]}. 
\end{equation}
This is an open compact subgroup of $\mathrm{GL}_2(F)$, because it contains $U_{\mathfrak{M}}^{n-1}$.  
Clearly, $F_2^{\times}$ normalizes $H_{\zeta}^n$. 
We consider the action of the group 
$U^0_{F_2}H_{\zeta}^n$ on 
the reduction of the affinoid 
$\X_{n,\zeta}$ in \eqref{ya}.
For an element 
$\alpha=\sum_{i=0}^{\infty} a_i \varpi^i \in 
\mathcal{O}_F^{\times}$ with 
$a_i \in \mathbb{F}_q$, we set
\[
[\alpha](X)=\sum_{i=0}^{\infty} a_iX^{q^{2i}}.
\]
For $g=\begin{pmatrix}
a & b \\
c & d
\end{pmatrix} \in \mathrm{GL}_2(\mathcal{O}_F)
$, we have 
\begin{gather}\label{abcd}
\begin{aligned}
g^\ast X_n & \equiv [a](X_n)+[c](Y_n) \mod (\varpi,u), \\ 
g^\ast Y_n & \equiv [b](X_n)+[d](Y_n) \mod (\varpi,u),  
\end{aligned}
\end{gather}
because $[\varpi]_{\mathrm{u}}(X) \equiv 
X^{q^2} \mod (\varpi,u)$. 
Let $\zeta \in \mathbb{F}_{q^2} \setminus \mathbb{F}_q$ and set $\zeta_1=\zeta^q-\zeta$. 
It is not difficult to check that  
$U^0_{F_2} H_{\zeta}^n$ acts on the affinoid 
$\X_{n,\zeta}$ by using \eqref{abcd}. 
For an element 
$x \in \mathcal{O}_{\overline{F}}$, 
we write $\bar{x}$ for 
its image by the canonical map 
$\mathcal{O}_{\overline{F}}
\to \mathbb{F}$. 
\begin{proposition}\label{ag}
The subgroup 
$U^0_{F_2}H_{\zeta}^n$
acts on the index set 
$\mathfrak{S}_n$ transitively. 
Let $\varpi_n \in \mathscr{F}[\mathfrak{p}_{F_2}^n]_{\mathrm{prim}}$. 
The stabilizer of 
$\overline{\X}_{n,n,\zeta,\varpi_n}$
in $U^0_{F_2}H_{\zeta}^n$
equals $H_{\zeta}^n$. 
Let 
$1+g=1+\varpi^{[\frac{n}{2}]}
g(a,b)
+\varpi^{n-1}
\begin{pmatrix}
a_2 & b_2 \\
c_2 & d_2
\end{pmatrix}
 \in H_{\zeta}^n$. 
 We define elements $\beta(g), \gamma(g) \in 
 \mathbb{F}_{q^2}$ as follows: 
 \begin{gather}\label{vvv}
 \begin{aligned}
 \beta(g) =\overline{a+b \zeta}, \quad
 \gamma(g) =\overline{\zeta_1^{-1}\left(\zeta^q (a_2+c_2 \zeta)-(b_2+d_2\zeta)\right)}. 
 \end{aligned}
 \end{gather}
 Then, $1+g$
acts on 
$\overline{\X}_{n,n,\zeta,\varpi_n}$
as follows: 
\begin{align*}
1+g \colon &
\overline{\X}_{n,n,\zeta,\varpi_n} 
\to 
\overline{\X}_{n,n,\zeta,\varpi_n} ;\\  
& (X,Y) 
\mapsto 
\begin{cases}
\left(X+\beta(g)^q Y^{q^{n-1}}+\gamma(g), 
Y+\beta(g)\right)& \textrm{if $n$ is odd}, \\
(X+\gamma(g),Y) & \textrm{if $n$ is even}.   
\end{cases}
\end{align*}
In particular, if $g=\varpi^{n-1} x$ for 
$x \in \mathcal{O}_{F_2}$, 
we have $\gamma(g)=\bar{x}$. 
\end{proposition}
\begin{proof}
We put $m_0=[n/2]$. 
We write $\begin{pmatrix}
a_1 & b_1 \\
c_1 & d_1
\end{pmatrix}$
for the matrix $g(a,b)$. 
By \eqref{abcd}, we have 
\begin{gather}\label{abcde}
\begin{aligned}
(1+g)^\ast X_n & \equiv 
X_n+\left([a_1](X_{n-m_0})+[c_1](Y_{n-m_0})\right)
+\left([a_2](X_1)+[c_2](Y_1)\right) \mod h+, \\
(1+g)^\ast Y_n& \equiv 
Y_n+\left([b_1](X_{n-m_0})+[d_1](Y_{n-m_0})\right)
+\left([b_2](X_1)+[d_2](Y_1)\right) \mod h+. 
\end{aligned}
\end{gather}
Hence, by \eqref{ffo} and \eqref{duo}, 
we obtain
\begin{align*}
(1+g)^\ast S_n & \equiv S_n+[a_1+c_1\zeta]\left(\varpi_{n-m_0}\right) \mod h/q^{n-1}+, \\
(1+g)^\ast U_n & \equiv U_n+\zeta_1\left(a_1+c_1 \zeta^q\right)S_{n-m_0}
+ \left(\zeta^q (a_2+c_2\zeta)-(b_2+d_2\zeta)\right) \varpi_1\mod h+. 
\end{align*}
Therefore, the required assertion follows from \eqref{rd}, \eqref{fs} and 
$\beta(g)=\overline{a+b \zeta}$. 
\end{proof}

\subsection{Action of $D^{\times}$ on unramified components}
Let $\mathscr{F}_0$ and $\mathscr{F}$ 
be as in \eqref{f0} and \eqref{fg2} respectively. 
Since we have $\mathscr{F}_0=\mathscr{F} \otimes_{\mathcal{O}_{F_2}} \mathbb{F}_{q^2}$, the reduction mod $\mathfrak{p}_{F_2}$ induces the injective $\mathcal{O}_F$-homomorphism 
\begin{equation}\label{ur}
\mathcal{O}_{F_2}=
\mathrm{End}_{\mathcal{O}_{F_2}}
(\mathscr{F}) \hookrightarrow 
\mathrm{End}_{\mathbb{F}}(\mathscr{F}_0)=\mathcal{O}_D
\end{equation}
by \cite[Proposition 4.2]{GH}. 
We have 
$\mathcal{O}_D=\mathcal{O}_{F_2} \oplus \mathcal{O}_{F_2} \varphi$ by \cite[Proposition 13.10]{GH}. 
We have 
$\varphi^2=\varpi$ and 
$\varphi \alpha=\alpha^q \varphi$ for $\alpha \in \mathbb{F}_{q^2}$. 
Let $\mathfrak{p}_D=(\varphi)$ be the two-sided 
maximal ideal of $\mathcal{O}_D$. 
We consider the closed subscheme 
$\Spf \mathcal{O}_{\widehat{F}^{\rm ur}} \hookrightarrow \Spf \mathcal{O}_{\widehat{F}^{\rm ur}}[[u]];\ u \mapsto 0$ 
corresponding to $\left(\mathscr{F} \widehat{\otimes}_{\mathcal{O}_{F_2}} \mathcal{O}_{\widehat{F}^{\rm ur}}, \rho\right)$.  
The $\mathcal{O}_{F_2}^{\times}$-action, which 
is 
induced by \eqref{ur} and the $\mathcal{O}_D^{\times}$-action on
$\Spf \mathcal{O}_{\widehat{F}^{\rm ur}}[[u]]$, fixes
the closed subscheme 
$\Spf \mathcal{O}_{\widehat{F}^{\rm ur}}$.

We set $C_2=\varphi F_2 \subset D$ and 
$\mathfrak{C}_2=C_2 \cap \mathcal{O}_D$.
We consider 
the non-degenerate $F$-bilinear 
form $(\cdot,\cdot)_D$ on $D$ defined by 
$(d_1,d_2)_D=\Trd_{D/F}(d_1d_2)$ for $d_1,d_2 \in
D$. 
Then, we have 
\[
C_2=\{d \in D \mid (d, x)_D=0\ \textrm{for any $x \in 
F_2$}\}. 
\]
Further we have natural decompositions 
\begin{equation}\label{decd0}
D=F_2 \oplus \varphi F_2,\quad 
\mathcal{O}_D=\mathcal{O}_{F_2} \oplus \mathfrak{C}_2.
\end{equation} 
Let $U_D^m=1+\mathfrak{p}_D^m$
for a positive integer $m$. 
We set 
\begin{equation}\label{deccd1}
H_{\zeta,D}^n=1+\mathfrak{p}_{F_2}^{n-1}+
\mathfrak{p}^{[\frac{n-1}{2}]} \mathfrak{C}_2 \subset 
U_D^{2[\frac{n-1}{2}]+1}.
\end{equation}
For an element $d=\sum_{i=0}^{\infty} d_i \varphi^i \in \mathcal{O}_D^{\times}$ with $d_i \in 
\mathbb{F}_{q^2}$, 
we set 
\[
[d](X)=\sum_{i=0}^{\infty} d_i X^{q^i}. 
\]
Then, we have 
\begin{gather}\label{efg}
\begin{aligned}
d^\ast X_n  \equiv [d^{-1}](X_n) \mod (\varpi,u), \quad   
d^\ast Y_n  \equiv [d^{-1}](Y_n) \mod (\varpi,u)
\end{aligned}
\end{gather}
(cf.\ \cite[Proposition14.7]{GH}). 
We describe the action of 
$U^0_{F_2} H_{\zeta,D}^n$ on the reduction of 
$\X_{n,\zeta}$ in \eqref{ya}. 
For an element 
$d \in \mathcal{O}_D$, 
we write $\bar{d}$ 
for its image by the canonical map 
$\mathcal{O}_D \to \mathbb{F}_{q^2}$. 
\begin{proposition}\label{ad}
The induced action of 
$U^0_{F_2} H_{\zeta,D}^n$ 
on $\mathfrak{S}_n$ is transitive. 
Let $\varpi_n \in \mathscr{F}[\mathfrak{p}_{F_2}^n]_{\mathrm{prim}}$. 
The stabilizer of 
$\overline{\X}_{n,n,\zeta,\varpi_n}$ in $U^0_{F_2} H_{\zeta,D}^n$
equals $H_{\zeta,D}^n$. 
Let $(1+d)^{-1}=1+\varpi^{[\frac{n-1}{2}]} \varphi x+\varpi^{n-1} y \in H_{\zeta,D}^n$. 
We set 
\[
\beta(d)=\bar{x}, \quad 
\gamma(d)=\bar{y} \in \mathbb{F}_{q^2}.
\] 
Then, $1+d$ acts on 
$\overline{\X}_{n,n,\zeta,\varpi_n}$
as follows:
\begin{align*}
1+d \colon & \overline{\X}_{n,n,\zeta,\varpi_n} \to
\overline{\X}_{n,n,\zeta,\varpi_n};\\
(X,Y)
\mapsto & 
\begin{cases}
\left(X+\beta(d)^q Y^{q^{n-1}}+\gamma(d), Y+\beta(d)^q\right)& \textrm{if $n$ is even}, \\[0.1cm]
(X+\gamma(d),Y) & \textrm{if $n$ is odd}. 
\end{cases}
\end{align*}
\end{proposition}
\begin{proof}
We set $m_1=[(n-1)/2]$. 
By \eqref{efg}, we have 
\begin{gather}\label{d0}
\begin{aligned}
(1+d)^\ast X_n & \equiv X_n+(x X_{n-m_1})^q+y X_1 \mod h+, \\
(1+d)^\ast Y_n & \equiv Y_n+(x Y_{n-m_1})^q+y Y_1
\mod h+. 
\end{aligned}
\end{gather}
Hence,  by \eqref{ffo} and \eqref{duo}, 
we have 
\begin{gather}\label{ud}
\begin{aligned}
(1+d)^\ast S_n & \equiv  S_n+(x \varpi_{n-m_1})^q \mod h/q^{n-1}+,\\
(1+d)^\ast U_n &\equiv 
U_n+\zeta_1 (x S_{n-m_1})^q+\zeta_1 y \varpi_1 \mod 
h+.
\end{aligned}
\end{gather}
The required assertion follows from \eqref{rd}, \eqref{fs}, \eqref{d0} and \eqref{ud}. 
\end{proof}
\subsubsection{Action of some diagonal elements on unramified components}\label{diaag1}
We simply write 
$G_D$ for $\mathrm{GL}_2(F) \times D^{\times}$.  
In this subsubsection, 
we describe the actions of 
some diagonal elements 
in $G_D$ and  
$W_F \times D^{\times}$ on 
$\overline{\X}_{n,n,\zeta,\varpi_n}$ in \eqref{ya}. 
\begin{lemma}\label{diggg}
We consider the $F$-embedding 
$\Delta_{\zeta} \colon F_2^{\times} \hookrightarrow G_D;\ \alpha \mapsto (\iota_{\zeta}(\alpha),\alpha)$. 
Then, $\Delta_{\zeta}(F_2^{\times})$ 
stabilizes $\overline{\X}_{n,n,\zeta,\varpi_n}$.
Let $\alpha \in F_2^{\times}$. 
Then, $\Delta_{\zeta}(\alpha)$ acts on 
$\overline{\X}_{n,n,\zeta,\varpi_n}$ as follows: 
\[
\overline{\X}_{n,n,\zeta,\varpi_n} \to 
\overline{\X}_{n,n,\zeta,\varpi_n};\
(X,Y) \mapsto \left(X,\overline{\left(\alpha/\varpi^{v_{F_2}(\alpha)}\right)}^{q-1} Y\right). 
\]
\end{lemma}
\begin{proof}
Since $\Delta_{\zeta}(F^{\times})$ acts trivially, we may assume that 
$\alpha \in  \mathcal{O}_{F_2}^{\times}$. 
We write $\alpha=a+b\zeta \in \mathcal{O}_{F_2}^{\times}$ with $a,b \in \mathcal{O}_F$. 
Let $\sigma \in \Gal(F_2/F)$ be the non-trivial element. 
By \eqref{ffo}, 
\eqref{abcd} and \eqref{efg},  
we have 
\begin{align*}
\alpha^\ast X_n & \equiv [\alpha^{-1}]\left([a](X_n)+[b]
(Y_n)\right) \equiv 
\varpi_n+[\alpha^{\sigma}/\alpha](S_n) 
\mod h+, \\
\alpha^\ast Y_n & \equiv [\alpha^{-1}]\left(
[-b \zeta^{q+1}](X_n)+
[a+b (\zeta^q+\zeta)](Y_n)\right) 
\mod h+. 
\end{align*}
Hence, by \eqref{duo}, 
we have 
\begin{align*}
\alpha^\ast S_n \equiv [\alpha^{\sigma}/\alpha](S_n)\mod h/q^{n-1}+, \quad 
\alpha^\ast U_n \equiv U_n 
\mod h+.
\end{align*}  
Therefore, the required assertion follows
from \eqref{rd} and \eqref{fs}. 
\end{proof}
Let $L$ be a non-archimedean local field. 
We take a uniformizer 
$\varpi_L$ of $L$. 
For $\sigma \in W_L$, we set 
$\bom{a}_{L, \varpi_L}^0(\sigma)=\bom{a}_L(\sigma)/
\varpi_L^{v_L(\bom{a}_L(\sigma))} \in \mathcal{O}_L^{\times}$. 
We write $I_L$ for the inertia subgroup of $L$. 
For $\sigma \in I_L$, the value 
$\bom{a}_{L,\varpi_L}^0(\sigma)$
is independent of the choice of the uniformizer, 
for which we write 
$\bom{a}^0_L(\sigma)$. 
We consider the subgroup 
\begin{equation}\label{if}
I^{(n)}_{F_2}=
\begin{cases}
\{(1,d,\sigma) \in 
G
\mid d \in 
\mathcal{O}_{F_2}^{\times},\ \sigma \in 
I_{F_2},\ \bom{a}_{F_2}^0(\sigma)d=1
\} & \textrm{if $n$ is odd}, \\ 
\{(g,1,\sigma) \in 
G
\mid g \in 
\mathcal{O}_{F_2}^{\times},\ \sigma \in 
I_{F_2},\ \bom{a}_{F_2}^0(\sigma)=g
\} & \textrm{if $n$ is even}.  
\end{cases}
\end{equation}
\begin{lemma}\label{gid}
The subgroup $I^{(n)}_{F_2}$ acts on 
$\overline{\X}_{n,n,\zeta,\varpi_n}$
trivially. 
\end{lemma}
\begin{proof}
Let $(1,d,\sigma) \in G$  be an element 
such that  $d \in \mathcal{O}_{F_2}^{\times}$, 
$\sigma \in I_{F_2}$ and $\bm{a}^0_{F_2}(\sigma)d=1$. 
Let $P \in \X_{n,n,\zeta,\varpi_n}(\C)$. 
By \eqref{efg}, we have 
\begin{equation}\label{efg'}
(d^\ast X_n)(P\sigma) \equiv 
[\bom{a}^0_{F_2}(\sigma)](\sigma^{-1}(X_n(P))) \mod 
h+. 
\end{equation}
We have $\sigma^{-1}(\varpi_n)=[\bom{a}^0_{F_2}(\sigma^{-1})]_{\mathscr{F}_0}(\varpi_n) \equiv [\bom{a}^0_{F_2}(\sigma^{-1})](\varpi_n) \mod 1$ 
by the Lubin-Tate theory. 
Hence, by \eqref{ffo} for $i=n$ 
and \eqref{efg'}, we obtain 
\begin{equation}\label{vvvc}
(d^\ast S_n)(P \sigma) \equiv 
[\bom{a}^0_{F_2}(\sigma)](\sigma^{-1}(S_n(P)))
\mod h+.
\end{equation} 
We consider the case where $n=2m-1$ is odd. 
By \eqref{fs}, \eqref{vvvc} and the Lubin-Tate 
theory, we obtain 
\[
(d^\ast Y)(P \sigma) \equiv 
\bom{a}^0_{F_2}(\sigma)\frac{\sigma^{-1}({\varpi}_m)}{\varpi_m}\sigma^{-1}(Y(P)) \equiv Y(P)
\mod 0+.
\] 
In the same way, 
we have $(d^\ast X)(P \sigma) \equiv X(P) \mod 0+$. 

We consider the case where $n=2m$ is even.  
By \eqref{rd}, 
we have 
\begin{equation}\label{asd}
(d^\ast Y)(P \sigma) \equiv 
\bom{a}^0_{F_2}(\sigma)\frac{\sigma^{-1}({\varpi}^q_{m+1})}{\varpi^q_{m+1}}\sigma^{-1}(Y(P)) \equiv d^{q-1} Y(P)
\mod 0+.
\end{equation}
For an element $(g,1,\sigma) \in I^{(n)}_{F_2}$, 
we write $(g,1,\sigma)=\Delta_{\zeta}(g)
(1,g^{-1},\sigma)$. 
Then, by Lemma \ref{diggg} and \eqref{asd},
the element $(g,1,\sigma)$ acts on the parameter 
$Y$ trivially. In the same way, 
we can check that the action of it on 
the parameter $X$ trivially.    
\end{proof}
\subsection{Action of $\mathrm{GL}(2)$ on ramified 
components}\label{GL}
As in \cite[II.2.2]{FGL}, an action of 
$G_D$ on the tower $\{\Y(\mathfrak{p}^n)\}_{n \geq 0}$
is explicitly described.

Let $E=F(\varpi_E)$  be as before. 
We consider the $F$-embedding 
\begin{equation}\label{emb1}
\Delta^{(1)}_E \colon 
E \hookrightarrow \mathrm{M}_2(F);\ 
a+b \varpi_E \mapsto 
\begin{pmatrix}
a & b \\
b \varpi & a
\end{pmatrix}\ 
\textrm{for $a,b \in F$}. 
\end{equation}
Let 
\begin{align*}
C_{1,E} &=
\left\{g \in \mathrm{M}_2(F) \mid (g,x)=0\ \textrm{for $x \in E$}\right\} \\
&=\left\{
h(x,y)=\begin{pmatrix}
x & y \\
-\varpi y  & -x
\end{pmatrix} \in \mathrm{M}_2(F)\ \Big|\ x,y \in F
\right\}, \\ 
\mathfrak{C}_{1,E}&=C_{1,E} \cap \mathrm{M}_2(\mathcal{O}_F).
\end{align*}
The former is a one-dimensional 
left and right $E$-vector space, 
and 
the latter is a free 
left and right 
$\mathcal{O}_E$-module of rank one. 
We set 
\[
\mathfrak{B}=\begin{pmatrix}
\mathfrak{p} & \mathcal{O}_F \\
\mathfrak{p} & \mathfrak{p}
\end{pmatrix} \subset 
\mathfrak{I}=\begin{pmatrix}
\mathcal{O}_F & \mathcal{O}_F \\
\mathfrak{p} & \mathcal{O}_F
\end{pmatrix} \subset 
\mathrm{M}_2(F).
\]
Then, $\mathfrak{I}$ is an $\mathcal{O}_F$-order in $\mathrm{M}_2(F)$
and $\mathfrak{B}$ is its Jacobson radical. 
The order $\mathfrak{I}$ is called the standard 
Iwahori order. 
We have natural decompositions
\begin{equation}\label{dec1}
\mathrm{M}_2(F) 
\simeq E \oplus C_{1,E}, \quad 
\mathfrak{I} \simeq \mathcal{O}_E \oplus 
\mathfrak{C}_{1,E}
\end{equation} 
as $E$-vector spaces and $\mathcal{O}_E$-modules 
respectively. 
For $m \geq 1$, we set 
$U_{\mathfrak{I}}^m=1+\mathfrak{B}^m$. 
Let $n=2m-1$ be an odd positive integer until the end of \S \ref{diag2}.
We set 
\begin{equation}\label{dece1}
H_{E}^n=1+\mathfrak{p}_E^n+\mathfrak{p}_E^m \mathfrak{C}_{1,E} \subset U_{\mathfrak{I}}^m. 
\end{equation}
Then $E^{\times}$ normalizes $H_E^n$. 
We simply write $\Spf R_n$ for the formal scheme
\eqref{in}. 
We set 
$R=(\varinjlim_n R_n)^{\widehat{}}$, 
where $(\cdot)^{\widehat{}}$ denote the 
$(x_1,x_2)$-adic completion. 
We set $z=\varpi \otimes 1 \in \mathcal{O}_F \widehat{\otimes}_{\mathbb{F}_q} R$ and 
\[
\varpi_{E,z}=\begin{pmatrix}
0 & 1 \\
z & 0
\end{pmatrix} \in 
\mathfrak{I} \widehat{\otimes}_{\mathbb{F}_q} R. 
\] 
We consider 
\[
S=\sum_{i=0}^{\infty} \mathrm{diag}(s_{1,i}, s_{2,i}) 
{}^t \varpi_{E,z}^i=
\begin{pmatrix}
\sum_{i=0}^{\infty} s_{1,2i} z^i & 
\sum_{i=1}^{\infty} s_{1,2i-1} z^i \\[0.3cm]
\sum_{i=1}^{\infty}  s_{2,2i-1} z^{i-1}& 
\sum_{i=0}^{\infty} s_{2,2i} z^i 
\end{pmatrix} \in {}^t \mathfrak{I} \widehat{\otimes}_{\mathbb{F}_q} R. 
\]
For an element $g=\sum_{i=0}^{\infty}\mathrm{diag}(g_{1,i}, g_{2,i}) \varpi_{E}^i \in 
\mathfrak{I}^{\times}=
(\mathrm{diag}(\mathbb{F}_q^2)[[\varpi_E]])^{\times}$ with 
$g_{k,i} \in \mathbb{F}_q$, 
we regard it as an element of 
$\mathfrak{I} \widehat{\otimes}_{\mathbb{F}_q} R$, for which we write $g_z$. 
As in \cite[p.\ 347]{FGL}, we define an
action of $\mathfrak{I}^{\times}$ on $R$ by 
\begin{equation}\label{fc}
g^\ast S = S\ {}^t g_z.  
\end{equation}
By regarding $S$ modulo $z^m \left({}^t \mathfrak{I} \widehat{\otimes}_{\mathbb{F}_q} R\right)$, 
we obtain an action of 
$\mathfrak{I}^{\times}$ on $R_n$. 

For an element $\alpha=\sum_{i=0}^{\infty} \alpha_i \varpi^i$
with $\alpha_i \in \mathbb{F}_q$,
we set 
\[
[\alpha](s_{i,j})
=\sum_{k=0}^{[j/2]} \alpha_k s_{i,j-2k} 
\]
for $i \in \underline{2}$. 
Note that $E^{\times}$ normalizes 
$H_E^n$. 
We describe the action of 
$U^0_E H_E^n$ on the 
reduction of $\Z_{\varpi_E,n}$ in \eqref{za}. 
\begin{lemma}
The induced 
action of $U^0_E H_{E}^n$ on the index set 
$\mathfrak{T}_n$ is 
transitive. Let $\varpi_{E,n+1} \in \mathscr{G}[\mathfrak{p}_E^{n+1}]_{\mathrm{prim}}$. 
The stabilizer of $\overline{\Z}_{n,n,\varpi_{E,n+1}}$ in $U^0_E H_E^n$ is $H_E^n$. 
Let $1+g=
1+\varpi_E^m 
h(a_1,b_1)+
\varpi_E^n
\begin{pmatrix}
a^{(2)}_{1,1} & a^{(2)}_{1,2} \\
\varpi a^{(2)}_{2,1} & a^{(2)}_{2,2}
\end{pmatrix} \in 
H_E^n
$. We set $\gamma_E(g)=\overline{a^{(2)}_{1,1}+a^{(2)}_{2,2}}$. 
Then, $1+g$ acts on 
$\overline{\Z}_{n,n,\varpi_{E,n+1}}$ as follows:
\begin{align*}
1+g \colon \overline{\Z}_{n,n,\varpi_{E,n+1}}
\to 
\overline{\Z}_{n,n,\varpi_{E,n+1}};\ 
(a,s) \mapsto 
\left(a+\gamma_E(g),s\right). 
\end{align*}
In particular, 
for $g=\varpi_E^n x$ with $x \in \mathcal{O}_E$, 
we have $\gamma_E(g)=2 \bar{x}$. 
\end{lemma}
\begin{proof}
Set $h_0=(2(q-1))^{-1}$.  
Note that 
\begin{equation}\label{exc}
\mathrm{diag}(x, y){}^t \varpi_{E,z}={}^t 
\varpi_{E,z} 
\mathrm{diag}(y, x). 
\end{equation}
Let $g'=\sum_{i=0} \varpi_E^i 
\mathrm{diag}(g_{1,i}, g_{2,i})$. 
By using 
\eqref{fc}, for $\iota \in \underline{2}$, 
we can check that  
\[
g'^\ast s_{i,k}=\sum_{j=0}^k g_{i+j,k-j}s_{i,j} 
\]
for any $0 \leq k \leq n$. 
Hence, we obtain 
\begin{align*}
(1+g)^{\ast} s_{i,n}  \equiv 
s_{i,n}+s_{i,0} a^{(2)}_{i,i}+(-1)^{m+i} 
\left([a_1](s_{i,m-1})+[b_1](s_{i,m-2})\right)\mod h_0+
\end{align*}
for $i \in \underline{2}$. 
By dividing this by $s_{i+1,0}$, 
we acquire 
\begin{align*}
(1+g)^\ast t_{i,n} & \equiv t_{i,n}+\frac{s_{i,0}}{s_{i+1,0}}a^{(2)}_{i,i}+(-1)^{m+i}\frac{[a_1](s_{i,m-1})+[b_1](s_{i,m-2})}{s_{i+1,0}} \mod 0+
\end{align*}
for $i \in \underline{2}$.
Therefore, by \eqref{rt} and \eqref{ut}, we have 
\begin{align*}
(1+g)^\ast u_{1,n} & \equiv u_{1,n} \mod v(u_{1,n})+, \\
(1+g)^\ast U_n & \equiv 
U_n+\frac{s_{1,0}}{s_{2,0}}a^{(2)}_{1,1}
+\frac{s_{2,0}}{s_{1,0}} a^{(2)}_{2,2} \mod 0+.
\end{align*}
The required assertion follows from $s_{1,0}/s_{2,0}  \equiv 1 \mod 0+$. 
\end{proof}
\subsection{Action of $D^{\times}$ on ramified components}
\label{D}
We consider 
the $F$-embedding 
\begin{equation}\label{emb2}
\Delta^{(2)}_{E} \colon E \hookrightarrow D;\ 
a+b \varpi_E \mapsto a+b \varphi  
\end{equation} 
for $a,b \in F$. 
We take an element $\zeta_0 \in 
\mathbb{F}_{q^2}^{\times}$ satisfying 
$\zeta_0^{q-1}=-1$. 
We consider an $E$-vector space 
$C_{2,E}=\zeta_0 E \subset 
D$. 
Set $\mathfrak{C}_{2,E}=C_{2,E} \cap \mathcal{O}_D$,
which is a free 
$\mathcal{O}_E$-module of rank one. 
Note that 
\[
C_{2,E}=\{d \in D \mid (d,x)_D=0\ \textrm{for any 
$x \in E$}\}. 
\]
We have decompositions 
\begin{equation}\label{decd1}
D = E \oplus C_{2,E}, \quad  
\mathcal{O}_D=\mathcal{O}_E \oplus 
\mathfrak{C}_{2,E}
\end{equation} 
as $E$-vector spaces and $\mathcal{O}_E$-modules
respectively. 
We set 
\begin{equation}\label{deccde1}
H_{E,D}^n
=1+\mathfrak{p}_E^n+\mathfrak{p}_E^m \mathfrak{C}_{2,E} \subset U_D^m. 
\end{equation}
Let $\iota \colon \mathcal{O}_{\widehat{F}^{\rm ur}}
\hookrightarrow R$ be the natural inclusion. 
For $d=\sum_{i=0}^{\infty} d_i \varphi^i \in 
\mathcal{O}_D^{\times}$ with $d_i \in \mathbb{F}_{q^2}$, 
we set 
\[
d_z=\sum_{i=0}^{\infty}\mathrm{diag}\left(\iota(d_i),  
\iota(d_i^q)\right) \varpi_E^i \in 
\mathfrak{I} \widehat{\otimes}_{\mathbb{F}_q} R. 
\]

We briefly recall \cite[II.2]{FGL}. 
Let $\mathrm{Nilp}\ \mathcal{O}_{\widehat{F}^{\rm ur}}$
be the category of 
$\mathcal{O}_{\widehat{F}^{\mathrm{ur}}}$-schemes on 
which $\varpi$ is locally nilpotent. 
For $S \in \mathrm{Nilp}\ \mathcal{O}_{\widehat{F}^{\rm ur}}$, let $S_0$ denote the closed subscheme $S \otimes_{\mathcal{O}_{\widehat{F}^{\mathrm{ur}}}} \mathbb{F}$. 
Let ${\widebreve{\mathcal{M}}}_{\mathcal{LT},\mathfrak{B}^{\times}}$ be the functor 
which associates to 
$S \in \mathrm{Nilp}\ \mathcal{O}_{\widehat{F}^{\rm ur}}$ the set of isomorphism 
classes of triples 
$(\mathscr{F}_S,\rho,H)$,
where $\mathscr{F}_S$ 
is a formal $\mathcal{O}_F$-module 
over $S$ with a quasi-isogeny 
$\rho \colon 
\mathscr{F}_0\times_{\mathbb{F}} S_0 
\to \mathscr{F}_{S_0}$, and 
$H$ is a finite flat group subscheme of 
$\mathscr{F}_S$  of degree $q$ 
with $\mathbb{F}_q$-action. 
Then, by \cite{Dr}, 
this functor is pro-representable. 
For $h \in \mathbb{Z}$, 
let ${\widebreve{\mathcal{M}}}^{(h)}_{\mathcal{LT},\mathfrak{B}^{\times}}$ denote the  
open and closed subscheme of 
${\widebreve{\mathcal{M}}}_{\mathcal{LT},\mathfrak{B}^{\times}}$, 
on which the 
universal quasi-isogeny has height $h$. 
We set 
\[
R_0=\mathcal{O}_{\widehat{F}^{\rm ur}}[[x_1, x_2]]/(x_1x_2-\varpi). 
\]
Then, we have 
${\widebreve{\mathcal{M}}}^{(0)}_{\mathcal{LT},\mathfrak{B}^{\times}} \simeq \Spf R_0$. We consider the $\mathcal{O}_{\widehat{E}^{\mathrm{ur}}}$-valued point: 
\begin{equation}\label{cmd}
\Spf \mathcal{O}_{\widehat{E}^{\rm ur}} 
\to \Spf R_0;\ x_i \mapsto -\varpi_E \quad \textrm{for 
$i \in \underline{2}$}.  
\end{equation}
As a point of ${\widebreve{M}}_{\mathcal{LT},\mathfrak{B}^{\times}}$, 
this corresponds to the isomorphism 
class of the triple 
$\left(\mathscr{G}', \rho,H_1\right)$, where $\mathscr{G}'$ 
is the base change of the formal $\mathcal{O}_E$-module 
in \eqref{gf2} to $\Spf \mathcal{O}_{\widehat{E}^{\rm ur}}$, 
$\rho \colon \mathscr{F}_0 \xrightarrow{\sim}\mathscr{G}' 
\otimes_{\mathcal{O}_{\widehat{E}^{\rm ur}}} \mathbb{F}$ is an isomorphism as
formal $\mathcal{O}_F$-modules 
over $\mathbb{F}$,  
and $H_1$ is the 
closed subscheme 
$\Spf 
\mathcal{O}_{\widehat{E}^{\rm ur}}[[X]]/([\varpi_E]_{\mathscr{G}'}(X))$ of $\mathscr{G}'$. 
The isomorphism $\rho$ induces the embedding 
$\mathcal{O}_E^{\times} \hookrightarrow \mathcal{O}_D^{\times}$ in \eqref{emb2}. 
Then, we consider 
the action of $\mathcal{O}_E^{\times}$ 
which is the restriction of the action of $\mathcal{O}_D^{\times}$
on ${\widebreve{M}}^{(0)}_{\mathcal{LT},\mathfrak{B}^{\times}}$. Then, $\mathcal{O}_E^{\times}$ fixes the point \eqref{cmd},
because the subscheme $H_1$ is stable under 
the $\mathcal{O}_E^{\times}$-action. 
Let $\mathcal{O}_{\widehat{F}^{\rm ur}}[[u]]=R(1)$
be as in \S \ref{f00}. 
Then, we have the natural map induced by 
forgetting the level structure of 
${\widebreve{M}}^{(0)}_{\mathcal{LT},\mathfrak{B}^{\times}}$:
\[
R(1)=\mathcal{O}_{\widehat{F}^{\rm ur}}[[u]] \to 
R_0;\ u \mapsto -(x_1^q+x_2). 
\]

Let $R_{\mathscr{F}^{\rm univ}}$ be the matrix 
$R_X$
when $h=0$ in the notation of \cite[Th\'{e}or\`{e}me II.2.1]{FGL}. 
Let $d \in \mathcal{O}_D^{\times}$ and 
let $\Delta_d \in ({}^t \mathfrak{I} \widehat{\otimes}_{\mathbb{F}_q} R_0)^{\times}=(\mathrm{diag}(R_0^2)[[{}^t \varpi_{E,z}]])^{\times}$ 
be the unique element $\Delta$ in the notation of 
\cite[p.\ 348]{FGL}.  
Then, by 
 the description of the action of $D^{\times}$ given in 
\cite[p.\ 348]{FGL}, the element 
$d$ induces  
\begin{align}
\label{equal}
d^\ast R_{\mathscr{F}^{\rm univ}} & = 
\Delta_d R_{\mathscr{F}^{\rm univ}} {}^t d_z, \\
\label{tor1}
d^\ast S & =\Delta_d S. 
\end{align}
By \cite[Th\'{e}or\`{e}me II.2.1]{FGL}, we have \begin{equation}\label{tor2}\Delta_d \equiv {}^t d^{-1}_z \mod (x_1,x_2). 
\end{equation}
For an element 
$\alpha=\sum_{i=0}^{\infty}\alpha_i \varpi_E^i \in \mathcal{O}_E^{\times}$ with $\alpha_i \in \mathbb{F}_q$, we set
\[
[\alpha](s_{i,j})=\sum_{k=0}^j \alpha_k s_{i,j-k} 
\]
for $i \in \underline{2}$. 
We describe 
the action of $U^0_E H_{E,D}^n$ on 
$\Z_{\varpi_E,n}$ in \eqref{za}. 
\begin{lemma}
The induced 
action of $U^0_E H_{E,D}^n$ on the index set 
$\mathfrak{T}_n$ is 
transitive. 
Let $\varpi_{E,n+1} \in \mathscr{G}[\mathfrak{p}_E^{n+1}]_{\mathrm{prim}}$. 
The stabilizer of 
$\overline{\Z}_{n,n,\varpi_{E,n+1}}$ is $H_{E,D}^n$. 
Let $(1+d)^{-1}=
1+\varphi^m 
\zeta_0 x+
\varphi^n y \in 
H_{E,D}^n
$ with $x \in \mathcal{O}_E$ and $y \in \mathcal{O}_D$. 
Then, $1+d$ acts on 
$\overline{\Z}_{n,n,\varpi_{E,n+1}}$ as follows:
\begin{align*}
1+d \colon \overline{\Z}_{n,n,\varpi_{E,n+1}}
\to 
\overline{\Z}_{n,n,\varpi_{E,n+1}};\ 
(a,s) \mapsto 
\left(a+\Tr_{\mathbb{F}_{q^2}/\mathbb{F}_q}(\bar{y}),s\right). 
\end{align*}
\end{lemma}
\begin{proof}
For $i \in \underline{2}$, we have $v(x_i)=1/2$ on $\Z_{n,n,\varpi_{E,n+1}}$. 
Let $\sigma \in \mathrm{Gal}(F_2/F)$
be the non-trivial element.
Let $i \in \underline{2}$. 
By \eqref{tor1} and \eqref{tor2}, we can check that 
\begin{align*}
(1+d)^\ast s_{i,n} & \equiv s_{i,n}+(-1)^{i-1}\zeta_0 [x](s_{i,m-1})+y^{\sigma^{i-1}} s_{i,0} \mod h_0+
\end{align*}
By dividing it by $s_{i+1,0}$, we obtain  
\begin{align*}
(1+d)^{\ast} t_{i,n}  \equiv t_{i,n}+(-1)^{i-1}\frac{\zeta_0 [x](s_{i,m-1})}{s_{i+1,0}} +\frac{s_{i,0}}{s_{i+1,0}}y^{\sigma^{i-1}} \mod 0+. 
\end{align*}
Therefore, by \eqref{rt} and \eqref{ut}, we have 
\begin{align*}
(1+d)^\ast u_{1,n} & \equiv u_{1,n} \mod v(u_{1,n})+, \\
(1+d)^\ast U_n & \equiv 
U_n+\frac{s_{1,0}}{s_{2,0}}y+\frac{s_{2,0}}{s_{1,0}} 
y^{\sigma} \mod 0+.
\end{align*}
The required assertion follows from $
s_{1,0}/s_{2,0} 
\equiv 1 \mod 0+$ and 
$y^{\sigma} \equiv y^q \mod 0+$. 
\end{proof}
\subsubsection{Action of some diagonal elements on
ramified components}\label{diag2}
We consider the diagonal $F$-embedding 
\begin{equation}\label{ebe}
\Delta_E \colon E \hookrightarrow 
\mathrm{M}_2(F) \times D;\ x \mapsto \left(\Delta^{(1)}_{E}(x),
\Delta^{(2)}_E(x)\right), 
\end{equation}
where $\Delta^{(1)}_{E}$ and $\Delta^{(2)}_{E}$
are in 
\eqref{emb1} and \eqref{emb2} respectively. 
Let $\varpi_{E,n+1} \in 
\mathscr{G}[\mathfrak{p}_E^{n+1}]_{\mathrm{prim}}$. 
\begin{lemma}\label{diagu}
Let $x \in E^{\times}$. 
The element $\Delta_E(x)$ 
acts on 
$\overline{\Z}_{n,n,\varpi_{E,n+1}}$
by $(a,s) \mapsto \left(a,(-1)^{v_E(x)}s\right)$. 
\end{lemma}
\begin{proof}
By using \eqref{exc}, 
we can directly check that
\[
\Delta_E\left(\varpi_E^j\right)^\ast s_{i,n} 
\equiv s_{i+j,n} \mod h_0+.
\]
Hence, by \eqref{so-1}, \eqref{rt}, \eqref{ut}
and \eqref{rrr2}, we have  
\begin{gather}\label{sy}
\begin{aligned}
\Delta_E\left(\varpi_E^j\right)^\ast s 
& \equiv 
(-1)^j s \mod 0+, \\
\Delta_E\left(\varpi_E^j\right)^\ast a &\equiv a \mod 0+.
\end{aligned}
\end{gather}
Let $\alpha=\sum_{i=0}^{\infty} a_i \varpi_E^i \in 
\mathcal{O}_E^{\times}$
with $a_i \in \mathbb{F}_q$.
We set $\alpha^{-1}=\sum_{i=0}^{\infty} b_i \varpi_E^i$
with $b_i \in \mathbb{F}_q$. 
By \eqref{fc}, \eqref{tor1} and \eqref{tor2}, 
for any $\iota \in \underline{2}$ and
$l \geq 1$, we have 
\begin{equation}\label{gr}
\Delta_E(\alpha)^\ast s_{\iota,l} \equiv 
\sum_{i+j+k=l,\ i, j,k \geq 0}  
a_j b_k s_{\iota+k,i} \mod h_0+.  
\end{equation}
Note that $s_{1,0}/s_{2,0} \equiv 1 \mod (k/4)$. 
By \eqref{so-1}, 
\eqref{rt}, \eqref{ut}, \eqref{rrr2} and \eqref{gr}, the element $\Delta_E(\alpha)$
acts on $\overline{\Z}_{n,n,\varpi_{E,n+1}}$
trivially. 
Hence, the required assertion follows from \eqref{sy}. 
\end{proof}
We consider the subgroup of $G$:
\begin{equation}\label{if2}
I''_E=\left\{(1,d,\sigma) \in G
\mid 
\sigma \in I_E,\ d \in 
\mathcal{O}_E^{\times},\ 
\bom{a}_E^0(\sigma)d=1
\right\}. 
\end{equation}
For $x \in \mathbb{F}_q^{\times}$,
let $\bigl(\frac{x}{\mathbb{F}_q}\bigr)
 \in \{\pm 1\}$
denote the quadratic residue symbol.
\begin{lemma}\label{diag1}
The 
$I''_E$ acts on 
$\overline{\Z}_{n,n,\varpi_{E,n+1}}$
by 
\[
(1,d,\sigma) \colon 
\overline{\Z}_{n,n,\varpi_{E,n+1}} \to \overline{\Z}_{n,n,\varpi_{E,n+1}};\ 
(a,s) \mapsto \left(a,\left(\frac{\bar{d}}{\mathbb{F}_q}\right)s\right)
\] 
for any $(1,d,\sigma) \in I''_E$. 
\end{lemma}
\begin{proof}
Let $P \in \Z_{n,n,\varpi_{E,n+1}}(\C)$
and $(1,d,\sigma) \in I''_E$.
We write 
$d^{-1}=\sum_{i=0}^{\infty} a_i \varphi^i \in 
\mathcal{O}_E^{\times}$
with $a_i \in \mathbb{F}_q$.  
Let $\iota \in \underline{2}$. 
We have 
\begin{equation}\label{zap}
d^\ast s_{\iota,n}(P \sigma) \equiv \sum_{i=0}^n
a_i \sigma^{-1}(s_{\iota+i, n-i}(P)) \mod h_0+. 
\end{equation}
By the Lubin-Tate theory and 
$\bom{a}^0_E(\sigma)d=1$, 
we have 
$\sigma^{-1}(\varpi_{E,n-i+1})=[d]_{\mathscr{G}}(\varpi_{E,n-i+1})$ for $0 \leq i \leq n$. 
Hence, by \eqref{rt} and \eqref{zap}, we obtain 
\begin{equation}\label{yap}
d^\ast u_{\iota,n}(P \sigma) \equiv 
 \sum_{i=0}^n
a_i \sigma^{-1}(u_{\iota+i, n-i}(P)) \mod 0+. 
\end{equation}
By \eqref{ut} and $\sigma \in I_E$, 
we acquire 
$d^\ast U_n(P \sigma) \equiv 
\sigma^{-1}(U_n(P)) \equiv 
U_n(P) \mod 0+$. 
By \eqref{rrr2} and \eqref{yap}, we have 
\[
d^\ast b_{\iota,n}(P \sigma) \equiv 
\left(\frac{\sigma^{-1}(\theta_1)}{\theta_1}\right)^{(q-1)/2} b_{\iota,n}(P) \mod 
0+.
\]
Hence, the required assertion follows
from 
$\overline{\sigma^{-1}(\theta_1)/\theta_1}=\bar{d}$. 
\end{proof}
\section{Preliminaries on \'etale cohomology}\label{33}
In this section, 
we collect several known 
facts on the cohomology in the form 
needed in this paper.
In \S \ref{331}, 
we give a simple 
criterion whether 
the cohomology 
of the reductions of affinoids 
in a rigid analytic variety contributes to 
the cohomology of 
the rigid analytic variety.
This criterion is a direct consequence of 
\'etale cohomology theory of 
adic spaces in \cite{Huet}. 
In \S\ref{As} and \S \ref{As2}, 
we collect classically well-known 
facts on the cohomology of Artin-Schreier curves
and the Deligne-Lusztig 
curve for $\mathrm{GL}_2(\mathbb{F}_q)$ respectively.
\subsection{Preliminary on 
\'{e}tale cohomology of rigid analytic varieties}\label{331}
We recall several general facts on 
\'etale cohomology of rigid analytic varieties 
from \cite{Huet}. 

Let $K$ be a non-archimedean complete 
discrete valued field. 
 Assume that 
its residue field is separably closed field
of characteristic $p>0$. 
We fix a separable algebraic closure 
$\overline{K}$ of $K$. 
Let $\C$ denote the completion of $\overline{K}$. 
Let $\ell \neq p$ be a prime number. 
In the following, 
we consider $\ell$-adic 
\'{e}tale cohomology (with compact support) of taut and separated rigid analytic varieties
(cf.\ \cite[\S0 and \S 5.6]{Huet}).  
We regard a rigid analytic variety 
as an adic space as in \cite[(1.1.11)]{Huet}. 

Let $\X$ be a taut and separated  
rigid analytic variety
over $K$. 
For a taut and separated morphism of rigid analytic varieties 
$f \colon \X \to 
\Y$, we write 
$Rf_! $ for the functor 
$R^+f_! $ 
in the notation of \cite[(5.6.10)]{Huet}. 
We briefly 
recall the definition of $Rf_!$. 
Let $f \colon 
\X \to \Y$ be a taut and 
separated morphism of rigid analytic 
varieties. 
Then, by \cite[Proposition 0.4.9 and 
Corollary 5.1.12]{Huet}, 
there exists a commutative diagram
of adic spaces 
\begin{equation}\label{dg}
\xymatrix{
\X \ar@{^{(}->}[r]^{j} \ar[d]_f &  X'  \ar[dl]^{f'}\\ 
\Y,  
}
\end{equation}
where $j$ is an open immersion and 
$f'$ is partially proper (cf.\ 
\cite[Definition 1.3.3 ii)]{Huet}). 
Note that partially proper is taut (cf.\ 
\cite[Lemma 5.1.10 i)]{Huet}). 
Let $Rf'_!$ be the right derived functor of 
the left exact functor $f'_!$ 
(cf.\ \cite[\S 0.4 A), \S 5.2 and \S 5.3]{Huet}). 
Then, 
we set 
\begin{equation}\label{defr}
Rf_! =Rf'_! \circ j_!. 
\end{equation}
This definition is independent of 
the choice of the compactification $j \colon 
\X \hookrightarrow X'$. 
Note that $Rf_!$ does not correspond to 
the right derived functor of $f_!$ in general (cf.\ 
\cite[\S0.4 B)]{Huet}). 
The natural transformations 
$Rf'_! \to Rf'_{\ast}$ and $j_! \to Rj_{\ast}$
induce $Rf_! \to Rf_{\ast}$. 

We set $\La_n=\mathbb{Z}/\ell^n \mathbb{Z}$ for $n \geq 1$. 
For a taut and separated rigid analytic variety 
$f \colon \X \to \Spa (K,\mathcal{O}_K)$, 
we set 
\begin{align*}
H^n(\X_{\C}, \overline{\mathbb{Q}}_{\ell})
& =\biggl(\varprojlim_n (Rf_\ast \La_n)_{\Spa (\C,\mathcal{O}_{\C})}\biggr) \otimes_{\mathbb{Z}_{\ell}} \overline{\mathbb{Q}}_{\ell}, \\ 
H_{\mathrm{c}}^n (\X_{\C},\overline{\mathbb{Q}}_{\ell})
&=\biggl(\varprojlim_n (Rf_! \La_n)_{\Spa (\C,\mathcal{O}_{\C})}\biggr) \otimes_{\mathbb{Z}_{\ell}} \overline{\mathbb{Q}}_{\ell} 
\end{align*}
(cf.\ \cite[Example 2.6.2 and Corollary 5.4.8]{Huet}). 
By the natural transformation $Rf_! \to Rf_\ast$, 
we have the canonical map 
\[
{\mathrm{can}}. \colon H_{\mathrm{c}}^n (\X_{\C},\overline{\mathbb{Q}}_{\ell})
\to 
H^n (\X_{\C},\overline{\mathbb{Q}}_{\ell}). 
\]

We consider a commutative diagram
of taut and separated rigid analytic varieties 
\[
\xymatrix{
\W \ar@{^{(}->}[r]^{j}\ar[dr]_{f_W} & 
\X \ar[d]^f \\
 & \Spa (K,\mathcal{O}_{K}),}
\]
where
$j$ is an open immersion. 
By applying $Rf_!$ to the adjunction 
map $j_! \La_n \to \La_n$ and using 
a natural isomorphism 
${Rf_W}_! \La_n \xrightarrow{\sim} Rf_! j_! \La_n$
in \cite[Theorem 5.4.3]{Huet}, 
we have ${Rf_W}_! \La_n \to Rf_! \La_n$. 
This induces  
the canonical map 
\begin{equation}\label{ccc}
H^i_{\mathrm{c}}(\W_{\C},\overline{\mathbb{Q}}_{\ell})
\to 
H_{\rm c}^i(\X_{\C},\overline{\mathbb{Q}}_{\ell}). 
\end{equation}
By applying 
$Rf_\ast$ to the 
adjunction map $\La_n \to Rj_{\ast} \La_n$, 
we have $Rf_{\ast} \La_n \to 
{Rf_W}_{\ast} \La_n$. 
This induces the restriction map 
\begin{equation}\label{ddd}
H^i(\X_{\C},\overline{\mathbb{Q}}_{\ell})
\to 
H^i(\W_{\C},\overline{\mathbb{Q}}_{\ell}). 
\end{equation}
\begin{lemma}\label{aho}
For each $i$, we have 
the commutative diagram 
\begin{equation}\label{com}
\xymatrix{
H^i_{\rm c}(\W_{\C},\overline{\mathbb{Q}}_{\ell})
\ar[r]^{\eqref{ccc}} \ar[d]^{\rm can.} & H^i_{\rm c}(\X_{\C},\overline{\mathbb{Q}}_{\ell}) \ar[d]^{\rm can.} \\
H^i(\W_{\C},\overline{\mathbb{Q}}_{\ell})
& H^i(\X_{\C},\overline{\mathbb{Q}}_{\ell}).\ar[l]_{\eqref{ddd}}}
\end{equation}
\end{lemma}
\begin{proof}
By \cite[Proposition 4.16]{Mi}, we have the commutative diagram
\begin{equation}\label{c1c}
\xymatrix{
Rf_! j_! \La_n \ar[r]^{\!\!\!\!\!\!\!(1)} \ar[d]^{\simeq} &
 Rf_! R j_\ast \La_n \ar[r]^{\!\!\!\!(2)} &  Rf_\ast R j_\ast \La_n \ar@{=}[dl]\\
R{f_W}_! \La_n \ar[r]^{\!\!\!\!(3)}& 
R{f_W}_{\ast} \La_n, &  
}
\end{equation}
where $(1)$, $(2)$ and $(3)$ are induced by the 
natural transformations  
$j_! \to Rj_{\ast}$, $Rf_! \to Rf_\ast$ and $R{f_W}_! \to
R{f_W}_\ast$ respectively. 
Note that the composite of the two adjoint maps 
$j_!j^\ast \to \mathrm{id}$ and 
$\mathrm{id} \to Rj_{\ast}j^\ast$ equals the natural 
transformation $j_!j^\ast \to Rj_{\ast}j^\ast$. 
We have the commutative diagram  
\begin{equation}\label{c2c}
\xymatrix{
& Rf_! \La_n \ar[r]^{(2)'}\ar[d]^{(b)} & Rf_\ast \La_n \ar[d]^{(b)'} \\
Rf_! j_! \La_n \ar[r]^{\!\!\!\!(1)}\ar[ur]^{(a)} &
 Rf_! R j_\ast \La_n \ar[r]^{\!\!\!\!(2)} &  Rf_\ast R j_\ast \La_n, 
}
\end{equation}
where $(a)$ is induced by $j_!j^\ast \to \mathrm{id}$, 
$(b)$ and $(b)'$ are induced by $\mathrm{id} \to Rj_{\ast}j^\ast$, and 
$(2)'$ is induced by $Rf_! \to Rf_\ast$. 
By considering the stalk at 
$\Spa(\C,\mathcal{O}_{\C})$
of the diagrams 
\eqref{c1c} and \eqref{c2c}, we obtain 
the commutative diagram
\[
\xymatrix{
H^i_{\rm c}(\W_{\C},\La_n)
\ar[r] \ar[d]^{\rm can.} & H^i_{\rm c}(\X_{\C},\La_n) \ar[d]^{\rm can.} \\
H^i(\W_{\C},\La_n)
& H^i(\X_{\C},\La_n).\ar[l]}
\]
By taking $\varprojlim_n$ of this diagram 
and applying 
$(-) \otimes_{\mathbb{Z}_{\ell}} \overline{\mathbb{Q}}_{\ell}$, we obtain the claim. 
\end{proof}
\subsubsection{Formal nearby cycle functor}
Let $K$ be a non-archimedean 
valued field of height one.  
Let $\widehat{K}$ denote the 
completion of $K$. 
Let $\mathcal{X}$ be a formal scheme which 
is locally finitely presented over 
$\mathcal{S}=\Spf \mathcal{O}_{\widehat{K}}$.
We set $\mathcal{X}_s=
\left(\mathcal{X}, \mathcal{O}_{\mathcal{X}}/\mathfrak{p}_{\widehat{K}} \mathcal{O}_{\mathcal{X}}\right)$. 
Then, we have the morphism of \'etale sites 
\[
\lambda_{\mathcal{X}} \colon (\mathcal{X}^{\mathrm{rig}})_{\mathrm{\acute{e}t}} 
\to \mathcal{X}_{\mathrm{\acute{e}t}} \simeq 
(\mathcal{X}_s)_{\mathrm{\acute{e}t}}, 
\]
which is given in 
\cite[(0.7.1) and Lemma 3.5.1]{Huet}. 
\begin{definition}
We write $R\Psi_{\mathcal{X}}^{\mathrm{ad}}$ 
for ${R\lambda_{\mathcal{X}}}_{\ast}$, 
which we call the formal nearby cycle functor. 
\end{definition}
Let 
$K$ be as in the beginning of \S \ref{331}. 
Let $\W=\Sp A$ be an affinoid variety over $K$.
We consider the formal 
scheme 
$\mathcal{W}=\Spf A^{\circ} \to 
\mathcal{S}$.  We regard it as an object 
in $\mathcal{O}_K$-$\mathcal{F}sch$, because 
$K$ has a discrete valuation 
(cf.\ \cite[Introduction]{BLR}).    
 
We assume that the formal scheme 
$\mathcal{W}$ is 
isomorphic to 
the formal completion of a scheme 
$W$, which is separated and of finite type 
over $S=\Spec \mathcal{O}_K$, along 
the special fiber $W_s$. 
Let $W_{\overline{S}}$
denote the base change of $W$ to 
$\overline{S}=\Spec \mathcal{O}_{\overline{K}}$. 
Let $W_{\bar{\eta}}$ denote the generic fiber of 
$W_{\overline{S}}$. Let 
$\mathcal{W}_{\overline{S}}$ denote the 
completion of $W_{\overline{S}}$ along the special 
fiber $W_s$. 
Then we have $\mathbf{W}=\mathcal{W}^{\mathrm{rig}}$
and  
$\mathbf{W}_{\C}=\mathcal{W}_{\overline{S}}^{\mathrm{rig}}$. 

Let $W_{\overline{S}} \subset W^{\mathrm{c}}$ be a compactification 
of $W_{\overline{S}}$ over $\overline{S}$. 
Let $\mathcal{W}_{\overline{S}}^{\mathrm{c}}$ be the formal completion 
of $W^{\mathrm{c}}$ along the special fiber. 
Then, 
we have the commutative diagram of formal schemes
\begin{equation}\label{dg6}
\xymatrix{
\mathcal{W}_{\overline{S}} \ar@{^{(}->}[r]^{j}\ar[d]_f & \mathcal{W}_{\overline{S}}^{\mathrm{c}} \ar[dl]
^{f^{\rm c}\!\!\!\!\!\!\!\!\!\!} \\
\Spf \mathcal{O}_{\C}, 
} 
\end{equation}
where $j$ is an open immersion (cf.\ 
\cite[the proof of Corollary 0.7.9]{Huet}
and \cite[Example 4.22 ii)]{Mi}). 
We write $\W_{\C}^{\rm c}$ for 
the rigid analytic variety 
$(\mathcal{W}_{\overline{S}}^{\mathrm{c}})^{\mathrm{rig}}$ over 
$\mathbf{C}$.
The diagram \eqref{dg6} induces 
the commutative diagram of rigid analytic varieties 
\[
\xymatrix{
\W_{\C} \ar@{^{(}->}[r]^{j_{\bar{\eta}}} \ar[d]_{f_{\bar{\eta}}} & \W_{\C}^{\mathrm{c}} \ar[dl]^{f_{\bar{\eta}}^{\rm c}\!\!\!\!\!\!\!\!\!\!} \\
\Spa (\C, \mathcal{O}_{\C}),
}
\] 
where 
$f_{\bar{\eta}}^{\rm c}$
is proper. 
By \eqref{dg6}, we have the commutative diagram 
of schemes 
\[
\xymatrix{
\mathcal{W}_{s} \ar@{^{(}->}[r]^{j_s}\ar[d]_{f_s} & \mathcal{W}_{s}^{\mathrm{c}} \ar[dl]^{f_s^{\rm c}\!\!\!\!\!\!\!\!\!\!} \\
\Spec \mathbb{F}. 
}
\]

We have 
\begin{equation}\label{iol2}
H^i(\mathcal{W}_s, R\Psi^{\mathrm{ad}}_{\mathcal{W}_{\overline{S}}}(\La_n))=
H^i(\W_{\C},\La_n).
\end{equation}
We recall a natural isomorphism 
\[
\xi \colon 
H^i_{\rm c}(\mathcal{W}_s, R\Psi^{\mathrm{ad}}_{\mathcal{W}_{\overline{S}}}(\La_n))
\xrightarrow{\sim} H^i_{\rm c}(\W_{\C},\La_n).  
\]
We define $\xi$ to be the composite 
of the following isomorphisms:
\begin{align*}
H^i_{\rm c}(\mathcal{W}_s, R\Psi^{\mathrm{ad}}_{\mathcal{W}_{\overline{S}}}(\La_n))
=H^i(\mathcal{W}_s^{\rm c}, {j_s}_!
R\Psi^{\mathrm{ad}}_{\mathcal{W}_{\overline{S}}}(\La_n)) \\
 \xrightarrow{G}
H^i(\mathcal{W}_s^{\rm c}, R\Psi^{\mathrm{ad}}_{\mathcal{W}^{\rm c}_{\overline{S}}}({j_{\bar{\eta}}}_!\La_n))=
H^i(\W^{\rm c}_{\C}, {j_{\bar{\eta}}}_!\La_n)
&=H_{\rm c}^i(\W_{\C}, \La_n), 
\end{align*}
where  
$G$ is induced by the natural isomorphism
${j_s}_! R\Psi^{\mathrm{ad}}_{\mathcal{W}_{\overline{S}}}
(\La_n) \xrightarrow{\sim}
R\Psi^{\mathrm{ad}}_{\mathcal{W}_{\overline{S}}^{\mathrm{c}}}
({j_{\bar{\eta}}}_! \La_n)$ in 
\cite[Corollary 0.7.5 and Corollary 3.5.11]{Huet}, and the last 
equality follows from properness of 
$f_{\bar{\eta}}^{\rm c}$ and \eqref{defr}.
The isomorphism $\xi$ is independent of 
the choice of the compactification 
$j \colon \mathcal{W}_{\overline{S}} 
\hookrightarrow \mathcal{W}_{\overline{S}}^{\rm c}$ by 
\cite[Lemma 4.25 i)]{Mi}.  
In \cite[Definition 4.24]{Mi},
the map $\xi$ is defined for a more general formal 
scheme. 
\begin{lemma}\label{formal}
We have the commutative diagram
\[
\xymatrix{
H^i_{\rm c}(\mathcal{W}_s, R\Psi^{\mathrm{ad}}_{\mathcal{W}_{\overline{S}}}(\La_n)) \ar[r]_{\xi\!\!\!\!\!\!\!\!\!}^{\simeq\!\!\!\!\!\!\!\!\!}\ar[d]^{\rm can.} & H^i_{\rm c}(\W_{\C},\La_n) \ar[d]^{\rm can.}\\
H^i(\mathcal{W}_s, R\Psi^{\mathrm{ad}}_{\mathcal{W}_{\overline{S}}}(\La_n)) \ar@{=}[r]^{\eqref{iol2}\!\!\!\!\!\!\!\!\!} & H^i(\W_{\C},\La_n). 
}
\]
\end{lemma}
\begin{proof}
We have the commutative diagram of \'etale sites
\begin{equation}\label{dg2}
\xymatrix{
(\W_{\C})_{\mathrm{\acute{e}t}} 
\ar[r]^{j_{\bar{\eta}}} \ar[d]^{\lambda_{\mathcal{W}_{\overline{S}}}} 
& (\W_{\C}^{\mathrm{c}})_{\mathrm{\acute{e}t}} \ar[d]^{\lambda_{\mathcal{W}_{\overline{S}}^{\rm c}}} \ar[r]^{\!\!\!\!\!\!\!\!\!\!\!\!\!\!\!\!\!\!f_{\bar{\eta}}^{\rm c}}\ar[d]^{\lambda_{\mathcal{W}_{\overline{S}}^{\mathrm{c}}}} & (\Spa(\C,\mathcal{O}_{\C}))_{\mathrm{\acute{e}t}} \ar[d]^{\lambda_{\mathcal{O}_{\C}}}\\
(\mathcal{W}_s)_{\mathrm{\acute{e}t}} \ar[r]^{j_s} & 
(\mathcal{W}^{\mathrm{c}}_s)_{\mathrm{\acute{e}t}}
\ar[r]^{\!\!\!\!\!\!\!\!\!f_s^{\rm c}} & (\Spec \mathbb{F})_{\mathrm{\acute{e}t}}. 
}
\end{equation}
Hence, we have the commutative diagram 
\begin{equation}\label{dg3}
\xymatrix{
{j_s}_! {R\Psi^{\rm ad}_{\mathcal{W}_{\overline{S}}}}
(\La_n) \ar[r]^{\simeq}\ar[d] 
& {R\Psi^{\rm ad}_{\mathcal{W}^{\rm c}_{\overline{S}}}} ({j_{\bar{\eta}}}_! \La_n) \ar[d] \\
R{j_s}_\ast {R\Psi^{\mathrm{ad}}_{\mathcal{W}_{\overline{S}}}} 
(\La_n) \ar[r]^{\!\!\!\!\simeq} & {R\Psi^{\rm ad}_{\mathcal{W}^{\rm c}_{\overline{S}}}}(R{j_{\bar{\eta}}}_\ast \La_n),  
}
\end{equation}
where the above horizontal isomorphism 
follows from \cite[Corollary 0.7.5]{Huet}, and
the left vertical and the right vertical 
morphisms are induced 
by ${j_s}_! \to R{j_s}_\ast$ and 
${j_{\bar{\eta}}}_! \to R{j_{\bar{\eta}}}_{\ast}$ respectively. 
By applying ${Rf^{\rm c}_s}_{\ast}$
to the diagram \eqref{dg3} and using \eqref{dg2}, 
we obtain 
the commutative diagram
\begin{equation*}
\xymatrix{
{Rf_s}_! {R\Psi^{\rm ad}_{\mathcal{W}_{\overline{S}}}}
(\La_n) \ar@{=}[r]\ar[d] & {Rf^{\rm c}_s}_{\ast} {j_s}_! {R\Psi^{\rm ad}_{\mathcal{W}_{\overline{S}}}} 
(\La_n) \ar[r]^{\simeq}\ar[d] 
& {Rf^{\rm c}_s}_{\ast} {R\Psi^{\rm ad}_{\mathcal{W}^{\rm c}_{\overline{S}}}} ({j_{\bar{\eta}}}_! \La_n) \ar[d] \\
{Rf_s}_{\ast}{R\Psi^{\mathrm{ad}}_{\mathcal{W}_{\overline{S}}}}
(\La_n)  \ar@{=}[r] & {Rf^{\rm c}_s}_{\ast} R{j_s}_\ast {R\Psi^{\mathrm{ad}}_{\mathcal{W}_{\overline{S}}}} 
(\La_n) \ar[r]^{\!\!\!\!\simeq} & {Rf^{\rm c}_s}_{\ast} {R\Psi^{\rm ad}_{\mathcal{W}^{\rm c}_{\overline{S}}}} (R{j_{\bar{\eta}}}_\ast \La_n)
}
\end{equation*}
\begin{equation*}
\xymatrix{
\ar@{=}[r]  & {R\lambda_{\mathcal{O}_{\C}}}_{\ast} 
{Rf^{\rm c}_{\bar{\eta}}}_{\ast} 
{j_{\bar{\eta}}}_! \La_n \ar@{=}[r]
\ar[d] & {R\lambda_{\mathcal{O}_{\C}}}_{\ast} 
{Rf_{\bar{\eta}}}_! \La_n \ar[d]\\
 \ar@{=}[r]  & {R\lambda_{\mathcal{O}_{\C}}}_{\ast} 
{Rf^{\rm c}_{\bar{\eta}}}_{\ast} R{j_{\bar{\eta}}}_{\ast} \La_n \ar@{=}[r] & {R\lambda_{\mathcal{O}_{\C}}}_{\ast} 
{Rf_{\bar{\eta}}}_\ast \La_n. 
}
\end{equation*}
Hence, the required assertion follows. 
\end{proof}
Let $R\Psi_{W}(\La_n)$ denote the 
nearby cycle complex of the constant sheaf $\La_n$ 
on the scheme   
$W$ over $S$ in \cite[\S2.1]{Del2}. 
By \cite[Theorem 0.7.7 or Theorem 3.5.13]{Huet}, 
we have the natural isomorphism 
\begin{equation}\label{formal1/2}
R\Psi_{W}(\La_n) \xrightarrow{\sim} 
R\Psi^{\rm ad}_{\mathcal{W}_{\overline{S}}}(\La_n). 
\end{equation}
Hence, 
we have natural isomorphisms induced by 
\eqref{iol2} and $\xi$: 
\begin{gather}\label{formal3}
\begin{aligned}
H^i(\mathcal{W}_s, R\Psi_{W}(\La_n)) & \simeq 
H^i(\W_{\C},\La_n), \\
H_{\rm c}^i(\mathcal{W}_s, R\Psi_{W}(\La_n)) & \simeq 
H_{\rm c}^i(\W_{\C},\La_n) 
\end{aligned}
\end{gather}
as in \cite[Corollary 0.7.9, Corollary 3.5.14 and 
Theorem 5.7.6]{Huet}. 
Note that the latter isomorphism is generalized in 
\cite[Lemma 2.13]{Hu2}. 
\begin{corollary}\label{formal2}
We have the commutative diagram
\[
\xymatrix{
H^i_{\rm c}(\mathcal{W}_s, R\Psi_{W}(\La_n)) \ar[r]^{\simeq\!\!\!\!\!\!\!\!\!}\ar[d]^{\rm can.} & H^i_{\rm c}(\W_{\C},\La_n) \ar[d]^{\rm can.}\\
H^i(\mathcal{W}_s, R\Psi_{W}(\La_n)) \ar[r]^{\simeq\!\!\!\!\!\!\!} & H^i(\W_{\C},\La_n). 
}
\]
\end{corollary}
\begin{proof}
The required assertion follows from Lemma \ref{formal} and \eqref{formal3}. 
\end{proof}
\subsubsection{Key lemma}
We state a key lemma to relate
the cohomology of the reductions of the affinoids 
in the Lubin-Tate curve in \S \ref{2} to the cohomology 
of the Lubin-Tate curve.  
We consider the following situation.
Let $K$ be as in the beginning of \S \ref{331}. 
Let $\X$ be a taut and separated rigid analytic variety over $K$. 
Let $\W=\Sp A \subset \X$ be an affinoid subdomain.
Assume that 
\begin{enumerate}
\item the morphism $\Spf A^{\circ} \to \Spf \mathcal{O}_K$ is smoothly algebraizable, and 
\item $\overline{\W}=\Spec \left(A^{\circ} \otimes_{\mathcal{O}_K} \mathbb{F} \right)$. 
\end{enumerate}
By the first assumption,
we take a scheme $W$ which is 
separated, smooth and 
of finite type over 
$S$ and whose formal completion 
along $W_s$ is isomorphic to 
$\mathcal{W}=\Spf A^{\circ}$. 
Since $W \to S$ is smooth, 
the natural morphism 
$\La_n \to R\Psi_W( \La_n)$ 
is an isomorphism by the smooth 
base change theorem as in 
\cite[Reformation 2.1.5]{Del2}. 
By the second assumption, 
we have $\overline{\W} \simeq \mathcal{W}_s$. 
Hence, by Corollary \ref{formal2}, 
we have the commutative diagram 
\begin{equation}\label{ahoy}
\xymatrix{
H_{\rm c}^i(\overline{\W},\overline{\mathbb{Q}}_{\ell}) \ar[r]^{\!\!\!\!\!\!\simeq}\ar[d]^{\rm can.} & 
H_{\rm c}^i(\W_{\C},\overline{\mathbb{Q}}_{\ell})\ar[d]^{\rm can.} \\
H^i(\overline{\W},\overline{\mathbb{Q}}_{\ell}) \ar[r]^{\!\!\!\!\!\!\simeq} & 
H^i(\W_{\C},\overline{\mathbb{Q}}_{\ell}). 
}
\end{equation}
\begin{lemma}\label{top}
Let the notation and the assumption be as above. \\ 
{\rm 1}.\ We have 
the following commutative diagram: 
\begin{equation}\label{kids}
\xymatrix{
H_{\rm c}^i(\overline{\W},\overline{\mathbb{Q}}_{\ell}) \ar[r]^{\!\!\!\!\simeq}\ar[d]^{\rm can.} & H^i_{\rm c}(\W_{\C},\overline{\mathbb{Q}}_{\ell})
\ar[r]^{\eqref{ccc}} \ar[d]^{\rm can.} & H^i_{\rm c}(\X_{\C},\overline{\mathbb{Q}}_{\ell}) \ar[d]^{\rm can.} \\
H^i(\overline{\W},\overline{\mathbb{Q}}_{\ell}) \ar[r]^{\!\!\!\!\simeq} & 
H^i(\W_{\C},\overline{\mathbb{Q}}_{\ell})
& H^i(\X_{\C},\overline{\mathbb{Q}}_{\ell}). \ar[l]^{\eqref{ddd}}}  
\end{equation}

Let 
\begin{equation}\label{ca}
H_{\rm c}^i(\overline{\W},\overline{\mathbb{Q}}_{\ell}) \to 
H^i_{\rm c}(\X_{\C},\overline{\mathbb{Q}}_{\ell})
\end{equation}
be the composite 
of the maps in the above horizontal maps in 
\eqref{kids}.  \\
{\rm 2}.\ Assume that the canonical map  
$
H_{\rm c}^i(\overline{\W},\overline{\mathbb{Q}}_{\ell}) \to 
H^i(\overline{\W}, \overline{\mathbb{Q}}_{\ell})
$
is injective on a subspace $W \subset H_{\rm c}^i(\overline{\W},\overline{\mathbb{Q}}_{\ell})$. 
Then, the restriction map $W \to H_{\mathrm{c}}^i(\X_{\C},\overline{\mathbb{Q}}_{\ell})$ of 
the canonical map \eqref{ca} to  
$W$
is an injection. 
\end{lemma}
\begin{proof} 
A quasi-compact and quasi-separated 
 rigid analytic variety 
is taut by \cite[Lemma 5.1.3 iv)]{Huet}.  
Hence, any affinoid rigid analytic variety 
is taut. 
The first assertion follows from 
Lemma \ref{aho} and 
the commutative diagram
\eqref{ahoy}.  
The second assertion immediately follows from 
\eqref{kids}. 
\end{proof}
\begin{remark}
Let the assumption be as in Lemma \ref{top}. 
The composite of the map $W \to H_{\mathrm{c}}^i(\X_{\C},\overline{\mathbb{Q}}_{\ell})$ in Lemma \ref{top} and the canonical 
map 
$H_{\mathrm{c}}^i(\X_{\C},\overline{\mathbb{Q}}_{\ell}) \to H^i(\X_{\C},\overline{\mathbb{Q}}_{\ell})$
is also injective. 
\end{remark}
\begin{remark}
One could probably rewrite results in this subsection in 
Berkovich's language in 
\cite{Be} and \cite{Be2} through 
comparison theorems in 
\cite[\S 8.3]{Huet}.   
\end{remark}
\begin{remark}
An open unit polydisk 
in rigid geometry 
is taut and separated. 
Note that a finite morphism 
is taut and separated, and tautness and 
separatedness are stable under composition.  
Since a Lubin-Tate space is a finite \'etale covering of 
an open unit polydisk, it  
is taut and separated. 
\end{remark}
\begin{remark}
We note that 
the reductions of formal models of affinoids in 
the Lubin-Tate perfectoid space in \cite{IT3}
and \cite{IT4} also satisfy the property in Lemma 
\ref{top}.2. Recently, in \cite{To}, 
Tokimoto generalizes 
\cite{IT3}. 
The reductions of formal models of affinoids in 
the Lubin-Tate perfectoid space in \cite{To}
  also satisfy it. 
  In a subsequent paper, 
  we will study corresponding affinoids 
  in the Lubin-Tate space and prove 
  the NALT for ramified essentially 
  tame representations in some case.  
\end{remark}
\subsection{Review on $\ell$-adic cohomology
of Artin-Schreier curves}\label{As}
Let $p$ be a prime number, and 
let $q$ be a power of $p$. 
For a finite abelian group $A$, 
we write $A^{\vee}$ for $\mathrm{Hom}_{\mathbb{Z}}\bigl(A,\overline{\mathbb{Q}}_{\ell}^{\times}\bigr)$. 

Let $\mathbb{A}^1$ be an affine line over $\mathbb{F}_q$. 
For $\psi \in \mathbb{F}_q^{\vee}$, 
let $\mathscr{L}_{\psi}$ denote the smooth 
$\overline{\mathbb{Q}}_{\ell}$-sheaf of rank one on 
$\mathbb{A}^1$ defined by 
the Artin-Schreier covering $a^q-a=x$
and $\psi$. 
Note that 
$\mathscr{L}_{\psi}$ is equal to $\mathfrak{F}(\psi)$ in the notation of \cite[ 1.8(i) in Sommes trig.]{DelCoet}. 
For a variety $Y$ over $\mathbb{F}_q$ and a function $f \colon Y \to 
\mathbb{A}^1$, 
let $\mathscr{L}_{\psi}(f)$ denote the 
pull-back $f^\ast \mathscr{L}_{\psi}$ to 
$Y$. 

We set 
$\mathbb{G}_m=\mathbb{A}^1 \setminus \{0\}$. 
Let $n$ be a positive integer which is  
prime to $p$. 
Let $m$ be a positive integer such that 
$\bom{\mu}_n(\mathbb{F}) \subset \mathbb{F}_{q^m}^{\times}$. We simply write $\bom{\mu}_n$ for 
$\bom{\mu}_n(\mathbb{F})$. 
For $c_0 \in \mathbb{F}_{q^m}^{\times}$ and $\chi \in \bom{\mu}_n^{\vee}$, 
let $\mathscr{K}_{\chi,c_0}$ be the smooth 
$\overline{\mathbb{Q}}_{\ell}$-sheaf  of rank one on 
$\mathbb{G}_m$ defined by the Kummer torsor 
$K_{n,c_0}=\mathbb{G}_m \to \mathbb{G}_m;\ y \mapsto 
c_0y^n$
and $\chi$. Note that 
$\mathscr{K}_{\chi,c_0}$ equals 
$\chi^{-1}(K_{n,c_0})$ 
in the notation of \cite[1.2 in 
Sommes trig.]{DelCoet}. 

We consider the Gauss sum 
\begin{equation}\label{mnc}
G_{m,n,c_0}(\chi,\psi)=-\sum_{x \in \mathbb{F}_{q^m}^{\times}}
\chi\left((x/c_0)^{\frac{q^m-1}{n}}\right)\psi\left(\Tr_{\mathbb{F}_{q^m}/\mathbb{F}_q}(x)\right). 
\end{equation}
\begin{lemma}\label{ll11}
{\rm 1}.\ We have 
$H_{\mathrm{c}}^i(\mathbb{A}_{\mathbb{F}}^1,\mathscr{L}_{\psi}(c_0y^n))=0$ except for $i=1$, and 
an isomorphism 
\begin{equation}\label{bob}
H_{\mathrm{c}}^1(\mathbb{A}_{\mathbb{F}}^1,\mathscr{L}_{\psi}(c_0y^n))
\simeq 
\bigoplus_{\chi \in \bom{\mu}_n^{\vee} \setminus \{1\}} 
H_{\mathrm{c}}^1(\mathbb{G}_{m,\mathbb{F}},\mathscr{L}_{\psi}
\otimes \mathscr{K}_{\chi,c_0}). 
\end{equation}
Furthermore, we have 
$\dim_{\overline{\mathbb{Q}}_{\ell}} H_{\mathrm{c}}^1(\mathbb{G}_{m,\mathbb{F}},\mathscr{L}_{\psi}
\otimes \mathscr{K}_{\chi,c_0})=1$. 
The geometric Frobenius element 
over $\mathbb{F}_{q^m}$
acts on 
$H_{\mathrm{c}}^1(\mathbb{G}_{m,\mathbb{F}},\mathscr{L}_{\psi}
\otimes \mathscr{K}_{\chi,c_0})$
as multiplication by $G_{m,n,c_0}(\chi,\psi)$. 
\\
{\rm 2}.\ The canonical map 
$H_{\mathrm{c}}^1(\mathbb{A}_{\mathbb{F}}^1,\mathscr{L}_{\psi}(c_0y^n)) \to 
H^1(\mathbb{A}_{\mathbb{F}}^1,
\mathscr{L}_{\psi}(c_0y^n)) 
$ is an isomorphism. 
\end{lemma}
\begin{proof}
By \cite[Remarques 1.8 b), c) in Sommes trig.]{DelCoet}, we have 
\[
H_{\rm c}^i(\mathbb{A}_{\mathbb{F}}^1,\mathscr{L}_{\psi}(c_0y^n))=0
\] 
for $i=0,2$. 
We prove \eqref{bob}.  
Let $f \colon \mathbb{A}^1 \to \mathbb{A}^1$
be the morphism defined by 
$y \mapsto c_0y^n$. 
Since $f$ is \'{e}tale over $\mathbb{G}_m$, 
by the projection formula, 
we have a decomposition 
\begin{equation}\label{pp1}
f_\ast \mathscr{L}_{\psi}(c_0y^n) \simeq 
\mathscr{L}_{\psi} \otimes f_\ast \overline{\mathbb{Q}}_{\ell} \simeq 
\bigoplus_{\chi \in \bom{\mu}_n^{\vee}} \mathscr{L}_{\psi}
\otimes \mathscr{K}_{\chi,c_0} \quad \textrm{on $\mathbb{G}_m$}. 
\end{equation}
We have short exact sequences 
\begin{gather}\label{pp2}
\begin{aligned}
0 & \to 
H^0(\{0\}_{\mathbb{F}},\overline{\mathbb{Q}}_{\ell})\simeq \overline{\mathbb{Q}}_{\ell} 
\to 
H_{\rm c}^1(\mathbb{G}_{m, \mathbb{F}},
\mathscr{L}_{\psi}(c_0y^n))
\to 
H_{\rm c}^1(\mathbb{A}^1_{\mathbb{F}},
\mathscr{L}_{\psi}(c_0y^n)) \to 
 0, \\
0 & \to 
H^0(\{0\}_{\mathbb{F}},\overline{\mathbb{Q}}_{\ell})\simeq \overline{\mathbb{Q}}_{\ell} 
 \to 
H_{\rm c}^1(\mathbb{G}_{m, \mathbb{F}},
\mathscr{L}_{\psi}) 
\to H_{\rm c}^1(\mathbb{A}^1_{\mathbb{F}},
\mathscr{L}_{\psi})=0 \to 0. 
\end{aligned}
\end{gather}
Hence, the required assertion for the cohomology with  compact support follows from 
\eqref{pp1} and \eqref{pp2}. 
We prove the second and the third 
assertions in $1$. 
These follow from the 
Grothendieck-Ogg-Shafarevich formula in 
\cite[Th\'eor\`em 7.1 in Expos\'e X]{SGA5}
and the 
Grothendieck trace formula respectively 
(cf.\ \cite[Proposition 4.3 in Sommes trig.]{DelCoet}). 

We prove $2$.  
In the same way as above,  
we have 
an isomorphism 
\[
H^1(\mathbb{A}_{\mathbb{F}}^1,\mathscr{L}_{\psi}(c_0y^n))
\simeq 
\bigoplus_{\chi \in \bom{\mu}_n^{\vee} \setminus \{1\}} 
H^1(\mathbb{G}_{m,\mathbb{F}},\mathscr{L}_{\psi}
\otimes \mathscr{K}_{\chi,c_0}),  
\]
and the commutative diagram
\begin{equation}\label{alb}
\xymatrix{
H_{\mathrm{c}}^1(\mathbb{A}_{\mathbb{F}}^1,\mathscr{L}_{\psi}(c_0y^n))
\ar[d]^{\mathrm{can}.}\ar[r]^{\!\!\!\!\!\!\!\!\!\!\!\!\!\!\!\!\!\!\!\!\!\!\!\!\!\!\simeq} &
\displaystyle\bigoplus_{\chi \in \bom{\mu}_n^{\vee} \setminus \{1\}} 
H_{\mathrm{c}}^1(\mathbb{G}_{m,\mathbb{F}},\mathscr{L}_{\psi}
\otimes \mathscr{K}_{\chi,c_0}) \ar[d]^{\mathrm{can}.}\\
 H^1(\mathbb{A}_{\mathbb{F}}^1,\mathscr{L}_{\psi}(c_0y^n))
\ar[r]^{\!\!\!\!\!\!\!\!\!\!\!\!\!\!\!\!\!\!\!\!\!\!\!\!\!\!\!\simeq} &
\displaystyle\bigoplus_{\chi \in \bom{\mu}_n^{\vee} \setminus \{1\}} 
H^1(\mathbb{G}_{m,\mathbb{F}},\mathscr{L}_{\psi}
\otimes \mathscr{K}_{\chi,c_0}).  
}
\end{equation}
By \cite[Proposition 4.3 in Sommes trig.]{DelCoet}, 
for any $\chi \in \bom{\mu}_n^{\vee} \setminus \{1\}$,  
the canonical map
\begin{equation}\label{sga}
H_{\mathrm{c}}^1(\mathbb{G}_{m,\mathbb{F}},\mathscr{L}_{\psi}
\otimes \mathscr{K}_{\chi,c_0})
\to 
H^1(\mathbb{G}_{m,\mathbb{F}},\mathscr{L}_{\psi}
\otimes \mathscr{K}_{\chi,c_0})
\end{equation}
is an isomorphism. 
Hence, by \eqref{alb} and \eqref{sga}, 
the canonical map 
$H_{\mathrm{c}}^1(\mathbb{A}_{\mathbb{F}}^1,\mathscr{L}_{\psi}(c_0y^n))
\to H^1(\mathbb{A}_{\mathbb{F}}^1,\mathscr{L}_{\psi}(c_0y^n))$
is an isomorphism.
Therefore, the required assertion $2$ follows.  
\end{proof}
\begin{corollary}\label{lc1}
Let $X_{m,n, c_0}$ be the affine smooth curve defined by
$a^q-a=c_0 y^n$ 
over $\mathbb{F}_{q^m}$. \\
{\rm 1}.\ We have an isomorphism
\[
H_{\rm c}^1(X_{m,n,c_0, \mathbb{F}},\overline{\mathbb{Q}}_{\ell})
\simeq \bigoplus_{\psi \in \mathbb{F}_q^{\vee} \setminus 
\{1\},\ \chi \in \bom{\mu}_n^{\vee} \setminus \{1\}}
H_{\rm c}^1(\mathbb{G}_{m,\mathbb{F}}, \mathscr{L}_{\psi}\otimes \mathscr{K}_{\chi,c_0}). 
\]
Furthermore, we have 
$\dim_{\overline{\mathbb{Q}}_{\ell}} H_{\rm c}^1(X_{m,n,c_0, \mathbb{F}},\overline{\mathbb{Q}}_{\ell})=(q-1)(n-1)$. \\
{\rm 2}.\ The canonical map
$H_{\rm c}^1(X_{m,n,c_0, \mathbb{F}},\overline{\mathbb{Q}}_{\ell})
\to 
H^1(X_{m,n,c_0, \mathbb{F}},\overline{\mathbb{Q}}_{\ell})
$
is an isomorphism. 
\end{corollary}
\begin{proof}
We have an isomorphism 
\begin{equation}\label{fora}
H_{\rm c}^1(X_{m,n,c_0, \mathbb{F}},\overline{\mathbb{Q}}_{\ell})
\simeq 
\bigoplus_{\psi \in \mathbb{F}_q^{\vee} \setminus \{1\}}
H_{\rm c}^1(\mathbb{A}^1_{\mathbb{F}},\mathscr{L}_{\psi}(c_0 y^n)). 
\end{equation}
Hence, the required assertions follow
from Lemma \ref{ll11}.  
\end{proof}
We consider the case $(m,n,c_0)=(1,2,1)$ in the notation 
of Corollary \ref{lc1}. 
\begin{lemma}\label{sign}
We consider the automorphism 
$i \colon X_{1,2,1} \to X_{1,2,1}$ defined by 
$(a,y) \mapsto (a,-y)$. 
Then, $i$ acts on 
$H_{\mathrm{c}}^1(X_{1,2,1, \mathbb{F}},\overline{\mathbb{Q}}_{\ell})$ 
as scalar multiplication by $-1$. 
\end{lemma}
\begin{proof}
Let $\psi \in \mathbb{F}_q^{\vee} \setminus \{1\}$. 
By Lemma \ref{ll11}.1, 
we have 
\[
\dim_{\overline{\mathbb{Q}}_{\ell}}H_{\mathrm{c}}^1(\mathbb{A}^1_{\mathbb{F}},\mathscr{L}_{\psi}(y^2))=1.
\] 
Hence, by  
the Grothendieck 
trace formula, 
$i$ acts on $H_{\mathrm{c}}^1(\mathbb{A}^1_{\mathbb{F}},\mathscr{L}_{\psi}(y^2))$
as scalar multiplication by $-1$.    
Therefore, the required assertion follows from 
\eqref{fora} for $(m,n,c_0)=(1,2,1)$. 
\end{proof}
It is easy to directly 
calculate the Gauss sum 
in some special cases.
\begin{lemma}\label{g}
Let $\zeta_1 \in \mathbb{F}_{q^2}^{\times}$ satisfying 
$\zeta_1^{q-1}=-1$. 
Then, we have $G_{2,q+1,\zeta_1}(\chi,\psi)=-q$. 
\end{lemma}
\begin{proof}
We have equalities 
\begin{gather}\label{gauss}
\begin{aligned}
G_{2,q+1,\zeta_1}(\chi,\psi)& =-\sum_{x^{q-1}=-1} \chi\left((x/\zeta_1)^{q-1}\right)
-\sum_{\mu \in \mathbb{F}_q^{\times}} \sum_{x^q+x=\mu}
\chi\left((x/\zeta_1)^{q-1}\right) \psi(\mu)\\
&=-(q-1)-\sum_{\mu \in \mathbb{F}_q^{\times}} \sum_{y^q+y=1}
\chi\left(-y^{q-1}\right) \psi(\mu),  
\end{aligned}
\end{gather}
where we change a variable $y=x/\mu$ 
and use $\zeta_1^{q-1}=-1$ at the second equality. The map $f \colon \mathbb{F}_{q^2}^{\times} \to
\bom{\mu}_{q+1};\ 
x \mapsto x^{q-1}$ is injective on the subset 
$S=\{y \in  \mathbb{F}_{q^2}^{\times} \mid y^q+y=1\}$.
Since $S$ consists of $q$ elements satisfying  
$f(S) \cap 
\{-1\}=\emptyset$, we obtain 
$f(S)=\bom{\mu}_{q+1} \setminus \{-1\}$. Hence, 
we have $\sum_{y \in S}
\chi(-y^{q-1})=-1$. 
Since we have $\sum_{\mu \in \mathbb{F}_q^{\times}}\psi(\mu)=-1$, 
we obtain the required assertion by \eqref{gauss}.
\end{proof}
We consider the affine curve $X_0$ 
defined by $X^{q^2}-X=Y^{q(q+1)}-Y^{q+1}$
over $\mathbb{F}_{q^2}$.
Let 
\[
Q=\left\{g(\alpha,\beta,\gamma) 
=\begin{pmatrix}
\alpha & \beta & \gamma \\
& \alpha^q & \beta^q \\
&& \alpha
\end{pmatrix} \in \mathrm{GL}_3(\mathbb{F}_{q^2})
\right\},  
\]
which is a subgroup of $\mathrm{GL}_3(\mathbb{F}_{q^2})$. This group appears also in
\cite[the proof of Proposition 4.3.4]{WeJL}. 
We identify the center 
$Z=\{g(1,0, \gamma) \mid \gamma \in 
\mathbb{F}_{q^2}
\} \subset Q$
with $\mathbb{F}_{q^2}$ by $g(1,0,\gamma) \mapsto
\gamma$. 
Let $Q$ act on $X_0$ by 
\begin{gather}\label{act}
\begin{aligned}
g(\alpha,\beta,\gamma) \colon 
 X_0 \to X_0;\ 
 (X,Y) \mapsto 
\left(X+\frac{\beta^q}{\alpha}Y+\frac{\gamma}{\alpha},\alpha^{q-1}Y+\frac{\beta}{\alpha}\right). 
\end{aligned}
\end{gather}
We regard 
$\mathbb{F}_q^{\times}$ as a normal 
subgroup 
of $Q$ by 
$\alpha \mapsto g(\alpha,0,0)$ for 
$\alpha \in \mathbb{F}_q^{\times}$.
By \eqref{act}, the group $Q$ acts on $X_0$ factoring through 
$Q \to Q/\mathbb{F}_q^{\times}$. 
We regard $\mathbb{F}_q^{\vee}$ as a subset 
of $\mathbb{F}_{q^2}^{\vee}$
via $\Tr_{\mathbb{F}_{q^2}/\mathbb{F}_q}^{\vee}$. 
We simply write $\mathcal{C}$ for $\mathbb{F}_{q^2}^{\vee} \setminus 
\mathbb{F}_q^{\vee}$. 
Let $Q_0$ be the subgroup of $Q$
consisting of all elements of the form 
$g(1,\beta,\gamma)$. 
The subgroup $Z$ is the center of $Q_0$, and the 
quotient $V=Q_0/Z$ is isomorphic to 
$\mathbb{F}_{q^2}$. 
Therefore, $Q_0$ is a finite Heisenberg group. 
The following lemma is a well-known fact 
in a representation theory of finite groups. 
\begin{lemma}\label{fcf}
Let $\psi \in \mathcal{C}$. 
There exists a unique representation 
$\tau^0_{\psi}$ of 
$Q$ such that 
\begin{equation}\label{ffc}
{\tau^0_{\psi}}|_{\mathbb{F}_q^{\times}}=\bom{1}^{\oplus q}, \quad 
\tau^0_{\psi}|_Z \simeq 
\psi^{\oplus q}, \quad 
\Tr \tau^0_{\psi} (g(\alpha,0,0))=-1\ \textrm{for any 
$\alpha \in \mathbb{F}_{q^2} \setminus \mathbb{F}_q$}, 
\end{equation}
where $\bom{1}$ denotes the trivial representation of 
$Q$. 
Furthermore, $\tau_{\psi}^0$ is irreducible. 
\end{lemma}
\begin{proof}
Let $\bom{\mu}_{q+1}$ act on $Q_0$ 
by 
\[
\zeta \colon 
g(1,\beta,\gamma) \mapsto 
g\left(\widetilde{\zeta},0,0\right)^{-1} g(1,\beta,\gamma)
g\left(\widetilde{\zeta},0,0\right)\quad  \textrm{for 
$\zeta \in \bom{\mu}_{q+1}$ and $g(1,\beta,\gamma)
\in Q_0$}, 
\]
where $\widetilde{\zeta} \in \mathbb{F}_{q^2}^{\times}$
is an element such that $\widetilde{\zeta}^{q-1}=\zeta$.
The center $Z$ is fixed by 
this action. For each non-trivial element 
$a \in \bom{\mu}_{q+1}$, 
then $a$ has only the trivial fixed point in $V$. 
We have an isomorphism 
$Q/\mathbb{F}_{q}^{\times} \simeq \bom{\mu}_{q+1} \ltimes Q_0$. 
Let $\psi \in \mathcal{C}$. 
By taking $(\bom{\mu}_{q+1}, Q_0/\ker \psi)$
as $(A,G)$ in the notation of \cite[\S22.1]{BH} and 
applying 
 \cite[Lemma in \S22.2]{BH} to this situation, 
 we obtain the 
 $\bom{\mu}_{q+1} \ltimes \left(Q_0/\ker \psi\right)$-representation $\eta_1$ in the notation of 
  \cite[Lemma in \S22.2]{BH}. 
  The inflation of $\eta_1$ to $Q$
  by the composite 
  \[
  Q \to Q/\mathbb{F}_q^{\times} \simeq 
  \bom{\mu}_{q+1} \ltimes Q_0 \to \bom{\mu}_{q+1} \ltimes \left(Q_0/\ker \psi\right)
  \] 
  satisfies \eqref{ffc} by \cite[Lemma 2 (1) in \S16.4]{BH}. 
  The uniqueness and irreducibility of $\tau_{\psi}^0$ follows again from \cite[Lemma 2 in \S16.4]{BH}. 
\end{proof}
\begin{remark}
The principle in 
\cite[Lemma 2 in \S 16.4]{BH} and 
\cite[Lemma in \S22.2]{BH} plays an important 
role in the theory of types for 
$\mathrm{GL}(2)$ (cf.\ \cite[Proposition 19.4 and 
\S22.4]{BH} or \S \ref{exBH}). 
\end{remark}
\begin{lemma}\label{hei}
{\rm 1}.\ 
We have an isomorphism 
\begin{equation}\label{tiso}
H_{\rm c}^1(X_{0,\mathbb{F}},\overline{\mathbb{Q}}_{\ell}) \simeq 
\bigoplus_{\psi \in \mathcal{C}} \tau_{\psi}^0
\end{equation}
as $Q$-representations. Furthermore, we have 
$\dim_{\overline{\mathbb{Q}}_{\ell}} H_{\rm c}^1(X_{0,\mathbb{F}},\overline{\mathbb{Q}}_{\ell})=q^2(q-1)$. 
\\
{\rm 2}.\ The geometric Frobenius 
element over $\mathbb{F}_{q^2}$ 
acts on $H_{\rm c}^1(X_{0,\mathbb{F}},\overline{\mathbb{Q}}_{\ell})$
as scalar multiplication by $-q$. \\
{\rm 3}.\ The canonical 
map $H_{\rm c}^1(X_{0,\mathbb{F}},\overline{\mathbb{Q}}_{\ell})
\to 
H^1(X_{0,\mathbb{F}},\overline{\mathbb{Q}}_{\ell})$ 
is an isomorphism. 
\end{lemma}
\begin{proof}
We prove the first assertion. 
The curve $X_0$ is isomorphic to  
$\coprod_{\mu \in \mathbb{F}_q}Y_{\mu}$ over 
$\mathbb{F}_{q^2}$, where  $Y_{\mu}$ is defined by 
$X^q+X=Y^{q+1}+\mu$. 
Let $\zeta_1, \xi \in \mathbb{F}_{q^2}^{\times}$ 
be elements 
such that $\zeta_1^{q-1}=-1$ and
$\xi^q+\xi=\mu$ respectively. 
By setting $a=-\zeta_1 (X-\xi)$, the curve 
$Y_0$ is defined by $a^q-a=\zeta_1 Y^{q+1}$. 
The stabilizer of $Y_0$ in $Z \simeq \mathbb{F}_{q^2}$
equals $\zeta_1 \mathbb{F}_q$.
Hence, 
we have 
an isomorphism 
\begin{equation}\label{tom}
H_{\rm c}^1(X_{0,\mathbb{F}},\overline{\mathbb{Q}}_{\ell}) \simeq \Ind_{\zeta_1 \mathbb{F}_q}^{\mathbb{F}_{q^2}} H_{\rm c}^1(Y_{0,\mathbb{F}},\overline{\mathbb{Q}}_{\ell}) 
\end{equation}
as $\mathbb{F}_{q^2}$-representations. 
By Corollary \ref{lc1}.1, 
we have an isomorphism 
\[
H_{\rm c}^1(Y_{0,\mathbb{F}},\overline{\mathbb{Q}}_{\ell}) \simeq \bigoplus_{\psi_0 \in (\zeta_1\mathbb{F}_q)^{\vee} \setminus \{1\}}\psi^{\oplus q}
\]
as $\zeta_1 \mathbb{F}_q$-representations. 
For $\psi \in \mathbb{F}_{q^2}^{\vee}$, we have 
$
\psi|_{\zeta_1 \mathbb{F}_q} \neq 1 \iff 
\psi \in \mathcal{C}$. 
Hence,  
we have an isomorphism $H_{\rm c}^1(X_{0,\mathbb{F}},\overline{\mathbb{Q}}_{\ell})
\simeq \bigoplus_{\psi \in \mathcal{C}}\psi^{\oplus q}$
 as $\mathbb{F}_{q^2}$-representations.  
Let $H_{\rm c}^1(X_{0,\mathbb{F}},\overline{\mathbb{Q}}_{\ell})_{\psi}$
denote the $\psi$-part, on which $\mathbb{F}_{q^2}$
acts by $\psi$. 
We consider the subgroup 
$A=\mathbb{F}_{q^2}^{\times}/\mathbb{F}_q^{\times} \subset Q/\mathbb{F}_q^{\times}$.
By Corollary \ref{lc1}.1, 
as $A$-representations, we have 
an isomorphism 
$H_{\rm c}^1(X_{0,\mathbb{F}},\overline{\mathbb{Q}}_{\ell})_{\psi} \simeq \bigoplus_{\chi \in A^{\vee} \setminus \{1\}} \chi$. 
Since $Z$ is the center of $Q$, 
we can regard $H_{\rm c}^1(X_{0,\mathbb{F}},\overline{\mathbb{Q}}_{\ell})_{\psi}
$ as a $Q$-representation. 
The $Q$-representation $H_{\rm c}^1(X_{0,\mathbb{F}},\overline{\mathbb{Q}}_{\ell})_{\psi}
$ satisfies \eqref{ffc}. 
Hence, this is isomorphic to
$\tau_{\psi}^0$. 
Therefore, \eqref{tiso} follows. 
The latter assertion follows from 
$\dim_{\overline{\mathbb{Q}}_{\ell}}\tau_{\psi}^0=q$
and $|\mathcal{C}|=q(q-1)$. 

The second assertion follows from Lemma
\ref{ll11}.1, Corollary \ref{lc1}.1, Lemma \ref{g} and \eqref{tom}.  
The third assertion follows from 
Corollary \ref{lc1}.2. 
\end{proof}
\begin{remark}
A similar analysis to the one in this subsection 
is found in 
\cite[\S 7]{WeLT} or 
\cite[\S 7.1]{IT}.  
\end{remark}
\subsection{Review on $\ell$-adic cohomology of Deligne-Lusztig curve for $\mathrm{GL}_2(\mathbb{F}_q)$}\label{As2}
We keep the same notation in 
\S \ref{As}. 
We take an element $\zeta_1 \in \mathbb{F}_{q^2}^{\times}$ such that $\zeta_1^{q-1}=-1$. 
Let $W$ be the affine smooth curve over 
$\mathbb{F}_{q^2}$  
defined by $S^qT-ST^q=\zeta_1$, 
which is called the Drinfeld curve (cf.\ \cite[p.117]{DL}). 
The isomorphism class of $W$ over $\mathbb{F}_{q^2}$ does not depend on the choice of $\zeta_1$. 
Let $\overline{W}$ be the smooth compactification of $W$, which is defined by $X^qY-XY^q=\zeta_1 Z^{q+1}$ in $\mathbb{P}_{\mathbb{F}_{q^2}}^2$.  
We consider the open immersion 
$j \colon W \hookrightarrow \overline{W};\ 
(S,T) \mapsto (S:T:1)$. 
Let $\bom{\mu}_{q+1}$
act on $\overline{W}$ by 
$[X:Y:Z] \mapsto [\xi X:\xi Y:Z]$
for $\xi \in \bom{\mu}_{q+1}$.
The open subscheme $W$ is stable under 
this $\bom{\mu}_{q+1}$-action. 
The closed subscheme 
$D_W=\overline{W} \setminus W$
with reduced scheme structure consists of $q+1$ closed points, 
and it is fixed by the action of 
$\bom{\mu}_{q+1}$. 
Let $\bom{1}$ be the trivial character of $\bom{\mu}_{q+1}$ valued in $\overline{\mathbb{Q}}_{\ell}$. 
We have a $\bom{\mu}_{q+1}$-equivariant 
short exact sequence 
\[
0 \to \bom{1}^{\oplus q}
\to H_{\rm c}^1(W_{\mathbb{F}},\overline{\mathbb{Q}}_{\ell})
\to H^1(\overline{W}_{\mathbb{F}},\overline{\mathbb{Q}}_{\ell}) \to 0.  
\]
By this and the Riemann-Hurwitz formula, 
we have 
\begin{gather}\label{d}
\begin{aligned}
\dim_{\overline{\mathbb{Q}}_{\ell}}H^1(\overline{W}_{\mathbb{F}},\overline{\mathbb{Q}}_{\ell})&=q(q-1), \\
\dim_{\overline{\mathbb{Q}}_{\ell}}H_{\rm c}^1({W}_{\mathbb{F}},\overline{\mathbb{Q}}_{\ell})&=q^2. 
\end{aligned}
\end{gather}
Let 
$i \colon D_W \hookrightarrow \overline{W}$ 
be the closed immersion. 
By the distinguished triangle 
$i_\ast Ri^!j_!\overline{\mathbb{Q}}_{\ell} \to
j_! \overline{\mathbb{Q}}_{\ell} \to 
Rj_{\ast} \overline{\mathbb{Q}}_{\ell}\xrightarrow{+1}$ on $\overline{W}$, 
we have a 
$\bom{\mu}_{q+1}$-equivariant long exact sequence 
\begin{gather}\label{d1}
\begin{aligned}
0 & \to \bom{1} \simeq H^0({W}_{\mathbb{F}},\overline{\mathbb{Q}}_{\ell}) 
\to H_{D_{W,\mathbb{F}}}^1(\overline{W}_{\mathbb{F}},j_!\overline{\mathbb{Q}}_{\ell})
\to 
H_{\rm c}^1(W_{\mathbb{F}},\overline{\mathbb{Q}}_{\ell}) \\
& \xrightarrow{\rm can.} 
 H^1(W_{\mathbb{F}},\overline{\mathbb{Q}}_{\ell})
\to H_{D_{W,\mathbb{F}}}^2(\overline{W}_{\mathbb{F}},j_!\overline{\mathbb{Q}}_{\ell}) \to H_{\rm c}^2(W_{\mathbb{F}},\overline{\mathbb{Q}}_{\ell}) \simeq
\bom{1}(-1) \to 0. 
\end{aligned}
\end{gather}
By the distinguished triangle 
$j_! \overline{\mathbb{Q}}_{\ell} \to \overline{\mathbb{Q}}_{\ell} \to
i_{\ast} \overline{\mathbb{Q}}_{\ell} \xrightarrow{+1}$
on $\overline{W}$, 
we have a $\bom{\mu}_{q+1}$-equivariant long exact sequence
\begin{align*}
0 & \to H^0(D_{W,\mathbb{F}},\overline{\mathbb{Q}}_{\ell})
\to H_{D_{W,\mathbb{F}}}^1(\overline{W}_{\mathbb{F}},j_!\overline{\mathbb{Q}}_{\ell}) \to 
H_{D_{W,\mathbb{F}}}^1(\overline{W}_{\mathbb{F}}, \overline{\mathbb{Q}}_{\ell}) \\
& \to H^1(D_{W,\mathbb{F}},\overline{\mathbb{Q}}_{\ell})=0 
\to H_{D_{W,\mathbb{F}}}^2(\overline{W}_{\mathbb{F}},j_!\overline{\mathbb{Q}}_{\ell}) \to 
H_{D_{W,\mathbb{F}}}^2(\overline{W}_{\mathbb{F}},\overline{\mathbb{Q}}_{\ell}) \to 0. 
\end{align*}
Since we have 
$H_{D_{W,\mathbb{F}}}^1(\overline{W}_{\mathbb{F}}, \overline{\mathbb{Q}}_{\ell})=0$ 
by the purity theorem, we obtain isomorphisms 
\begin{gather}\label{d2}
\begin{aligned}  
H_{D_{W,\mathbb{F}}}^1(\overline{W}_{\mathbb{F}},j_!\overline{\mathbb{Q}}_{\ell}) &\xleftarrow{\sim} 
H^0(D_{W,\mathbb{F}},\overline{\mathbb{Q}}_{\ell}) 
\simeq \bom{1}^{\oplus (q+1)},  
\\ H_{D_{W,\mathbb{F}}}^2(\overline{W}_{\mathbb{F}},j_!\overline{\mathbb{Q}}_{\ell}) 
& \xrightarrow{\sim} 
H_{D_{W,\mathbb{F}}}^2(\overline{W}_{\mathbb{F}},\overline{\mathbb{Q}}_{\ell})
\simeq \bom{1}(-1)^{\oplus (q+1)} 
\end{aligned}
\end{gather}
as $\bom{\mu}_{q+1}$-representations again
by the purity theorem. 
By \eqref{d1} and \eqref{d2}, we obtain 
a $\bom{\mu}_{q+1}$-equivariant long exact sequence
\begin{equation}\label{d3}
0 \to \bom{1}^{\oplus q}\to 
H_{\rm c}^1(W_{\mathbb{F}},\overline{\mathbb{Q}}_{\ell})  \to 
 H^1(W_{\mathbb{F}},\overline{\mathbb{Q}}_{\ell})
 \to 
\bom{1}(-1)^{\oplus q} \to 0. 
\end{equation}
Let $X_{\rm DL}$ be  
the affine curve over $\mathbb{F}_q$
defined by 
$\left(S^qT-ST^q\right)^{q-1}=-1$, 
which 
is called the Deligne-Lusztig 
curve for $\mathrm{GL}_2(\mathbb{F}_q)$. 
Let $\mathbb{F}_{q^2}^{\times}$
act on $X_{\rm DL}$ by $\xi 
\colon (S,T) \mapsto (\xi^{-1} S,\xi^{-1} T)$
for $\xi \in \mathbb{F}_{q^2}^{\times}$. 
Let $\mathrm{GL}_2(\mathbb{F}_q)$
act on $X_{\rm DL}$ by 
$g \colon (S,T) \mapsto 
(aS+cT,bS+dT)$ for $g=
\begin{pmatrix}
a & b \\
c & d
\end{pmatrix} \in \mathrm{GL}_2(\mathbb{F}_q)
$. 
Clearly, the actions of $\mathbb{F}_{q^2}^{\times}$
and $\mathrm{GL}_2(\mathbb{F}_q)$ 
commute. 
The curve $X_{\rm DL}$
is isomorphic to a disjoint union 
of $q-1$ copies of $W$ over $\mathbb{F}_{q^2}$. The stabilizer of each 
connected component of $X_{\rm DL}$ in $\mathbb{F}_{q^2}^{\times}$ equals $\bom{\mu}_{q+1}$. 
Hence, by inducing \eqref{d3} from $\bom{\mu}_{q+1}$
to $\mathbb{F}_{q^2}^{\times}$, 
we acquire an $\mathbb{F}_{q^2}^{\times}$-equivariant 
long exact sequence 
\begin{equation}\label{d4}
0 \to \left(\Ind_{\bom{\mu}_{q+1}}^{\mathbb{F}_{q^2}^{\times}}\bom{1}\right)^{\oplus q}  \xrightarrow{\varphi} 
H_{\rm c}^1(X_{\mathrm{DL},\mathbb{F}},\overline{\mathbb{Q}}_{\ell})  \xrightarrow{\rm can.}
 H^1(X_{\mathrm{DL},\mathbb{F}},\overline{\mathbb{Q}}_{\ell})
 \to 
\left(\left(\Ind_{\bom{\mu}_{q+1}}^{\mathbb{F}_{q^2}^{\times}}\bom{1}\right)(-1)\right)^{\oplus q} \to 0.  
\end{equation}
Let $H_{\rm c}^1(X_{\mathrm{DL},\mathbb{F}},\overline{\mathbb{Q}}_{\ell})_{\rm cusp} \subset 
H_{\rm c}^1(X_{\mathrm{DL},\mathbb{F}},\overline{\mathbb{Q}}_{\ell})$ be the cuspidal part
regarded as a $\mathrm{GL}_2(\mathbb{F}_q)$-representation.
We say that a character $\chi \in 
(\mathbb{F}_{q^2}^{\times})^{\vee}$ is in general position 
if $\chi$ does not factor through $\Nr_{\mathbb{F}_{q^2}/\mathbb{F}_q} \colon 
\mathbb{F}_{q^2}^{\times} \to \mathbb{F}_q^{\times}$.
We write $C \subset (\mathbb{F}_{q^2}^{\times})^{\vee}$
for the set of all characters in general position. 
By the Deligne-Lusztig theory 
for $\mathrm{GL}_2(\mathbb{F}_q)$ in \cite[Theorem 6.2]{DL}, 
we have a decomposition 
\begin{equation}\label{rro}
H_{\rm c}^1(X_{\mathrm{DL},\mathbb{F}},\overline{\mathbb{Q}}_{\ell})_{\rm cusp} =
\bigoplus_{\chi \in C} H_{\rm c}^1(X_{\mathrm{DL},\mathbb{F}},\overline{\mathbb{Q}}_{\ell})_{\chi} 
\end{equation}
as $\mathrm{GL}_2(\mathbb{F}_q) 
\times \mathbb{F}_{q^2}$-representations, 
and  
 the $\chi$-part $H_{\rm c}^1(X_{\mathrm{DL},\mathbb{F}},\overline{\mathbb{Q}}_{\ell})_{\chi}$ is an 
 irreducible and cuspidal $\mathrm{GL}_2(\mathbb{F}_q)$-representation
(cf.\ \cite[\S6.4]{BH} and \cite[Corollary 6.9]{Yo}). 
\begin{lemma}\label{DL}
{\rm 1}.\ The canonical map 
$H_{\rm c}^1(X_{\mathrm{DL},\mathbb{F}},\overline{\mathbb{Q}}_{\ell})  \to
 H^1(X_{\mathrm{DL},\mathbb{F}},\overline{\mathbb{Q}}_{\ell})$
 is injective on the cuspidal part. \\
{\rm 2}.\ 
Let $\chi \in C$. 
Then, the geometric Frobenius element over 
$\mathbb{F}_{q^2}$ acts on 
$H_{\rm c}^1(X_{\mathrm{DL},\mathbb{F}},\overline{\mathbb{Q}}_{\ell})_{\chi}$ as scalar 
multiplication by $-q$. 
\end{lemma}
\begin{proof}
Since $H_{\rm c}^1(X_{\mathrm{DL},\mathbb{F}},\overline{\mathbb{Q}}_{\ell})_{\chi}$ is irreducible, 
the geometric Frobenius element over 
$\mathbb{F}_{q^2}$ acts on it 
as scalar multiplication by Schur's lemma. 
Hence, the second assertion follows from the Grothendieck trace formula. 

We prove the first assertion. 
By \cite[Theorem (1) in \S 6.4]{BH},
 any irreducible and cuspidal representation of 
$\mathrm{GL}_2(\mathbb{F}_q)$ is 
$(q-1)$-dimensional.
Clearly, we have $|C|=q(q-1)$. 
Hence, by \eqref{rro},  
we have 
\begin{equation}\label{d5}
\dim_{\overline{\mathbb{Q}}_{\ell}} H_{\rm c}^1(X_{\mathrm{DL},\mathbb{F}},\overline{\mathbb{Q}}_{\ell})_{\rm cusp} =q(q-1)^2. 
\end{equation}
Since 
$\Ind_{\bom{\mu}_{q+1}}^{\mathbb{F}^{\times}_{q^2}}\bom{1}$ contains no character of 
$\mathbb{F}_{q^2}^{\times}$ in general position, 
we have  
\begin{equation}\label{d6}
\Image \varphi \cap H_{\rm c}^1(X_{\mathrm{DL},\mathbb{F}},\overline{\mathbb{Q}}_{\ell})_{\rm cusp} =\{0\} 
\end{equation} 
by \eqref{rro}. 
By \eqref{d5} and \eqref{d6}, 
we obtain  
$H_{\rm c}^1(X_{\mathrm{DL},\mathbb{F}},\overline{\mathbb{Q}}_{\ell}) \simeq \Image\varphi \oplus H_{\rm c}^1(X_{\mathrm{DL},\mathbb{F}},\overline{\mathbb{Q}}_{\ell})_{\rm cusp}$. Hence, the required assertion follows from \eqref{d4}. 
\end{proof}
\begin{remark}
By setting 
$a=S/T$ and $t=1/T$, 
the curve ${X_{\rm DL}}$ is isomorphic to
the affine curve defined by 
$a^q-a=\zeta_1 t^{q+1}$ with 
$t \neq 0$. By using this fact, 
Lemma \ref{ll11}.1 and Lemma \ref{g}, 
we can deduce Lemma \ref{DL}.2. 
\end{remark}

\section{Statement of main theorem}\label{5}
In this section, we state our main theorem
in Theorem \ref{Mainc}.
This theorem is reduced to Proposition \ref{Main}. 
A proof of Proposition \ref{Main} 
will be given in \S \ref{66} in a purely local
manner. 
In 
\S \ref{exBH}, we give a summary of the theory 
of types for $\mathrm{GL}(2)$, and introduce a statement  
of the explicit LLC and LJLC. 
In \S \ref{mimp}, 
admitting Proposition \ref{Main}, 
we deduce Theorem \ref{Mainc}. 
To do so, 
we need to 
show Corollary \ref{stst}. 
To prove this, we need to 
understand some group action on $\pi_0$ 
of the Lubin-Tate tower. 
We can do this by using determinant morphisms given 
 in \cite[V.4]{FGL}. These are done in 
\S \ref{gc}.  
Furthermore, in \S \ref{MR}, by using 
Proposition \ref{Main} and Theorem \ref{Mainc}, 
we will show the equivalence of 
the explicit LLC and LJLC, and 
the NALT for $\mathrm{GL}(2)$
(cf.\ Corollaries \ref{MC} and \ref{NALexp}).

For a character $\chi$ of $L^{\times}$, 
we often identify it 
with a character of $W_L$ via the Artin reciprocity 
map $\bom{a}_L$. 
\subsection{Theory of types for $\mathrm{GL}(2)$}
\label{exBH}
For any admissible pair $(L/F,\chi)$, 
one can explicitly construct an irreducible cuspidal 
representation of $\mathrm{GL}_2(F)$, 
an irreducible smooth representation of
$D^{\times}$, and a two-dimensional 
irreducible smooth representation of $W_F$ so that 
they match under 
the LLC and the LJLC. 
In this subsection, 
in the odd residue characteristic case, 
we briefly recall 
the construction in 
\cite[\S19, \S34.1, \S 56]{BH} and
a main statement 
of explicit LLC and LJLC. 
Under this assumption, 
this theory is described 
with respect to admissible pairs. 
\subsubsection{Admissible pair}
We briefly recall 
admissible pairs from \cite[Definition in \S18.2]{BH}. 
We consider a pair $(L/F,\chi)$,
where 
$L/F$ is a tamely ramified quadratic field 
extension and $\chi$ is a smooth character of $L^{\times}$. 
\begin{definition}(\cite[Definition 18.2]{BH})
The pair $(L/F,\chi)$ is admissible if 
\begin{itemize}
\setlength{\itemsep}{-0.1cm}
\item $\chi$ does not factor through the norm
map $\Nr_{L/F} \colon L^{\times} \to F^{\times}$, and 
\item if $\chi|_{U_L^1}$ factors through 
$\Nr_{L/F}$, then $L/F$ is unramified.
\end{itemize}
Admissible pairs $(L/F,\chi)$ and 
$(L'/F,\chi')$ are said to be $F$-isomorphic 
if there is an $F$-isomorphism 
$\phi \colon L \xrightarrow{\sim} L'$ 
such that $\chi=\chi'\circ \phi$.
We write $\mathbb{P}_2(F)$ for the set of $F$-isomorphism classes of admissible 
pairs.   

The level of $\chi$ means 
the least integer $n \geq 0$
such that $\chi|_{U_L^{n+1}}$ is trivial, which is denoted by $l(\chi)$. 
We say that an admissible pair $(L/F,\chi)$
such that $l(\chi)=n$
is minimal if 
$\chi|_{U_L^n}$ does not factor through $\Nr_{L/F}$. 
\end{definition}
By the local class field theory, 
the level of $\chi$
equals the Swan conductor exponent of $\chi$
regarded 
as a character of $W_L$. 

For a character $\phi$ of $F^{\times}$, 
we write $\phi_L$ for the composite $\phi \circ
\Nr_{L/F}$. 
For any admissible pair $(L/F,\chi)$
is $F$-isomorphic to 
one of the form $(L/F,\chi' \otimes\phi_L)$,
with a character $\phi$ of $F^{\times}$
and a minimal admissible pair $(L/F,\chi')$. 
\subsubsection{Level and conductor}
For an irreducible smooth representation $\pi$
of $\mathrm{GL}_2(F)$, 
let $l(\pi) \in 2^{-1}\mathbb{Z}$ denote the 
normalized level of $\pi$ in \cite[\S12.6]{BH}. 
\begin{definition}
For an irreducible 
cuspidal representation $\pi$
of $\mathrm{GL}_2(F)$, 
we define a conductor of $\pi$
to be $2(l(\pi)+1) \in \mathbb{Z}$, 
which we denote by $c(\pi)$. 
\end{definition}

For an irreducible smooth representation  
$\rho$ of $D^{\times}$, 
let $m(\rho)$ be the largest integer $i$ such that 
$\rho|_{U_D^i}$ is non-trivial. 
If $\rho|_{\mathcal{O}_D^{\times}}$ is trivial, 
we set $m(\rho)=-1$. If $\dim \rho>1$, 
the integer $m(\rho)$ equals the level of $\rho$ in \cite[\S54.1]{BH}.  
\begin{definition}(\cite[p.\ 185]{Tu})
We define a conductor of $\rho$ to be 
$m(\rho)+2$, which we denote by $c(\rho)$.  
\end{definition}
\subsubsection{Unramified case}\label{U}
Let $n \geq 1$ be a positive integer. 
Let $(F_2/F,\chi)$ be a minimal admissible 
 pair such that $l(\chi)=n-1$. 
We take $F$-embeddings $F_2 \hookrightarrow \mathrm{M}_2(F)$ and $F_2 \hookrightarrow D$.
We set
\begin{align*}
J_{1,n}&=F_2^{\times} U_{\mathfrak{M}}^{[\frac{n}{2}]}
\subset \mathrm{GL}_2(F),\\
J_{2,n}&=F_2^{\times} U_D^{n-1} \subset D^{\times}, 
\end{align*}
with $U_{\mathfrak{M}}^0
=\mathrm{GL}_2(\mathcal{O}_F)$ and $U_D^0=\mathcal{O}_D^{\times}$.
Let $\psi_0 \in \mathbb{F}_q^{\vee} \setminus \{1\}$.
Let $\psi_F$ be a character of $F$
such that $\psi_F(x)=\psi_0(\bar{x})$ for $x \in \mathcal{O}_F$. 
For a finite extension $L/F$, 
let $\psi_L$ denote the composite 
$\psi_F \circ \Tr_{L/F}$. 
For $n \geq 2$,  let 
$\alpha \in \mathfrak{p}_{F_2}^{-(n-1)}$
be an element 
such that $\chi(1+x)=\psi_{F_2}(\alpha x)$
for $x \in \mathfrak{p}_{F_2}^{[(n-1)/2]+1}$. 
\paragraph{Level zero case}
First, we consider the case $n=1$. 
We naturally 
identify $U^0_{F_2}/U_{F_2}^1$ with $\mathbb{F}_{q^2}^{\times}$. 
Let $C$ be as in \S \ref{As2}. 
Then, $\chi$ is a tamely ramified 
character of $F_2^{\times}$ such that 
$\chi|_{U^0_{F_2}/U_{F_2}^1} \in C$. 
We write $\chi_0$ for $\chi|_{U^0_{F_2}/U_{F_2}^1}$. 
We take an $\mathbb{F}_q$-embedding 
$\mathbb{F}_{q^2} \hookrightarrow
\mathrm{M}_2(\mathbb{F}_q)$. 
There exists a unique irreducible cuspidal representation $\pi_{\chi_0}$ of 
$\mathrm{GL}_2(\mathbb{F}_q)$ 
which satisfies 
\begin{equation}\label{stein}
\Ind_{\mathbb{F}_{q^2}^{\times}}^{\mathrm{GL}_2(\mathbb{F}_q)} \chi_0
\simeq \overline{\mathrm{St}} \otimes 
\pi_{\chi_0}, 
\end{equation}
where $\overline{\mathrm{St}}$ 
is the Steinberg representation of 
$\mathrm{GL}_2(\mathbb{F}_q)$. 
Note that $\dim \pi_{\chi_0}=q-1$ and 
$\dim \overline{\mathrm{St}}=q$.  
The isomorphism class of the representation \eqref{stein} is independent of
the choice of the $\mathbb{F}_q$-embedding 
$\mathbb{F}_{q^2} \hookrightarrow
 \mathrm{M}_2(\mathbb{F}_q)$, 
 and depends only on
 $\chi_0$.  

Let $\La_{\chi}$ be 
the irreducible smooth 
representation of $J_{1,1}$ such that
\begin{itemize}
\item the restriction 
$\La_{\chi}|_{\mathrm{GL}_2(\mathcal{O}_F)}$ 
is isomorphic to 
the inflation of $\pi_{\chi_0}$ by 
$\mathrm{GL}_2(\mathcal{O}_F) \to 
\mathrm{GL}_2(\mathbb{F}_q)$, 
and 
\item $\La_{\chi}|_{F_2^{\times}}$ is a multiple of $\chi$. 
\end{itemize}
Let $\La'_{\chi}$ be 
the character of $J_{2,1}$ satisfying $\La'_{\chi}|_{U_D^1}=1$ and 
$\La'_{\chi}|_{F_2^{\times}}=\chi$. 
Let $\Delta_0$ denote the 
unramified character of $F_2^{\times}$
of order two.

\paragraph{Unramified case of positive level}
We consider the case $n \geq 2$. 
We define smooth representations 
$\La_{\chi}$
and $\La'_{\chi}$ of 
$J_{1,n}$ and $J_{2,n}$ respectively as follows. 
First, we define $\La_{\chi}$ (cf.\ \cite[\S19.4]{BH}). 
Assume that $n$ is even. 
We define $\La_{\chi}$ 
by 
\[
\La_{\chi}(x(1+y))=\chi(x) 
\psi_F(\Tr (\alpha y))\ 
\textrm{for $x \in F_2^{\times}$ and $1+y \in U_{\mathfrak{M}}^{[\frac{n}{2}]}$}. 
\]
Assume that $n$ is odd. We write $n=2m-1$. 
Let $\theta$ be 
the character of $U_{F_2}^1 U_{\mathfrak{M}}^m$
which satisfies
\[
\theta(x(1+y))=\chi(x) \psi_F (\Tr(\alpha y))\ 
\textrm{for $x \in U_{F_2}^1$ and $1+y \in 
U_{\mathfrak{M}}^m$}. 
\]
There exists a unique $q$-dimensional
irreducible representation 
$\eta_{\theta}$
 of $U_{F_2}^1 U_{\mathfrak{M}}^{m-1}$ 
 such that 
the restriction to $U_{F_2}^1 U_{\mathfrak{M}}^{m}$
is a multiple of $\theta$ by \cite[Lemma 15.6]{BH} 
(cf.\ \cite[\S16.4]{BH}). 
We define $\La_{\chi}$ to be the 
irreducible representation of $J_{1,n}$
such that 
\begin{itemize}
\item $\La_{\chi}|_{U_{F_2}^1 U_{\mathfrak{M}}^{m-1}}
=\eta_{\theta}$, 
\item $\La_{\chi}|_{F^{\times}}$ 
is a multiple of $\chi|_{F^{\times}}$, and 
\item $\Tr \La_{\chi}(\mu)=-\chi(\mu)$ for $\mu
\in \bom{\mu}_{q^2-1}(F_2) \setminus \bom{\mu}_{q-1}(F)$ 
\end{itemize}
as in \cite[Corollary 19.4]{BH} (cf.\ 
\cite[\S22.4]{BH}). 

Secondly, we define $\La'_{\chi}$ (cf.\ \cite[\S 56]{BH}). 
Assume that $n$ is 
odd. 
We define $\La'_{\chi}$ by
\[
\La'_{\chi}(x(1+y))=\chi(x) \psi_F(\Trd_{D/F}(\alpha y))\ \textrm{for $x \in F_2^{\times}$ and 
$1+y \in U_D^{n-1}$}. 
\]
Assume that $n$ is even. 
We define a character $\theta'$ of $U_{F_2}^1 U_D^n$
by 
\[
\theta'(x(1+y))=\chi(x) \psi_F(\Trd_{D/F}(\alpha y))\ 
\textrm{for $x \in U_{F_2}^1$ and $1+y \in U_D^n$}. 
\]
There exists a unique 
$q$-dimensional irreducible 
representation $\eta_{\theta'}$ of $U_{F_2}^1 U_D^{n-1}$
whose restriction to 
$U_{F_2}^1 U_D^n$ is a multiple of $\theta'$.  
We define $\La'_{\chi}$ to be the unique 
$q$-dimensional irreducible 
representation of $J_{2,n}$
such that
\begin{itemize}
\item $\La'_{\chi}|_{U_{F_2}^1 U_D^{n-1}}=
\eta_{\theta'}$, 
\item $\La'_{\chi}|_{F^{\times}}$ is a multiple of $
\chi|_{F^{\times}}$, and 
\item $\Tr \La'_{\chi}(\mu)=-\chi(\mu)$ for $\mu \in \bom{\mu}_{q^2-1}(F_2) \setminus \bom{\mu}_{q-1}(F)$. 
\end{itemize}

For a positive integer $n$, we set 
\begin{align*}
\pi_{\chi} &=\cInd_{J_{1,n}}^{\mathrm{GL}_2(F)} \La_{\chi}, \\  
\rho_{\chi}&=\Ind_{J_{2,n}}^{D^{\times}}
\La'_{\chi}, \\ 
\tau_{\chi}&=\Ind_{F_2/F}(\Delta_0\chi). 
\end{align*}
These are irreducible (cf.\ \cite[Theorems 11.4,  
15.1 and 54.4]{BH}). 
They are independent of the choice of 
the $F$-embeddings 
$F_2 \hookrightarrow \mathrm{M}_2(F)$ and $F_2 \hookrightarrow D$.  
The representation $\pi_{\chi}$ is cuspidal by
\cite[Theorem 14.5]{BH}. 
By the conductor-discriminant formula in 
\cite[Corollary in IV \S2]{Se},  
the Artin conductor exponent of 
$\tau_{\chi}$ equals $2n$. 
\subsubsection{Ramified case}\label{R}
Let $E$ be a totally tamely 
ramified separable
quadratic extension of $F$.
Let $n=2m-1$ be a positive odd 
integer. 
Let $(E/F,\chi)$ be a minimal admissible pair
such that $l(\chi)=n$.  
Let $\alpha \in \mathfrak{p}_E^{-n}$ be an element 
such that $\chi(1+x)=\psi_E(\alpha x)$
for $1+x \in U_E^m$.
We choose $F$-embeddings $E \hookrightarrow \mathrm{M}_2(F)$
and $E \hookrightarrow D$. 
Let 
\begin{align*}
J_{E,1,n}&=E^{\times} U_{\mathfrak{I}}^m \subset
\mathrm{GL}_2(F),\\
J_{E,2,n}&=E^{\times} U_D^m \subset D^{\times}. 
\end{align*}
We define a character $\La_{E,\chi}$ 
of $J_{1,E,n }$ by 
\begin{align*}
\La_{E,\chi}\left(x(1+y)\right)=\chi(x)\psi_F \left( \Tr (\alpha y)\right)\ \textrm{for $x \in E^{\times}$ and $1+y \in U_{\mathfrak{I}}^m$}. 
 \end{align*}
Similarly, we define a character 
$\La'_{E, \chi}$ of $J_{E,2,n}$ by 
\begin{align*}
\La'_{E,\chi}\left(x(1+y)\right)=(-1)^{v_E(x)}\chi(x)
\psi_F \left(\Trd_{D/F} (\alpha y)\right)
\ \textrm{for $x \in E^{\times}$ and $1+y \in U_D^m$}. 
\end{align*}
Let $\kappa_{E/F}$ be the non-trivial 
character of $F^{\times}$ factoring through $F^{\times}/\Nr_{E/F}(E^{\times})
\simeq U^0_F/\Nr_{E/F}(U^0_E)$.
This character is tamely ramified of order two. 
Hence, we have 
\begin{equation}\label{ke}
\kappa_{E/F}(x)=
\left(\frac{\bar{x}}{\mathbb{F}_q}\right)\ \textrm{for any $x \in U_F^0$}. 
\end{equation}  
We consider the quadratic Gauss sum 
\[
\tau(\kappa_{E/F},\psi_F)=\sum_{x \in 
(\mathcal{O}_F/\mathfrak{p})^{\times}}
\kappa_{E/F}(x) \psi_F(x)=
\sum_{x \in \mathbb{F}_q^{\times}}\left(\frac{x}{\mathbb{F}_q}\right)\psi_0(x), 
\]
where we use \eqref{ke} at the second equality. 
Recall that 
\begin{equation}\label{gaussf}
\tau(\kappa_{E/F},\psi_F)^2=\left(\frac{-1}{\mathbb{F}_q}\right) q. 
\end{equation}
Let $\lambda_{E/F}(\psi_F)$ denote the 
Langlands constant of the extension 
$E/F$ (cf.\ \cite[\S34.3]{BH}). 
Then, we have 
\begin{gather}\label{ke2}
\begin{aligned}
\lambda_{E/F}(\psi_F)& =\tau(\kappa_{E/F},\psi_F)q^{-\frac{1}{2}}, \\  
\lambda_{E/F}(\psi_F)^2 &=\kappa_{E/F}(-1)
\end{aligned}
\end{gather} 
by \cite[Proposition  34.3 (2)]{BH} and \eqref{gaussf}. 
We choose a uniformizer $\varpi_E$ of $E$. 
Let $\zeta(\alpha,\chi)=\overline{\varpi_E^n\alpha} 
\in \mathbb{F}_q$.
As in \cite[\S34.4]{BH}, we define 
a tamely ramified character $\Delta_{E,\chi}$
of $E^{\times}$ by 
\begin{itemize}
\item $\Delta_{E,\chi}(x)=
\bigl(\frac{\bar{x}}{\mathbb{F}_q}\bigr)$ for 
$x \in U_E^0$,  
\item $\Delta_{E,\chi}(\varpi_E)=
\kappa_{E/F}\left(\zeta(\alpha,\chi)\right)
\lambda_{E/F}(\psi_F)^n$. 
\end{itemize}
The order of $\Delta_{E,\chi}$ is divisible by $4$. 
By \eqref{ke} and \eqref{ke2}, we have  
\begin{equation}\label{lam}
\lambda_{E/F}(\psi_F)^n=\kappa_{E/F}(-1)^{m-1} 
\lambda_{E/F}(\psi_F)=\left(\frac{-1}{\mathbb{F}_q}\right)^{m-1}\lambda_{E/F}(\psi_F). 
\end{equation}
We set
\begin{align*}
\pi_{\chi}&=\cInd_{J_{E,1,n}}^{\mathrm{GL}_2(F)} \La_{E,\chi}, \\ 
\rho_{\chi}&=\Ind_{J_{E,2,n}}^{D^{\times}}
\La'_{E,\chi}, \\ 
\tau_{\chi}&=\Ind_{E/F}(\Delta_{E,\chi} \chi). 
\end{align*}
These are irreducible. 
They are independent of the choice of 
the $F$-embeddings $E \hookrightarrow \mathrm{M}_2(F)$ and 
$E \hookrightarrow D$.  
The representation $\pi_{\chi}$ is cuspidal by 
\cite[Theorem 14.5]{BH}. 
By the conductor-discriminant formula, 
the Artin conductor exponent of $\tau_{\chi}$
is equal to $n+2$. 

Let 
\begin{itemize}
\item $\mathcal{A}^0(F)$ be the set of equivalent  
classes of irreducible cuspidal representations 
of $\mathrm{GL}_2(F)$, 
\item $\mathcal{A}_1^0(D)$ the set of equivalent 
classes of 
irreducible smooth representations of $D^{\times}$
whose dimension is greater than one, and 
\item $\mathcal{G}^0(F)$ the set of equivalent 
classes of irreducible 
smooth representations of $W_F$ of degree two.
\end{itemize}

Let $(L/F,\chi)$ be an admissible pair. 
We choose a minimal admissible pair $(L/F,\chi')$
and a smooth 
character $\phi$ of $F^{\times}$ such that 
$\chi=\chi' \otimes \phi_L$.
Then, we set 
\begin{gather}\label{bins2}
\begin{aligned}
\pi_{\chi}& =\pi_{\chi'} \otimes 
(\phi \circ \det),\\ 
\rho_{\chi}&=\rho_{\chi'} \otimes 
(\phi \circ \Nrd_{D/F}),\\   
\tau_{\chi}&=\tau_{\chi'} \otimes 
\phi.
\end{aligned}
\end{gather}
The equivalent classes of them 
are independent of the choice of 
the pair $(\chi',\phi)$. 
Under the assumption $p \neq 2$, 
it is known that the above construction
induces bijections 
\begin{gather}\label{bins}
\begin{aligned}
& \mathbb{P}_2(F) \xrightarrow{\sim} \mathcal{A}^0(F);\ (L/F,\chi) \mapsto
\pi_{\chi}, \\
& \mathbb{P}_2(F) \xrightarrow{\sim}  \mathcal{A}_1^0(D);\ 
(L/F,\chi) \mapsto
\rho_{\chi}, \\
& \mathbb{P}_2(F) \xrightarrow{\sim} \mathcal{G}^0(F);\ (L/F,\chi) \mapsto
\tau_{\chi}
\end{aligned}
\end{gather}
(cf.\ \cite[Theorem 20.2, \S 54, Theorem 34.1]{BH}). 
The last bijection is different from 
that in \cite[Theorem 34.1]{BH}. 
The correspondences \eqref{bins} 
are modified so that they are compatible with 
the LLC and the LJLC as in 
\cite[p.\ 219, \S56]{BH}.
For a smooth representation $\pi$ of a locally 
profinite group, 
let $\pi^{\vee}$ denote its 
contragredient. 
\begin{lemma}\label{w3}
For any $(L/F,\chi) \in \mathbb{P}_2(F)$,
we have 
\begin{align} \label{w1}
\pi_{\chi}^{\vee}
& =\pi_{\chi^{\vee}}, \quad 
\rho_{\chi}^{\vee}=\rho_{\chi^{\vee}}, \quad 
\tau_{\chi}^{\vee}=\tau_{\chi^{\vee}}, \\ \label{w2}
\pi_{\chi \otimes \phi_L} & \simeq \pi_{\chi} \otimes 
(\phi \circ \det), \quad 
\rho_{\chi \otimes \phi_L} \simeq 
\rho_{\chi} \otimes (\phi \circ \Nrd_{D/F}), \quad 
\tau_{\chi \otimes \phi_L} \simeq 
\tau_{\chi} \otimes \phi
\end{align}
for any character $\phi$ of $F^{\times}$. 
\end{lemma}
\begin{proof}
We prove \eqref{w1}. 
The required assertion for $\pi_{\chi}$ follows from 
\cite[Theorem 20.2 (3)]{BH}. 
Similarly, the claim for $\rho_{\chi}$ is proved 
by using results in \cite[\S 54.4]{BH}. 
The claim for $\tau_{\chi}$ is clear 
by the construction of it.

The claim \eqref{w2} follows 
from the construction \eqref{bins2}.  
\end{proof}
\begin{lemma}\label{sec}
For $(L/F,\chi) \in \mathbb{P}_2(F)$, we have 
$c(\pi_{\chi})=c(\rho_{\chi})$. 
\end{lemma}
\begin{proof}
We can directly compute the both sides by using 
\cite[(19.6.2)]{BH}. 
\end{proof}
\subsubsection{Explicit LLC and LJLC}\label{54}
Let 
\begin{align*}
\mathrm{LL} & \colon \mathcal{A}^0(F) \xrightarrow{\sim}
\mathcal{G}^0(F);\ \pi \mapsto \mathrm{LL}(\pi), \\
\mathrm{JL} & \colon \mathcal{A}^0(F) \xrightarrow{\sim}
\mathcal{A}_1^0(D);\ \pi \mapsto \mathrm{JL}(\pi)
\end{align*}
denote the LLC and the LJLC respectively  
(cf.\ \cite[p.\ 219, \S56.1]{BH}). 
\begin{lemma}\label{cue}
For any irreducible 
cuspidal representation $\pi$ of $\mathrm{GL}_2(F)$, we have $c(\pi)=c\left(\mathrm{JL}(\pi)\right)$. 
\end{lemma}
\begin{proof}
For an irreducible smooth representation
of $\rho$ of $D^{\times}$, 
let $l(\rho)$ denote the level of $\rho$ in
\cite[\S 54.1]{BH}. 
By \cite[\S 56.1]{BH}, 
we have $l(\mathrm{JL}(\pi))=2l(\pi)$. 
Hence, the required assertion follows. 
\end{proof}
\begin{remark}
For any $\pi \in \mathcal{A}^0(F)$, 
the Artin conductor exponent of $\mathrm{LL}(\pi)$ equals $c(\pi)$. 
We will not use this fact later. 
\end{remark}
The following theorem is stated in \cite[p.\ 219, p.\ 334]{BH}, which we call the explicit LLC and LJLC. 
This theorem is due to Bushnell and Henniart. 
\begin{theorem}\label{exp}(Explicit LLC and LJLC)
Assume that $p \neq 2$. 
Let the notation be as in \eqref{bins}. 
For any admissible pair $(L/F,\chi)$, 
we have 
\[
\mathrm{JL}(\pi_{\chi})=\rho_{\chi}, \quad 
\mathrm{LL}(\pi_{\chi})=\tau_{\chi}. 
\]
\end{theorem}
\subsection{Main theorem and its application}\label{mimp}
\subsubsection{Local fundamental representation}
\label{lfr}
We have introduced 
Lubin-Tate curves in the previous sections.
In the following, only in \S \ref{lfr} and \S \ref{gc}, 
we consider any dimensional 
Lubin-Tate spaces. 
We introduce a main subject, which is called 
local fundamental representation,  
in the non-abelian Lubin-Tate theory. 
Main references are \cite[\S2.4]{Bo}, 
\cite[\S1.5]{Ca}, \cite[\S 3.5]{Da}, \cite[\S 4.5]{Fa2}
and \cite[\S2.5]{St}.

We choose an isomorphism 
$\iota \colon \mathbb{C} \xrightarrow{\sim} \overline{\mathbb{Q}}_{\ell}$. 
Let $q^{\frac{1}{2}} \in \overline{\mathbb{Q}}_{\ell}$
denote the second root of $q$ such that 
$\iota^{-1}(q^{\frac{1}{2}})$ is positive in 
$\mathbb{R}$. 
Let $d \geq 1$. 
For each $n \geq 1$, 
let $\X^d(\mathfrak{p}^n)$ be 
the $(d-1)$-dimensional 
Lubin-Tate 
space with Drinfeld level $\mathfrak{p}^n$-structure. 
We consider 
\[
\mathcal{H}^d_{\rm c}=
\varinjlim_n H_{\rm c}^{d-1}(\X^d(\mathfrak{p}^n)_{\C},\overline{\mathbb{Q}}_{\ell})\left(\frac{d-1}{2}\right), 
\]
where $\left(\frac{1}{2}\right)$ means 
the Tate twist on which 
the geometric Frobenius automorphism  
over $\mathbb{F}_q$
acts as scalar multiplication by $q^{-\frac{1}{2}}$. 
Let $D$ be the central division 
algebra over $F$ of invariant $1/d$.  
We consider the homomorphism 
\begin{equation}\label{cha}
\delta_d 
\colon G_d=\mathrm{GL}_d(F) \times D^{\times} \times W_F
\to \mathbb{Z};\ 
(g,x,\sigma) \mapsto 
v\left(\det (g) \Nrd_{D/F}(x)^{-1} \bom{a}_{F}(\sigma)^{-1}\right). 
\end{equation}
Let $G_d^0$ be the kernel of this homomorphism.
Then, there exists an action of 
$G_d^0$ on 
the tower 
$\{\X^d(\mathfrak{p}^n)_{\C}\}_{n \geq 0}$ by \cite[\S1.3]{Ca}. Hence, 
$\mathcal{H}^d_{\mathrm{c}}$ 
can be regarded as a 
representation of $G_d^0$.  
This is a smooth/continuous representation of 
$G_d^0$ (cf.\ 
\cite[\S 3.1]{Bo}, \cite[Lemma II.2.8]{HT} 
and \cite[Lemma 2.5.1 and  
Remark after it]{St}). 
In the below definition, 
we consider the usual topology on $\mathrm{GL}_d(F) \times D^{\times}$, 
the discrete topology on $W_F$, and 
the product topology on $G_d$. 
Then, we regard 
$\mathcal{H}_{\rm c}^d$ as a smooth 
representation of $G_d^0$.
\begin{definition}
We define a
$G_d$-representation
\[
\mathcal{U}^d_{\rm c}=\cInd_{G_d^0}^{G_d} 
\mathcal{H}^d_{\rm c}, 
\]
which we call the local fundamental representation. 
\end{definition}

\subsubsection{Geometrically connected components of Lubin-Tate spaces}\label{gc}
We recall group action on 
$\pi_0$ of Lubin-Tate 
spaces. 
To do so, we give complements on \cite{FGL}.  
As a result, we show Corollary \ref{stst}.  
To show this, 
we  use also a result in 
\cite{St2}. 

We recall notations and results in 
\cite[V]{FGL}.
Let $d \geq 1$.  
We simply write $\underline{d}$ for 
$\mathbb{Z}/d\mathbb{Z}$. 
Let 
\[
B^d_{\mathcal{LT},m}=\mathcal{O}_{\widehat{F}^{\rm ur}}[[(x_i)_{i \in \underline{d}}]][(s_{i,j})_{i \in \underline{d},\ 0 \leq  j \leq m d-1}]\big/
\mathcal{I}_{m,d}, 
\]
where 
$\mathcal{I}_{m,d}$ is generated by 
\[
\prod_{k=1}^d x_k-(-1)^d \varpi, 
\quad s_{i,0}^{q-1}-x_i, \quad 
s_{i,j}^q-x_{i-j} s_{i,j}-s_{i,j-1} \quad 
\textrm{for $i \in \underline{d}$ and $1 \leq j \leq md-1$}. 
\]
We set 
$\Z^d(\mathfrak{p}^m)=
\left(\Spf B^d_{\mathcal{LT},m}\right)^{\rm rig}$,
which is a $(d-1)$-dimensional 
rigid analytic variety over $\widehat{F}^{\rm ur}$. 
Note that 
$\Z^2(\mathfrak{p}^m)$ equals 
$\Y(\mathfrak{p}^{2m-1})$ in \S \ref{rcc}. 
As in \cite[Remarque II.2.3]{FGL}, 
$\Z^d(\mathfrak{p}^m)$ is an 
intermediate covering between 
$\X^d(\mathfrak{p}^m)$
and $\X^d(\mathfrak{p}^{m+1})$. 
Then, the tower 
$\{\Z^d(\mathfrak{p}^m)_{\C}\}_{m \geq 0}$
admits an action of $G_d^0$. 
Let 
$\mathfrak{I}$ be the inverse image of 
the subring consisting of all 
upper triangular matrices in $\mathrm{M}_d(\mathbb{F}_q)$
by the canonical map 
$\mathrm{M}_d(\mathcal{O}_F) \to \mathrm{M}_d(\mathbb{F}_q)$, 
which is $\mathfrak{B}$ in the notation of 
\cite[I.1]{FGL}.  
Then, $\Z^d(\mathfrak{p}^m)$ is the quotient of
$\mathbf{X}^d(\mathfrak{p}^{m+1})$
by $U_{\mathfrak{I}}^{md}=1+\varpi^m \mathfrak{I}$.
Let $\mathcal{O}_F \widehat{\otimes}_{\mathbb{F}_q} B^d_{\mathcal{LT},m}$ denote the completion 
of $\mathcal{O}_F \otimes_{\mathbb{F}_q} B^d_{\mathcal{LT},m}$ for the $\varpi \otimes 1$-adic 
topology. We set $z=\varpi \otimes 1 \in \mathcal{O}_F \widehat{\otimes}_{\mathbb{F}_q} B^d_{\mathcal{LT},m}$ and  
$
P_z=\begin{pmatrix}
\bom{0}_{d-1} & E_{d-1} \\
z & {}^t \bom{0}_{d-1}
\end{pmatrix} \in \mathrm{M}_d\left(\mathcal{O}_F \widehat{\otimes}_{\mathbb{F}_q} B^d_{\mathcal{LT},m}\right)$.  
Let 
\[
S_m=\sum_{j = 0}^{m d-1}
 \mathrm{diag}(s_{1,j},\ldots,s_{d,j}) {}^t P_z^j
 \in \mathrm{M}_d\left(\mathcal{O}_F \widehat{\otimes}_{\mathbb{F}_q} B^d_{\mathcal{LT},m}\right).
\]
We set
\[
\sum_{i=0}^{\infty} s_i z^i=
\det S_m.  
\]
We take an element $\mu \in \mathbb{F}_{q^2}^{\times}$ such that $\mu^{q-1}=-1$.  
We have the injective ring homomorphism 
\begin{equation}\label{set}
{\det}_{m,\mu} \colon 
B^1_{\mathcal{LT},m} \hookrightarrow  B^d_{\mathcal{LT},m};\ 
s_{0,i} \mapsto \mu s_i \quad \textrm{for each $1 \leq i \leq m-1$}.
\end{equation}
For $m \geq 1$, 
we have the trivial commutative diagram
\begin{equation}\label{set2}
\xymatrix{
\mathcal{B}^d_{\mathcal{LT},m} \ar@{^{(}->}[r]^{\!\!\!\!\!\rm can.} & \mathcal{B}^d_{\mathcal{LT},m+1} \\
\mathcal{B}^1_{\mathcal{LT},m} \ar@{^{(}->}[r]^{\!\!\!\!\!\rm can.}\ar@{^{(}->}[u]^{{\det}_{m, \mu}} & 
\mathcal{B}^1_{\mathcal{LT},m+1} \ar@{^{(}->}[u]^{{\det}_{m+1, \mu}}. 
}
\end{equation}
Let $\mathscr{LT}$ denote the 
formal $\mathcal{O}_F$-module 
over $\mathcal{O}_F$ defined by 
\[
[\varpi]_{\mathscr{LT}}(X)=X^q+\varpi X, \quad 
X+_{\mathscr{LT}} Y=X+Y, \quad 
[\zeta]_{\mathscr{LT}}(X)=\zeta X \quad 
\textrm{for $\zeta \in \mathbb{F}_q$}. 
\]
Let 
\[
\mathscr{LT}[\mathfrak{p}^m]_{\rm prim}
=\left\{x \in \overline{F} \mid 
[\varpi^m]_{\mathscr{LT}}(x)=0,\ 
[\varpi^{m-1}]_{\mathscr{LT}}(x) \neq 0
\right\},  
\] and 
$F_{\mathscr{LT},m}=
\widehat{F}^{\rm ur}(\mathscr{LT}[\mathfrak{p}^m]_{\rm prim})$. 
By the Lubin-Tate theory, 
we have the isomorphism
\begin{equation}\label{fly}
\mathrm{Gal}\left(F_{\mathscr{LT},m}/\widehat{F}^{\rm ur }\right)
\simeq (\mathcal{O}_F/\mathfrak{p}^m)^{\times};\ 
\sigma \mapsto a_{\sigma},  
\end{equation} 
where 
$[a_{\sigma}]_{\mathscr{LT}}(x)=
\sigma(x)$ 
for $x \in \mathscr{LT}[\mathfrak{p}^m]_{\rm prim}$. 
Note that we have an isomorphism 
\[
B^1_{\mathcal{LT},m}=\mathcal{O}_{\widehat{F}^{\rm ur}}[s_{1,m-1}]/
([\varpi^m]_{\mathscr{LT}}(s_{1,m-1})/
[\varpi^{m-1}]_{\mathscr{LT}}(s_{1,m-1})
) \simeq 
\mathcal{O}_{F_{\mathscr{LT},m}}
\] 
for $m \geq 1$. 
The map \eqref{set} induces the morphism of rigid analytic varieties 
\begin{equation}\label{ccon}
\Z^d(\mathfrak{p}^m) \to 
\Sp F_{\mathscr{LT},m}.
\end{equation}
For a rigid analytic variety $\Y$ over 
$\C$, let 
$\pi_0(\Y)$ be 
the set of the  
connected components of $\Y$. 
The map \eqref{ccon} induces 
\begin{equation}\label{conn}
\pi_0(\Z^d(\mathfrak{p}^m)_{\C}) 
\to \pi_0\left(\Sp \left(F_{\mathscr{LT},m} \times_{\widehat{F}^{\rm ur}} \C\right)\right) \simeq (\mathcal{O}_F/\mathfrak{p}^m)^{\times},  
\end{equation}
where the isomorphism 
is given by \eqref{fly}. 
By \cite[Theorem 4.4 (i)]{St2} and 
$\det U_{\mathfrak{I}}^{md}=U_F^m$, 
the cardinality of $\pi_0(\Z^d(\mathfrak{p}^m)_{\C})$ equals $q^{m-1}(q-1)$.  
Hence, the map \eqref{conn} is bijective. 
Therefore, by taking the projective limit of 
\eqref{conn}, we have 
\begin{equation*}\label{conn2}
\varprojlim_m \pi_0(\Z^d(\mathfrak{p}^m)_{\C}) \xrightarrow{\sim} \mathcal{O}_F^{\times}. 
\end{equation*}

As in \cite[p.\ 398]{FGL}, we set 
\[
\widehat{B}^d_{\mathcal{LT},\infty}=
\biggl(\bigcup_{m=1}^{\infty} B^d_{\mathcal{LT},m}\biggr)^{\widehat{}}, 
\]
where $(\cdot)^{\widehat{}}$ denotes the 
$(x_1,\ldots,x_d)$-adic completion. 
By \eqref{set} and \eqref{set2}, we have the injective homomorphism 
\begin{equation}\label{debt}
\widehat{\det}_{\infty,\mu} \colon 
\widehat{B}^1_{\mathcal{LT},\infty} \to 
\widehat{B}^d_{\mathcal{LT},\infty}. 
\end{equation}

Let $A^d_{\mathcal{I}nt}$ be as in 
\cite[V.1]{FGL}. 
By \cite[pp.\ 398--399]{FGL}, 
we have the injective ring homomorphism 
\begin{equation}\label{fgl00}
(\ast) \colon \widehat{B}^d_{\mathcal{LT},\infty} \hookrightarrow
A^d_{\mathcal{I}nt}. 
\end{equation}
By \cite[V.4]{FGL}, we have the determinant map 
\begin{equation}\label{fgl0}
{\det}_{\mu} \colon 
A_{\mathcal{I}nt}^1 \to A_{\mathcal{I}nt}^d. 
\end{equation}
\begin{proposition}
For $m \geq 1$, we have the commutative diagram
\begin{equation}\label{fgl}
\xymatrix{
B^d_{\mathcal{LT},m} \ar@{^{(}->}[r]^{\rm can.} & 
\widehat{B}^d_{\mathcal{LT},\infty} \ar@{^{(}->}[r]^{(\ast)} & A^d_{\mathcal{I}nt} \\
B^1_{\mathcal{LT},m} \ar@{^{(}->}[u]^{{\det}_{m,\mu}}\ar@{^{(}->}[r]^{\rm can.} & \widehat{B}^1_{\mathcal{LT},\infty} \ar[r]^{\!\!\simeq}\ar@{^{(}->}[u]^{\widehat{\det}_{\infty,\mu}}
& A^1_{\mathcal{I}nt}.  \ar@{^{(}->}[u]^{{\det}_{\mu}}
}
\end{equation}
\end{proposition} 
\begin{proof}
The left commutativity in \eqref{fgl}
is clear. 
Hence, we prove the  
right commutativity.  
By definition, 
we have $A^1_{\mathcal{I}nt}=\mathbb{F}((t^{1/q^{\infty}}))$ as in 
\cite[V.1]{FGL}. The inverse map of 
$\widehat{B}^1_{\mathcal{LT},\infty} \to 
A^1_{\mathcal{I}nt}$ is given by
$t \mapsto \lim_{j\to +\infty} s_{0,j}^{q^j}$ (cf.\ 
\cite[p.\ 384]{FGL}).
Let $\widehat{A}^d_{\mathcal{LT},\infty}$ 
be as in 
\cite[pp.\ 370--371]{FGL}. 
We have the natural injective map 
\begin{equation}\label{q1}
\widehat{B}^d_{\mathcal{LT},\infty} 
\hookrightarrow 
\widehat{A}^d_{\mathcal{LT},\infty},   
\end{equation}
and 
the isomorphism
\begin{equation}\label{q2}
(\mathrm{d\acute{e}composition})^{\ast} \colon \widehat{A}^d_{\mathcal{LT},\infty} \xrightarrow{\sim}
A^d_{\mathcal{I}nt} 
\end{equation}
in \cite[p.\ 398]{FGL} (cf.\ The inverse map of 
\eqref{q2} is given in \cite[V.2.1]{FGL}). The composite 
of \eqref{q1} and \eqref{q2} equals 
the injective map \eqref{fgl00}. 
By \cite[V.4]{FGL}, we have the determinant 
map 
\[
{\det}_{\mu} \colon 
\widehat{A}^1_{\mathcal{LT},\infty} \hookrightarrow 
\widehat{A}^d_{\mathcal{LT},\infty}. 
\]
By the definition of this and the definition 
of ${\widehat{\det}_{\infty,\mu}}$ in \eqref{debt}, 
we have 
the commutative diagram 
\[
\xymatrix{
\widehat{B}^d_{\mathcal{LT},\infty} \ar@{^{(}->}[r]^{\eqref{q2}} &
\widehat{A}^d_{\mathcal{LT},\infty} \\
\widehat{B}^1_{\mathcal{LT},\infty} \ar[u]^{\widehat{\det}_{\infty,\mu}} \ar@{=}[r] &
\widehat{A}^1_{\mathcal{LT},\infty}. \ar[u]^{{\det}_{\mu}}
}
\]
Furthermore, by \cite[p.\ 404]{FGL}, 
we have 
the commutative diagram 
\[
\xymatrix{
\widehat{A}^d_{\mathcal{LT},\infty} \ar[rr]_{\simeq}^{(\mathrm{d\acute{e}composition})^{\ast}} & &
A^d_{\mathcal{I}nt} \\
\widehat{A}^1_{\mathcal{LT},\infty} 
\ar[u]^{{\det}_{\mu}}\ar[rr]_{\simeq}^{(\mathrm{d\acute{e}composition})^{\ast}} &  &
A^1_{\mathcal{I}nt}. \ar[u]^{{\det}_{\mu}}
}
\]
The required assertion follows
from the above two commutative diagrams. 
\end{proof} 
\begin{remark}
A similar determinant morphism 
is studied in \cite[(2.7.3)]{WeSemi} 
by using \cite{He}, 
which is based on 
the theory of displays due to Zink in \cite{Zi}. 
In \cite{FGL}, formal models of Lubin-Tate tower 
are described on the basis of 
the theory of coordinate modules, 
because the characteristic of $F$ is positive.
\end{remark}
Let 
\[
\delta_d^1 \colon G_d \to F^{\times};\ 
(g,d,\sigma) \mapsto 
\det(g) \Nrd_{D/F}(d)^{-1} \bom{a}_{F}(\sigma)^{-1}
\]
and $G_d^1=\ker \delta_d^1$. 
The restriction of $\delta_d^1$ to
$G_d^0$ induces the surjective homomorphism  
\[
\delta_d^1 \colon G_d^0 \to \mathcal{O}_F^{\times}. 
\]
\begin{corollary}\label{stst0}
{\rm 1}.\ 
The subgroup $G_d^1$ acts on 
$\varprojlim_m 
\pi_0(\Z^d(\mathfrak{p}^m)_{\C})$
trivially. \\
{\rm 2}.\ Let 
\[
H_d=\mathrm{GL}_d(\mathcal{O}_F) \times \mathcal{O}_D^{\times} \times 
I_F \subset G_d^0. 
\]
Then, $H_d$ acts on 
$\varprojlim_m 
\pi_0(\Z^d(\mathfrak{p}^m)_{\C}) \simeq 
\mathcal{O}_F^{\times}$ via $\delta_d^1$.  
\end{corollary}
\begin{proof}
We prove the first assertion. 
Let $(g,d,\sigma) \in G_d^1$. 
We can write 
$(g,x,\sigma)=(g,x_1,1)(1,x_2,\sigma)$ with 
$\det(g)=\Nrd_{D/F}(x_1)$ and 
$\Nrd_{D/F}(x_2) \bom{a}_F(\sigma)=1$. 
We set 
\[
F \widehat{\otimes}_{\mathbb{F}_q} 
\widehat{B}^d_{\mathcal{LT},\infty}=
\left(\mathcal{O}_F \widehat{\otimes}_{\mathbb{F}_q} 
\widehat{B}^d_{\mathcal{LT},\infty}\right)\left[z^{-1}\right].
\] 
For an element 
$d'=\sum_{i \in \mathbb{Z}} a_i \varphi^i \in D^{\times}$ with $a_i \in \mathbb{F}_{q^d}$, 
we set 
\[
d'_z=\sum_{i \in \mathbb{Z}} 
\mathrm{diag}\left(1 \otimes a_i,1 \otimes a_i^q,\ldots,1 \otimes a_i^{q^{d-1}}\right) P_z^i
\in \mathrm{M}_d(F \widehat{\otimes}_{\mathbb{F}_q} 
\widehat{B}^d_{\mathcal{LT},\infty}), 
\]
Note that $\Nrd_{D/F}(d')_z=\det (d'_z)$ in 
$\widehat{B}^1_{\mathcal{LT},\infty}$. 
For an element $g \in \mathrm{M}_d(F)$, 
let $g_z$ denote the image of $g$
by the natural map 
$\mathrm{id} \otimes 1 \colon \mathrm{M}_d(F) \to \mathrm{M}_d(F \widehat{\otimes}_{\mathbb{F}_q} 
\widehat{B}^d_{\mathcal{LT},\infty})$. 
Let $T_{\mathcal{LT}}$ be as in 
\cite[IV.1]{FGL}. 
Then, by 
\cite[Remarque II.2.4 or the proof of Corollaire IV.2.5]{FGL}, the group $\mathrm{GL}_d(F) \times D^{\times}$ 
acts on the matrix $T_{\mathcal{LT}}$ by 
\begin{equation}\label{ta}
T_{\mathcal{LT}} \mapsto 
{}^t {d'^{-1}_z} T_{\mathcal{LT}} {}^t g_z 
\quad
\textrm{for $(g,d') \in \mathrm{GL}_d(F) 
\times D^{\times}$}.
\end{equation} 
The determinant morphism \eqref{fgl0} is given by 
\[
\Spf A^d_{\mathcal{I}nt}
\to \Spf A^1_{\mathcal{I}nt};\ 
T_{\mathcal{LT}} \mapsto 
\mu \det T_{\mathcal{LT}}. 
\]
Hence, by \eqref{ta}, the element $(g,x_1,1)$ acts on 
$A^1_{\mathcal{I}nt}$ trivially. 
By the diagram \eqref{fgl}, 
the element $(g,x_1,1)$ acts on 
$\widehat{B}^1_{\mathcal{LT},\infty}$
trivially. We write $(1,x_2,\sigma)=
(1,x'_2,1) (1,\varphi^{-n_{\sigma}},\sigma)$
with $x'_2 \in \mathcal{O}_D^{\times}$.
Let $\bom{a}^0_{F,\varpi}(\sigma)=
\bom{a}_F(\sigma)/\varpi^{n_{\sigma}} \in \mathcal{O}_F^{\times}$. 
By the Lubin-Tate theory, 
$(1,\varphi^{-n_{\sigma}},\sigma)$ 
acts on $\widehat{B}^1_{\mathcal{LT},\infty}$  
as scalar multiplication by 
$\bom{a}^0_{F,\varpi}(\sigma)^{-1}$.
 Hence, by \eqref{ta}, 
 the element $(1,x_2,\sigma)$ acts on 
 $\widehat{B}^1_{\mathcal{LT},\infty}$ 
 trivially. 
Therefore, the required assertion follows. 

The second assertion follows from 
\eqref{ta} and the proof of the first assertion. 
\end{proof}
\begin{remark}
By $\X^d(\mathfrak{p}^{m+1}) \to 
\Z^d(\mathfrak{p}^m) \to 
\X^d(\mathfrak{p}^{m})$ for each $m \geq 1$, 
we have an isomorphism 
\[
\varprojlim_m \pi_0(\Z^d(\mathfrak{p}^m)_{\C})
\simeq \varprojlim_m \pi_0(\X^d(\mathfrak{p}^m)_{\C}).
\]
Then, Corollary \ref{stst0}.2 is proved in 
\cite[Theorem 4.4 (i)]{St2}. 
\end{remark}
\begin{corollary}\label{stst01}
The group $G_d^0$ acts on 
\[
 \varprojlim_m 
 \pi_0(\Z^d(\mathfrak{p}^m)_{\C}) \xrightarrow{\sim} \mathcal{O}_F^{\times}
\]
via $\delta_d^1$. 
\end{corollary}
\begin{proof}
We have 
$G_d^0=G_d^1H_d$, because 
the restriction $\delta_d^1|_{H_d} \colon 
H_d \to \mathcal{O}_F^{\times}$ is surjective. 
Hence, by Corollary \ref{stst0}, the required assertion follows. 
\end{proof}
\begin{corollary}\label{stst}
The $G_d$-representation 
$\mathcal{U}^d_{\rm c}$ is invariant under 
twisting 
by any character of $G_d$ factoring through 
$\delta_d^1 \colon G_d \to F^{\times}$. 
\end{corollary}
\begin{proof}
For each $m \geq 1$, 
we have the natural morphisms of rigid analytic varieties  
$\X^d(\mathfrak{p}^{m+1}) \to \Z^d(\mathfrak{p}^m) \to 
\X^d(\mathfrak{p}^m)$. 
This induces the isomorphism  
\begin{equation}\label{zapp}
\mathcal{H}^d_{\rm c}=
\varinjlim_m H_{\rm c}^{d-1}(\mathbf{X}^d(\mathfrak{p}^m)_{\C},\overline{\mathbb{Q}}_{\ell})\left(\frac{d-1}{2}\right)
\xrightarrow{\sim}
\varinjlim_m H_{\rm c}^{d-1}
(\Z^d(\mathfrak{p}^m)_{\C},\overline{\mathbb{Q}}_{\ell})\left(\frac{d-1}{2}\right). 
\end{equation}
Let $\alpha \in \mathcal{O}_F^{\times}$. 
For an integer $m \geq 1$, 
let $\alpha_m$ denote the image of $\alpha$ 
by the canonical map 
$\mathcal{O}_F^{\times} \to 
\mathcal{O}^{\times}_F/U_F^m$. 
Let $\Z^{d,\alpha_m}(\mathfrak{p}^m)$
denote the 
connected component of $\Z^d(\mathfrak{p}^m)$
corresponding to $\alpha_m$ in 
$\pi_0(\Z^d(\mathfrak{p}^m)) \simeq 
\mathcal{O}_F^{\times}/U_F^m$. 
Then, 
$\{\Z^{d,\alpha_m}(\mathfrak{p}^m)\}_{m \geq 1}$
makes a projective system.
Let $\xi$ be a character of $F^{\times}$. 
Let $\xi_0=\xi|_{\mathcal{O}_F^{\times}}$. 
The image of $\xi_0$ equals 
$\bom{\mu}_n(\overline{\mathbb{Q}}_{\ell})$
with some integer $n \geq 1$. 
We put $U_{\xi}=\ker \xi_0$. 
We consider the composite 
\[
\xi' \colon G_d^0 \xrightarrow{\delta_d^1} \mathcal{O}_F^{\times} \xrightarrow{\xi_0}
\bom{\mu}_n(\overline{\mathbb{Q}}_{\ell}). 
\]
Let $G^0_{\xi}=\ker \xi'$. 
For $m \geq 1$, we write $U_{\xi,m}$ for the image of 
$U_{\xi}$ by $\mathcal{O}_F^{\times} \to 
\mathcal{O}_F^{\times}/U_F^m$.  
We consider the projective system 
$\{\Z^d_{\xi, m}\}_{m \geq 1}=
\left\{\coprod_{\alpha_m \in U_{\xi,m}} 
\Z^{d,\alpha_m}(\mathfrak{p}^m)\right\}_{m \geq 1}$. 
By Corollary \ref{stst01}, 
the stabilizer of $\{\Z^d_{\xi,m}\}_{m \geq 1}$ in $G_d^0$ equals 
$G_{\xi}^0$. 
Since the quotient $G_d^0/G_{\xi}^0$ is cyclic, 
we have $G_d^0$-equivariant isomorphisms 
\begin{align*}
\mathcal{H}^d_{\rm c} & \simeq  
\varinjlim_m H_{\rm c}^{d-1}
(\Z^d(\mathfrak{p}^m)_{\C},\overline{\mathbb{Q}}_{\ell})
\left(\frac{d-1}{2}\right) \\
& \xrightarrow{\sim} 
\Ind_{G_{\xi}^0}^{G_d^0} \left(\varinjlim_m H_{\rm c}^{d-1}
(\Z^d_{\xi,m},\overline{\mathbb{Q}}_{\ell})\right)\left(\frac{d-1}{2}\right) 
\end{align*}
by \eqref{zapp}. 
Hence, 
we have an isomorphism 
\[
\mathcal{H}^d_{\rm c}  \otimes  \xi'
\simeq \mathcal{H}^d_{\rm c} 
\] 
as $G_d^0$-representations. 
Therefore, 
we have isomorphisms 
\[
\mathcal{U}_{\rm c}^d \otimes 
(\xi \circ \delta_d^1) \simeq 
\cInd_{G_d^0}^{G_d}
\left(\mathcal{H}^d_{\rm c} \otimes \xi'\right)
\simeq \cInd_{G_d^0}^{G_d}
\mathcal{H}^d_{\rm c}=\mathcal{U}_{\rm c}^d
\] as $G_d$-representations. 
Hence, the required assertion follows. 
\end{proof}

\subsubsection{Main results}\label{MR}
In the following, we always assume that 
$d=2$. We omit the indices
$d$ in the notations 
in \S \ref{lfr} and \ref{gc}. 
 
We have an isomorphism 
\begin{equation}\label{dat}
\varinjlim_n H_{\mathrm{c}}^1(\mathrm{LT}(\mathfrak{p}^n)_{\C},\overline{\mathbb{Q}}_{\ell})\left(\frac{1}{2}\right) 
\simeq \bigoplus_{h \in \mathbb{Z}}
\left(\varinjlim_n H_{\mathrm{c}}^1(\X^{(h)}(\mathfrak{p}^n)_{\C},\overline{\mathbb{Q}}_{\ell})\right)\left(\frac{1}{2}\right)
\simeq \mathcal{U}_{\mathrm{c}}
\end{equation}
as $G$-representations (cf.\ \cite[(3.5.2)]{Da} or 
\cite[\S4.5]{Fa2}). 
Let $\mathrm{LT}(\mathfrak{p}^n)/\varpi^{\mathbb{Z}}$
denote the quotient of $\mathrm{LT}(\mathfrak{p}^n)$
by the action of $\varpi \in D^{\times}$.   
Note that the subgroup 
$\{(x,x) \in G_D \mid x \in F^{\times}\} 
\subset G_D$ acts on \eqref{dat} trivially. 
 We set 
 \begin{equation}\label{rep}
 \overline{\mathcal{U}}_{\rm c}
 =\varinjlim_m H_{\mathrm{c}}^1
 ((\mathrm{LT}(\mathfrak{p}^m)/\varpi^{\mathbb{Z}})_{\C},\overline{\mathbb{Q}}_{\ell})\left(\frac{1}{2}\right), 
 \end{equation}
 which is regarded 
 as a 
 $\mathrm{GL}_2(F)/\varpi^{\mathbb{Z}} \times 
 D^{\times}/\varpi^{\mathbb{Z}} \times W_F$
-representation. 
As a $\mathrm{GL}_2(F)/\varpi^{\mathbb{Z}} \times 
 D^{\times}/\varpi^{\mathbb{Z}}$-representation, 
 this is smooth (cf.\ \cite[Lemma 2.5.1]{St}). 
By \eqref{dat}, the representation \eqref{rep}
is regarded as a $G$-subrepresentation of 
$\mathcal{U}_{\rm c}$.

We will prove the following proposition 
in a purely local manner in \S \ref{66}. 
\begin{proposition}\label{Main}
Let $\ch F$ denote the characteristic of $F$. 
Assume that  $\ch F=p \neq 2$. 
For any admissible 
pair $(L/F,\chi)$, 
let $\pi_{\chi}$, $\tau_{\chi}$ and $\rho_{\chi}$
be as in \eqref{bins}. 
We choose a uniformizer 
$\varpi$ of $F$. 
Assume that $(L/F,\chi)$ is minimal 
and 
$\chi(\varpi)=1$. 
Then, there exists a $G$-equivariant injection
\[
\pi_{\chi} \otimes \rho_{\chi}^{\vee} \otimes 
\tau_{\chi}^{\vee} 
\hookrightarrow \overline{\mathcal{U}}_{\rm c}. 
\]
\end{proposition}
By admitting
Proposition \ref{Main} and using Corollary \ref{stst}, we obtain 
our main theorem in this paper. 
\begin{theorem}\label{Mainc}
Let $(L/F,\chi)$ be an admissible pair. \\
{\rm 1}.\ We have a $G$-equivariant injection 
\begin{equation}\label{injj0}
\pi_{\chi} \otimes \rho_{\chi}^{\vee} \otimes 
\tau_{\chi}^{\vee} 
\hookrightarrow \mathcal{U}_{\mathrm{c}}.
\end{equation}
{\rm 2}.\ The injection \eqref{injj0} induces 
the isomorphism 
\[
\mathrm{Hom}_{\mathrm{GL}_2(F)}
(\mathcal{U}_{\rm c},\pi_{\chi}) \xrightarrow{\sim}
\rho_{\chi} \otimes \tau_{\chi} 
\] 
as $D^{\times} \times W_F$-representations. 
\end{theorem}
\begin{proof}
We prove the first assertion. 
Let $(L/F,\chi)$ be an admissible 
pair. We take a minimal admissible 
pair $(L/F, \chi')$ and 
a character $\xi$ of $F^{\times}$ such 
that 
$\chi=\chi' \otimes \xi_L$. 
We take a second root   
$c_0 \in \overline{\mathbb{Q}}_{\ell}^{\times}$
of $\chi'(\varpi)$.  
Let $\lambda$ be the unramified character of 
$F^{\times}$ which sends 
$\varpi$ to $c_0$. 
Then, 
$\left(L/F,\chi \otimes (\xi \lambda)^{-1}_L\right)$ 
is a minimal admissible pair such that 
$\left(\chi \otimes (\xi \lambda)^{-1}_L\right)(\varpi)=1$. 
Let $\phi$ be the character of $G$ which is 
the composite of $\delta^1$ and the character 
$\xi \lambda$ of 
$F^{\times}$.  
By Lemma \ref{w3} and Proposition \ref{Main}, 
we have a 
$G$-equivariant injection 
\begin{equation}\label{ST0}
\left(\pi_{\chi} \otimes \rho_{\chi}^{\vee} \otimes 
\tau_{\chi}^{\vee} \right) \otimes \phi^{-1}
\simeq \pi_{\chi \otimes (\xi \lambda)_L^{-1}} \otimes \rho_{\chi \otimes (\xi \lambda)_L^{-1}}^{\vee} \otimes 
\tau_{\chi \otimes (\xi \lambda)_L^{-1}}^{\vee} \hookrightarrow
\overline{\mathcal{U}}_{\mathrm{c}} \subset \mathcal{U}_{\mathrm{c}}. 
\end{equation}
By twisting this by $\phi$ and using Corollary 
\ref{stst}, 
we obtain the  
$G$-equivariant injection
\begin{equation}\label{ST}
\pi_{\chi} \otimes \rho_{\chi}^{\vee} \otimes 
\tau_{\chi}^{\vee} \hookrightarrow \mathcal{U}_{\mathrm{c}} \otimes \phi \simeq \mathcal{U}_{\mathrm{c}}. 
\end{equation}
Hence, we obtain the claim. 

We prove the second assertion. 
We simply write $\pi$, $\rho$ and $\tau$
for $\pi_{\chi}$, $\rho_{\chi}$ and 
$\tau_{\chi}$ respectively. 
We simply write $\xi_1$, $\xi_2$ and $\xi_3$
for $\phi|_{\mathrm{GL}_2(F) \times\{1\} \times \{1\}}$, $\phi|_{\{1\} \times D^{\times} \times \{1\}}$
and $\phi|_{\{1\} \times \{1\} \times W_F}$
respectively. 
Clearly, we have $\phi=\xi_1 \otimes \xi_2 \otimes \xi_3$ as $G$-representations. 
We have $D^{\times} \times W_F$-equivariant homomorphisms 
 \begin{gather}\label{zxz}
 \begin{aligned}
 \mathrm{Hom}_{\mathrm{GL}_2(F)}(\mathcal{U}_{\mathrm{c}}, \pi) & \simeq 
 \mathrm{Hom}_{\mathrm{GL}_2(F)}\left(\mathcal{U}_{\mathrm{c}} \otimes \xi_1^{-1}, \pi \otimes \xi_1^{-1}\right) \\
 & \simeq 
 \mathrm{Hom}_{\mathrm{GL}_2(F)}\left(\mathcal{U}_{\mathrm{c}}
, \pi \otimes \xi_1^{-1}\right) \otimes 
 \left(\xi_2 \otimes \xi_3\right) \\
 & \xrightarrow{\sim} \mathrm{Hom}_{\mathrm{GL}_2(F)/\varpi^{\mathbb{Z}}}\left(\overline{\mathcal{U}}_{\rm c}, \pi \otimes \xi_1^{-1}\right) 
 \otimes \left(\xi_2 \otimes \xi_3\right) \\
 &\to \rho \otimes \tau, 
 \end{aligned}
 \end{gather}
 where  
 we use $\mathcal{U}_{\mathrm{c}} \simeq \mathcal{U}_{\mathrm{c}} \otimes \phi$ by Corollary \ref{stst}
 at the second isomorphism, 
 the third isomorphism follows from the same argument 
 as the proof of \cite[Theorem 2.5.2]{St}, and 
 the fourth homomorphism is induced by 
\eqref{ST0}. 
The fourth map in \eqref{zxz} is surjective, 
because $\overline{\mathcal{U}}_{\rm c}$
is a smooth 
$\mathrm{GL}_2(F)/\varpi^{\mathbb{Z}}$-representation, and
$\pi \otimes \xi_1^{-1}$ is an injective object 
in the category of smooth $\mathrm{GL}_2(F)/\varpi^{\mathbb{Z}}$-representations 
by \cite[Theorem 5.4.1]{Cs}. 
Hence, \eqref{zxz} gives the  
$D^{\times} \times W_F$-equivariant surjection 
\begin{equation}\label{zax} 
\mathrm{Hom}_{\mathrm{GL}_2(F)}\left(\mathcal{U}_{\mathrm{c}},\pi\right) \twoheadrightarrow
\rho \otimes \tau, 
\end{equation}
which corresponds to the one induced by \eqref{ST}. 
By \cite[Theorem 3.7]{Mi0} and 
\cite[Theorem 2.5.2 (ii)]{St} 
(cf.\ \cite{Mi} and \cite[Proposition 1.1]{IT2}), 
it is shown that 
\[
\dim_{\overline{\mathbb{Q}}_{\ell}}\mathrm{Hom}_{\mathrm{GL}_2(F)}
(\mathcal{U}_{\rm c}, \pi)=
2 \dim_{\overline{\mathbb{Q}}_{\ell}} \mathrm{JL}(\pi)
\]
in a purely geometric manner. 
Hence, it suffices to show  
\begin{equation}\label{de}
\dim_{\overline{\mathbb{Q}}_{\ell}} \mathrm{JL}(\pi)=\dim_{\overline{\mathbb{Q}}_{\ell}} \rho
\end{equation}
to prove that \eqref{zax} is a bijection. 
By Lemmas \ref{sec} and \ref{cue}, we have 
$c(\rho)=c\left(\mathrm{JL}(\pi)\right)$. Since the $\mathrm{JL}$ preserves 
character twists, to prove \eqref{de}, 
it suffices to show it in the case where 
$\pi$ is minimal
by \eqref{w2}. 
By the construction of the middle bijection in \eqref{bins} in the minimal case, the dimension of a minimal irreducible smooth representation of $D^{\times}$ depends only on its conductor and the number $q$. Hence, we obtain 
the claim. 
 \end{proof}
We recall the non-abelian 
Lubin-Tate theory for $\mathrm{GL}(2)$.  
This is proved by Deligne if $F=\mathbb{Q}_p$ and 
$p \neq 2$ in a letter to 
Piatetskii-Shapiro, 
and by Carayol
in general   
(cf.\ \cite{Ca2} and \cite{Ca}). 
\begin{theorem}(NALT for $\mathrm{GL}(2)$) \\\label{MC} 
{\rm 1}.\ For any irreducible cuspidal 
representation $\pi$ of $\mathrm{GL}_2(F)$, 
there exists a $G$-equivariant injection
\begin{equation}\label{injj}
\pi \otimes \mathrm{JL}(\pi)^{\vee} \otimes 
\mathrm{LL}(\pi)^{\vee} 
\hookrightarrow \mathcal{U}_{\rm c}. 
\end{equation}
{\rm 2}.\ Let the notation be as in $1$. 
The inclusion \eqref{injj} induces 
the isomorphism 
\[
\mathrm{Hom}_{\mathrm{GL}_2(F)}
(\mathcal{U}_{\rm c}, \pi) \xrightarrow{\sim} \mathrm{JL}(\pi) 
\otimes \mathrm{LL}(\pi) 
\] 
as $D^{\times} \times W_F$-representations. 
\end{theorem}
In the following, 
we introduce a consequence of 
 Theorem \ref{Mainc}.
\begin{theorem}\label{NALexp}
We assume that $\ch F=p \neq 2$. 
Then, Theorem \ref{exp} is equivalent to 
Theorem \ref{MC}.
\end{theorem}
\begin{proof}
By using \eqref{bins} and Theorem \ref{Mainc}, we immediately 
obtain 
the required assertion. 
\end{proof}
\begin{remark}
Let $(F^{\times})^{\vee}$ denote the set of all 
smooth 
characters of $F^{\times}$. 
We have an isomorphism 
\begin{equation}\label{gca}
\varinjlim_m 
H_{\rm c}^2(\mathrm{LT}(\mathfrak{p}^m)_{\C},
\overline{\mathbb{Q}}_{\ell}) \simeq 
\bigoplus_{\chi \in (F^{\times})^{\vee}} 
\chi \circ \delta^1
\end{equation}
as $G$-representations by 
Poincar\'e duality and 
results in \S \ref{gc}. 
Let $\mathrm{St}$ 
denote the Steinberg representation of $\mathrm{GL}_2(F)$. 
For a character $\chi$ of $F^{\times}$, 
let $\mathrm{St}_{\chi}$ denote the 
twist of $\mathrm{St}$ by $\chi \circ \det$.  
By using \eqref{gca}, 
we can prove that  
\[
\mathrm{Hom}_{\mathrm{GL}_2(F)}(\mathcal{U}_{\rm c},
\mathrm{St}_{\chi}) \simeq 
\chi \circ \Nrd_{D/F} \otimes 
\chi
\]
 as $D^{\times} 
 \times W_F$-representations 
by applying 
\cite[Corollary 4.8]{Fa2} to the case 
where $n=2$, $q=1$ and 
$i(\mathfrak{s})=1$ 
(cf. \cite[Proposition 2.1]{IT}).  
\end{remark}

\section{Proof of Proposition \ref{Main}}\label{66}
In this section, on the basis of the analysis 
given in \S \ref{2}, \S \ref{4} and \S \ref{33}, 
we will give a proof of Proposition \ref{Main} case by case. 
\subsubsection{Unramified case of level zero}
\label{zeroo}
Let $\varpi_1 \in 
\mathscr{F}[\mathfrak{p}_{F_2}]_{\mathrm{prim}}$. 
We recall the group action on $\overline{\X}_{1,1}$. 
By Lemma \ref{xy}, 
the reduction $\overline{\X}_{1,1}$ is
isomorphic to the curve ${X_{\rm DL}}$ in \S \ref{As2}. 
\begin{lemma}\label{zero}
Let 
$(g,d) \in 
\mathrm{GL}_2(\mathcal{O}_F) \times \mathcal{O}_D^{\times}$. 
Then, the induced action of $(g,d)$ on  
$\overline{\X}_{1,1} \simeq X_{\rm DL}$
equals the action of 
$(\bar{g},\bar{d}) \in 
\mathrm{GL}_2(\mathbb{F}_q) \times \mathbb{F}_{q^2}^{\times}$ given in 
\S \ref{As2}. 
\end{lemma}
\begin{proof}
It is well-known (cf.\ \cite{Yo}). 
We omit its proof. 
\end{proof}
Let $\chi \in C$. 
Let $\pi^0_{\chi}$ be the inflation of 
the irreducible cuspidal 
$\mathrm{GL}_2(\mathbb{F}_q)$-representation 
$H_{\rm c}^1(X_{\mathrm{DL},\mathbb{F}},\overline{\mathbb{Q}}_{\ell})_{\chi}$ in \S \ref{As2}
by the reduction map 
$\mathrm{GL}_2(\mathcal{O}_F) \twoheadrightarrow
\mathrm{GL}_2(\mathbb{F}_q)$. 
We regard $\chi$ as a character of 
$\mathcal{O}_D^{\times}$ via the canonical map 
$\mathcal{O}_D^{\times} \to \mathbb{F}_{q^2}^{\times}$, 
for which we write $\chi_D$.

Let 
\[
W'_{F_2}=\left\{(1,\varpi^{-n_{\sigma}},\sigma) 
\in G 
\mid 
\sigma \in W_{F_2}\right\}. 
\]
Let $\Delta'_0 \chi'$ denote the character 
of $W'_{F_2}$ defined by  
\[
\Delta'_0 \chi'(1,\varpi^{-n_{\sigma}},\sigma)=
\Delta_0(\sigma)
\chi\left(\overline{\bom{a}^0_{F_2,\varpi}(\sigma)}\right)\]
for $\sigma \in W_{F_2}$. 
\begin{lemma}\label{level10}
We have an isomorphism 
\begin{equation}\label{bb0}
H_{\rm c}^1(\overline{\X}_{1,1},\overline{\mathbb{Q}}_{\ell})_{\rm cusp}\left(\frac{1}{2}\right) \simeq 
\bigoplus_{\chi \in C} 
\left(\pi^0_{\chi^{\vee}} \otimes \chi_D \otimes \Delta'_0 \chi'\right)
\end{equation}
as $\mathrm{GL}_2(\mathcal{O}_F) \times \mathcal{O}_D^{\times} \times W'_{F_2}$-representations. 
\end{lemma}
\begin{proof}
The required assertion follows from 
 Lemma \ref{ural}, \eqref{rro},
Lemmas \ref{DL}.2 
 and \ref{zero}. 
\end{proof}
The subgroup 
$F^{\times} \subset G_D$ acts on 
the Lubin-Tate tower trivially (cf.\ \cite[p.\ 20 in \S1.3]{Ca}). 
We regard 
the both sides in \eqref{bb0}
as   
$(F^{\times}(\mathrm{GL}_2(\mathcal{O}_F) \times \mathcal{O}_D^{\times})) \times W'_{F_2}$-representations with trivial $F^{\times}$-action. 
We set 
\begin{align*}
\mathcal{G}_1&=(F^{\times}(\mathrm{GL}_2(\mathcal{O}_F) \times \mathcal{O}_D^{\times}))  W'_{F_2} \simeq (F^{\times}(\mathrm{GL}_2(\mathcal{O}_F) \times \mathcal{O}_D^{\times})) \times W'_{F_2}, \\
\bom{J}_1&=(1,\varpi^{\mathbb{Z}},1)\mathcal{G}_1
=J_{1,1} \times J_{2,1} \times W_{F_2}. 
\end{align*}
We consider the affinoid 
$\X_{1,1} \subset 
\X^{(0)}(\mathfrak{p}) \subset \mathrm{LT}(\mathfrak{p})$, for which we write 
$\X^{(0)}_1$.  
Let $\varpi \colon \mathrm{LT}(\mathfrak{p})
\to 
 \mathrm{LT}(\mathfrak{p})$ be the automorphism induced by the action of  
$\varpi \in F^{\times} \subset D^{\times}$ 
(cf. \cite[\S 2.2.2]{St}). 
Then, we consider 
$\coprod_{i \in \mathbb{Z}} \varpi^i \X^{(0)}_1$, 
on which $(1,\varpi^{\mathbb{Z}},1) \mathcal{G}_1$ acts. 
Recall that 
$\mathrm{GL}_2(\mathcal{O}_F) \times D^{\times}
\times W_F$ acts on $\mathrm{LT}(\mathfrak{p})_{\C}$ as in \S \ref{got}. 
Then, we have a $(1,\varpi^{\mathbb{Z}},1) \mathcal{G}_1$-equivariant injection 
\begin{equation}\label{tor0}
\X_{1,1,\C} \simeq 
\left(\coprod_{i \in \mathbb{Z}} \varpi^i \X^{(0)}_{1,\C}\right)\Big/\varpi^{\mathbb{Z}} \hookrightarrow 
\left(\mathrm{LT}(\mathfrak{p})/\varpi^{\mathbb{Z}}\right)_{\C},  
\end{equation}
where $\X_{1,1,\C}$ 
admits a trivial action of the element 
$\varpi \in F^{\times} \subset D^{\times}$. 
The right hand side of \eqref{tor0}
 is isomorphic to a disjoint union of 
two copies of $\X(\mathfrak{p})$.
\begin{proposition}\label{level1-1}
Let $(F_2/F,\chi)$ be a minimal admissible 
pair such that $l(\chi)=0$ and 
$\chi(\varpi)=1$.
We have a $\bom{J}_1$-equivariant 
injection
\[
\pi_{\chi^{\vee}} \otimes \rho_{\chi} \otimes 
\tau_{\chi} \hookrightarrow 
\overline{\mathcal{U}}_{\mathrm{c}}. 
\]
\end{proposition}
\begin{proof}
By applying Lemma \ref{top}.2
with $W=H_{\rm c}^1(\overline{\X}_{1,1},\overline{\mathbb{Q}}_{\ell})_{\rm cusp}$
by Lemma \ref{DL}.1, \eqref{dat} and \eqref{tor0}, 
we have $\bom{J}_1$-equivariant injections 
\begin{equation}\label{J1}
H_{\rm c}^1(\overline{\X}_{1,1},\overline{\mathbb{Q}}_{\ell})_{\rm cusp} \hookrightarrow 
H_{\mathrm{c}}^1((\mathrm{LT}(\mathfrak{p})/\varpi^{\mathbb{Z}})_{\C}, \overline{\mathbb{Q}}_{\ell})\left(\frac{1}{2}\right)
\subset 
\overline{\mathcal{U}}_{\mathrm{c}}. 
\end{equation}
We set $\chi_0=\chi|_{U_{F_2}^0/U_{F_2}^1}$. 
We have $\chi_0 \in C$. 
We regard $\pi_{\chi_0^{\vee}} \otimes 
{\chi_0}_D 
\otimes \Delta'_0 \chi'_0$ as a $\bom{J}_1$-representation with trivial $(1,\varpi,1)$-action. 
Then, 
this is clearly 
isomorphic to $\La_{\chi^{\vee}} \otimes 
\La'_{\chi} \otimes \Delta_0 \chi$ 
in \S \ref{U} as 
$\bom{J}_1$-representations.
Hence, by Lemma \ref{level10} and \eqref{J1}, we have 
a $\bom{J}_1$-equivariant injection 
\[
\La_{\chi^{\vee}} \otimes \La'_{\chi} \otimes 
\Delta_0 \chi \hookrightarrow 
\overline{\mathcal{U}}_{\mathrm{c}}.
\] 
We consider the usual topology on 
$G_D$, and 
the discrete topology on $W_F$. 
Then, we regard 
$\overline{\mathcal{U}}_{\rm c}$ as a smooth 
representation of $G$. 
Hence, the required assertion follows from 
Frobenius reciprocity
and irreducibility of 
\[
\pi_{\chi^{\vee}} \otimes \rho_{\chi} \otimes
\tau_{\chi}=
\cInd_{\bom{J}_1}^G\left(\La_{\chi^{\vee}} \otimes \La'_{\chi} \otimes 
\Delta_0 \chi\right). 
\]
\end{proof}
\begin{remark}
In \cite{Yo}, he constructs a semi-stable model 
of $\X(\mathfrak{p})$ in any dimensional 
case.  
As a result, 
restricted to the height two case, 
he proves that 
\[
H_{\rm c}^1(X_{{\rm DL},\mathbb{F}},\overline{\mathbb{Q}}_{\ell})
\simeq 
H_{\rm c}^1(\X(\mathfrak{p})_{\C},\overline{\mathbb{Q}}_{\ell}). 
\]
Hence, 
the kernel $H$ of the canonical map 
$H_{\rm c}^1(X_{{\rm DL},\mathbb{F}},\overline{\mathbb{Q}}_{\ell})
\to 
H^1(X_{{\rm DL},\mathbb{F}},\overline{\mathbb{Q}}_{\ell})$
also contributes to the cohomology of 
the generic fiber $\X(\mathfrak{p})_{\C}$.
For $\chi \in (\mathbb{F}_q^{\times})^{\vee}$, 
let $\overline{\mathrm{St}}_{\chi}$ 
denote the twist by $\chi \circ \det$ of the 
Steinberg
representation $\overline{\mathrm{St}}$ of $\mathrm{GL}_2(\mathbb{F}_q)$. 
Then, by \eqref{d4}, the kernel $H$ is isomorphic to
$\bigoplus_{\chi \in (\mathbb{F}_q^{\times})^{\vee}}\left(\overline{\mathrm{St}}_{\chi^{\vee}} \otimes \chi \circ
\Nr_{\mathbb{F}_{q^2}/\mathbb{F}_q}\right)$ as 
$\mathrm{GL}_2(\mathbb{F}_q) \times \mathbb{F}_{q^2}^{\times}$-representations. 
Since we focus on cuspidal representations, 
the analysis in this subsection 
is enough for our purpose. 
\end{remark}

\subsubsection{Unramified case of positive level}
Let $n \geq 2$ be a positive integer.  
We choose an element $\zeta \in \mathbb{F}_{q^2} \setminus \mathbb{F}_q$. 
We consider the $F$-embeddings 
$F_2 \hookrightarrow \mathrm{M}_2(F)$
and 
$F_2 \hookrightarrow D$ as in \eqref{aii} and 
\eqref{ur}
respectively. 
Let $H_{\zeta}^n$ and $H_{\zeta,D}^n$ be as in 
\eqref{decc1} and \eqref{deccd1} respectively. 
We put 
\begin{align*}
V_{\zeta}^n & =\mathfrak{p}^{[\frac{n}{2}]}\mathfrak{C}_1 \times \mathfrak{p}^{[\frac{n-1}{2}]}\mathfrak{C}_2 \subset 
\mathrm{M}_2(\mathcal{O}_F) \times \mathcal{O}_D, \\
K_{\zeta}^n & =1+\mathfrak{p}_{F_2}^{n-1}+\mathfrak{p}^{[\frac{n+1}{2}]}\mathfrak{C}_1 \subset H_{\zeta}^n, \\  
K_{\zeta,D}^n & =1+\mathfrak{p}_{F_2}^{n-1}+\mathfrak{p}^{[\frac{n}{2}]}\mathfrak{C}_2 \subset 
H_{\zeta,D}^n. 
\end{align*} 
Then 
$F_2^{\times} \times F_2^{\times}$
normalizes 
$K_{\zeta}^n \times K^n_{\zeta,D}$ 
and $H_{\zeta}^n \times H_{\zeta, D}^n$ in
$G_D$.
We set 
\begin{align*}
\mathcal{L}_{n,\zeta}=
F^{\times}U_{F_2}^1\left(H_{\zeta}^{n+1} \times 
H_{\zeta,D}^{n+1}\right) & \subset 
\mathcal{P}_{n,\zeta}=F^{\times}U_{F_2}^1\left(K_{\zeta}^{n} \times 
K_{\zeta,D}^{n}\right) \\
& \subset \cK_{n,\zeta}
=F_2^{\times} 
\left(H_{\zeta}^n \times  H_{\zeta, D}^n\right) 
\subset 
G_D.  
\end{align*}
These subgroups are studied in 
\cite[\S 4.3]{WeJL} and \cite[\S 3.6]{WeLT}. 
It is not difficult to check 
that $\mathcal{L}_{n,\zeta}$ is a normal subgroup 
of $\mathcal{K}_{n,\zeta}$. 
Any element of $\mathcal{K}_{n,\zeta}$
has the form 
$\varpi^k \left(x,x+\varpi^{n-1}y\right)+v$
with $k \in \mathbb{Z}$, 
$x \in \mathcal{O}_{F_2}^{\times}$, 
$y \in \mathcal{O}_{F_2}$ and 
$v \in V_{\zeta}^n$. 

In the sequel, we define an isomorphism 
$\phi_n \colon 
\mathbf{V}_{\zeta}^n=V_{\zeta}^n/V_{\zeta}^{n+1} \xrightarrow{\sim} \mathbb{F}_{q^2}$
such that 
\begin{equation}\label{kko}
\phi_n(xvy)=\left(xy^q\right)^{q^{n-1}} \phi_n(v) \quad \textrm{for 
$x,y \in \mathbb{F}_{q^2}$ and $v \in \mathbf{V}_{\zeta}^n$}.
\end{equation} 
Assume that $n$ is odd. 
We set $n=2m-1$. Let 
$v_0=\begin{pmatrix}
1 & \zeta+\zeta^q \\
0 & -1
\end{pmatrix}$. 
We can easily check that  
\begin{gather}\label{kko1}
\begin{aligned}
v_0 \zeta &= \zeta^q v_0, \quad 
g(a,b)
= \left(a+b\zeta \right)  v_0 
\quad  \textrm{for any $g(a,b) \in C_1$}, \\[0.1cm]
(c+d \zeta)^{-1}g(a,b)& =
(c+d \zeta)^{-(q+1)}
g\left(ac+ad(\zeta+\zeta^q)+bd \zeta^{q+1}, bc-ad\right)\ \textrm{for $c+d \zeta \in \mathbb{F}_{q^2}^{\times}$}. 
\end{aligned}
\end{gather}
We have 
$\mathbf{V}_{\zeta}^n  \xleftarrow{\sim}\mathfrak{p}^{m-1}\mathfrak{C}_1/
\mathfrak{p}^{m}\mathfrak{C}_1 \simeq 
\mathbb{F}_{q^2}\overline{\varpi^{m-1}v_0}$. 
We define $\phi_n$ by 
\begin{equation}\label{kko2}
\phi_n\left(x \overline{\varpi^{m-1}v_0}\right)=x
\end{equation}
 for $x \in \mathbb{F}_{q^2}$. 
We can easily check \eqref{kko} by \eqref{kko1}. 

Assume that $n$ is even. 
We set $n=2m$.
We have 
$\mathbf{V}_{\zeta}^n\xleftarrow{\sim}\mathfrak{p}^{m-1}\mathfrak{C}_2/
\mathfrak{p}^{m}\mathfrak{C}_2 \simeq  
\mathbb{F}_{q^2}\overline{\varpi^{m-1}\varphi}$.
We define $\phi_n$ by 
\begin{equation}\label{kko3}
\phi_n\left(\overline{\varpi^{m-1}\varphi} x\right)=x
\end{equation}
 for 
$x \in \mathbb{F}_{q^2}$. 
Then we have \eqref{kko}. 

We have $V_{\zeta}^n V_{\zeta}^n \subset 
\mathfrak{p}_{F_2}^{n-1} \times \mathfrak{p}_{F_2}^{n-1}$. We consider the map 
$\mathfrak{p}_{F_2}^{n-1} \times \mathfrak{p}_{F_2}^{n-1} \to \mathbb{F}_{q^2}$ defined by 
$(x_1,x_2) \mapsto 
\overline{(x_1-x_2)/\varpi^{n-1}}$. 
For $v,w \in V_{\zeta}^n$, we write 
$v \cdot w$ for the image of $vw$ by this map.  
Then, we can check that 
$v \cdot w=\left(\phi_n(v)\phi_n(w)^q\right)^{q^{n-1}}$
for $v,w \in V_{\zeta}^n$. 
Hence, we obtain the isomorphism
\begin{equation}\label{isoo}
\mathcal{K}_{n,\zeta}/\mathcal{L}_{n,\zeta}
\xrightarrow{\sim} Q;\ 
\left(\left(x,x+\varpi^{n-1}y\right)+v\right)^{(-1)^{n-1}} \mapsto 
g\left(\bar{x},\phi_n(v),(-1)^n\bar{y}\right) 
\end{equation}
(cf. \cite[the proof of Proposition 4.3.4]{WeJL}). 
Let $\nu \colon \mathcal{K}_{n,\zeta} \twoheadrightarrow
Q$ denote the composite of 
$\mathcal{K}_{n,\zeta} \to \mathcal{K}_{n,\zeta}/\mathcal{L}_{n,\zeta}$ and the isomorphism 
\eqref{isoo}. 
The image of the subgroup 
$\mathcal{P}_{n,\zeta}$ by the map $\nu$
equals the center $Z$ of $Q$.  
Let $s_1 \colon \mathrm{M}_2(F) \to F_2$ be the
projection and 
$s_2 \colon D \to F_2;\ a+\varphi b \mapsto a$ 
for $a,b \in F_2$ (cf.\ \eqref{dec0} and \eqref{decd0}). We set 
$s \colon \mathrm{M}_2(F) \times D \to F_2;\ (x,y) \mapsto 
s_1(x)-s_2(y)$. 
Let $\psi \in \mathcal{C}$. 
We define a character 
$\widetilde{\psi}_{\zeta}$ of the subgroup 
$\mathcal{P}_{n,\zeta} \subset \mathcal{K}_{n,\zeta}$ 
by 
\[
\widetilde{\psi}_{\zeta}\left(x(1+w)\right)=\psi \left( 
\overline{\varpi^{-(n-1)}s(w)}\right)\ \textrm{for 
$x \in F^{\times}U_{F_2}^1$ and 
$1+w \in K_{\zeta}^n \times  K_{\zeta, D}^n$}. 
\]
Let ${\tau}_{\zeta, \psi}$
denote the inflation of the irreducible $Q$-representation $\tau_{\psi}^0$ in Lemma \ref{hei}.1
by $\nu$. Then, ${\tau}_{\zeta,\psi}$
is the $q$-dimensional 
irreducible representation  of 
$\cK_{n,\zeta}$ satisfying 
\begin{equation}\label{cop}
\tau_{\zeta,\psi}|_{F^{\times}}=\bom{1}^{\oplus q}, 
\quad 
\tau_{\zeta,\psi}|_{\mathcal{P}_{n,\zeta}}=\widetilde{\psi}_{\zeta}^{\oplus q}, \quad
\Tr \tau_{\zeta,\psi}(x)=-1\ \textrm{for $x \in 
\bom{\mu}_{q^2-1}(F_2)
\setminus \bom{\mu}_{q-1}(F)$},   
\end{equation}
where $\bom{1}$ is the trivial character 
of $\mathcal{K}_{n,\zeta}$. 
\begin{remark}
The representation $\tau_{\zeta,\psi}$ appears
in \cite[Theorem 5.0.3]{WeJL} and \cite[Proposition 3.8]{WeLT}.  
\end{remark}
We consider the subgroup of $G$:  
\begin{align*}
W^{(n)}_{F_2}&=
\begin{cases}
\{(1,\varpi^{-n_{\sigma}},\sigma) \in 
G
\mid 
\sigma \in W_{F_2}
\} & \textrm{if $n$ is odd}, \\
 \{(\varpi^{n_{\sigma}},1,\sigma) \in 
G
\mid 
\sigma \in W_{F_2}
\} & \textrm{if $n$ is even}. 
\end{cases}
\end{align*}
We have a natural isomorphism 
$a^{(n)} \colon 
W_{F_2} \xrightarrow{\sim} W^{(n)}_{F_2}$. 
Let $I^{(n)}_{F_2}$ be as in \eqref{if}.  
The subgroup $W^{(n)}_{F_2}$ normalizes 
$I^{(n)}_{F_2}$. 
We consider the subgroups 
$\widetilde{W}^{(n)}_{F_{2,n-1}}=W^{(n)}_{F_{2,n-1}} 
I^{(n)}_{F_2}
\subset \widetilde{W}^{(n)}_{F_2}=W^{(n)}_{F_2}I^{(n)}_{F_2}$ in $G$. 
We have 
\[
W^{(n)}_{F_2} \cap 
I^{(n)}_{F_2}=\left\{(1,1,\sigma) \mid 
\sigma \in I_{F_2},\ \bom{a}^0_{F_2}(\sigma)=1\right\}. 
\]
We define a character $\psi'_{\zeta}$ of 
$\widetilde{W}^{(n)}_{F_{2,n-1}}$ by 
\[
\psi'_{\zeta}
\left(a^{(n)}(\sigma)y\right)=
(-1)^{n_{\sigma}}
\psi\left(\bom{\pi}_{n-1}(\sigma)\right)\ \textrm{for 
$\sigma \in W_{F_{2,n-1}}$ and 
$y \in I^{(n)}_{F_2}$}. 
\]
In the following, we often regard a subgroup 
$H \subset G_D$ 
as a 
subgroup of $G$ by $G_D \hookrightarrow G;\ (g,d) \mapsto 
(g,d,1)$. Then, 
$\widetilde{W}^{(n)}_{F_2}$ normalizes 
$\mathcal{K}_{n,\zeta}$. 
We set $\mathcal{H}_{n,\zeta}
=\mathcal{K}_{n,\zeta} 
\widetilde{W}^{(n)}_{F_{2,n-1}}$.
Note that 
\[
\mathcal{K}_{n,\zeta} \cap 
\widetilde{W}^{(n)}_{F_{2,n-1}}
=
\begin{cases}
\left\{(1,d,1) \mid d \in U_{F_2}^{n-1}\right\} & \textrm{if 
$n$ is odd}, \\[0.2cm]
\left\{(g,1,1) \mid g \in U_{F_2}^{n-1}\right\} & \textrm{if 
$n$ is even}.  
\end{cases}
\]
Then, the representation $\tau_{\zeta,\psi}$
of $\mathcal{K}_{n,\zeta}$ and the character 
$\psi'_{\zeta}$ of $\widetilde{W}^{(n)}_{F_{2,n-1}}$ 
are consistent. 
Further 
$\widetilde{W}^{(n)}_{F_{2,n-1}}$ normalizes 
$\tau_{\zeta,\psi}$, because  
we have $\nu(w^{-1}h w)=\nu(h)$ for any $w \in \widetilde{W}^{(n)}_{F_{2,n-1}}$
and $h \in \cK_{n,\zeta}$. 
We set 
\[
\tau'_{\zeta,\psi}(hw)=\psi'_{\zeta}(w)\tau_{\zeta,\psi}(h)
\quad \textrm{for $h \in \mathcal{K}_{n,\zeta}$ and 
$w \in \widetilde{W}^{(n)}_{F_{2,n-1}}$},  
\]
which 
is a representation of $\mathcal{H}_{n,\zeta}$.

Let $\varpi_n \in \mathscr{F}[\mathfrak{p}_{F_2}^n]_{\mathrm{prim}}$. 
Then, $\mathcal{H}_{n,\zeta}$ stabilizes 
the affinoid $\X_{n,n,\zeta,\varpi_n}$, 
and the action on the reduction 
$\overline{\X}_{n,n,\zeta,\varpi_n}$
is described in 
\S \ref{4}. 
\begin{proposition}\label{ppp0}
We have an isomorphism  
\[
H_{\rm c}^1(\overline{\X}_{n,n,\zeta,\varpi_n},\overline{\mathbb{Q}}_{\ell})\left(\frac{1}{2}\right) \simeq 
\bigoplus_{\psi \in \mathcal{C}} 
\tau'_{\zeta, \psi} 
\]
as $\mathcal{H}_{n,\zeta}$-representations. 
\end{proposition}
\begin{proof}
By using \eqref{kko1}, 
in the notation of Proposition \ref{ag}, we have 
\[
\beta\left(1+\varpi^{[\frac{n}{2}]}(c+d \zeta)^{-1}g(a,b)\right)=
\overline{(a+b \zeta)/(c+d \zeta)}\  \textrm{for $c+d  \zeta \in \mathbb{F}_{q^2}^{\times}$ and 
$g(a,b) \in C_1$}. 
\]
Let $X_0$ be the curve in \S \ref{As}. 
Let $\mathcal{K}_{n,\zeta}$ act 
on $X_0$ through the isomorphism \eqref{isoo}. 
Then, by Propositions \ref{rx}, \ref{ag}, \ref{ad},  Lemma \ref{diggg}, \eqref{act}, \eqref{kko1}, \eqref{kko2} and \eqref{kko3}, 
we have a $\mathcal{K}_{n,\zeta}$-equivariant 
purely inseparable map 
\[
\overline{\X}_{n,n,\zeta,\varpi_n}
\to 
X_0;\ (X,Y) \mapsto \left(X,Y^{q^{n-1}}\right). 
\]
Hence, by Lemma \ref{hei}.1,  
we have isomorphisms 
\[
H_{\rm c}^1(\overline{\X}_{n,n,\zeta, \varpi_n},\overline{\mathbb{Q}}_{\ell}) \xleftarrow{\sim} 
H_{\mathrm{c}}^1(X_{0,\mathbb{F}},\overline{\mathbb{Q}}_{\ell})
\simeq 
\bigoplus_{\psi \in \mathcal{C}} \tau_{\zeta,\psi} 
\]
as $\mathcal{K}_{n,\zeta}$-representations. 
The claim on the action of Weil group 
follows 
from Lemmas \ref{ural}, \ref{gid} and \ref{hei}.2. 
\end{proof}
The group $U^0_{F_2} \times U^0_{F_2}$
normalizes $\cK_{n,\zeta}$ and 
$\widetilde{W}^{(n)}_{F_2}$ respectively. 
We set 
\begin{equation*}
\mathcal{G}_{n,\zeta}
=\left(U^0_{F_2} \times U^0_{F_2}\right)
\cK_{n,\zeta} \widetilde{W}^{(n)}_{F_2}
\subset G. 
\end{equation*}
Then, we have 
\begin{equation}\label{trig}
\mathcal{G}_{n,\zeta}=\Delta_{\zeta}(\varpi)^{\mathbb{Z}} \left(\left(U^0_{F_2} H_{\zeta}^n \times 
U^0_{F_2}H_{\zeta,D}^n\right) W^{(n)}_{F_2}\right). 
\end{equation}
Note that 
\[
\left(U^0_{F_2} H_{\zeta}^n \times 
U^0_{F_2}H_{\zeta,D}^n\right) W^{(n)}_{F_2}
\simeq 
U^0_{F_2} H_{\zeta}^n \times 
U^0_{F_2}H_{\zeta,D}^n \times W^{(n)}_{F_2}. 
\]
The group $\mathcal{G}_{n,\zeta}$ acts on 
$\overline{\X}_{n,\zeta}$ by \S \ref{4}. 
\begin{lemma}\label{ppp01}
We have an isomorphism 
\[
H_{\rm c}^1(\overline{\X}_{n,\zeta},\overline{\mathbb{Q}}_{\ell})
\simeq 
\Ind_{\mathcal{H}_{n,\zeta}}^{\mathcal{G}_{n,\zeta}}
H_{\rm c}^1
(\overline{\X}_{n,n,\zeta,\varpi_n},\overline{\mathbb{Q}}_{\ell})
\]
as $\mathcal{G}_{n,\zeta}$-representations. 
\end{lemma}
\begin{proof}
We can check that 
\begin{equation}\label{index}
[\mathcal{G}_{n,\zeta}:\mathcal{H}_{n,\zeta}]=\left|U^0_{F_2}/U_{F_2}^{n-1}\right|.
\end{equation}
By Frobenius reciprocity, 
we have a $\mathcal{G}_{n,\zeta}$-equivariant injection  
\begin{equation}\label{aa}
H_{\rm c}^1(\overline{\X}_{n,\zeta},\overline{\mathbb{Q}}_{\ell}) \hookrightarrow 
\Ind_{\mathcal{H}_{n,\zeta}}^{\mathcal{G}_{n,\zeta}}
H_{\rm c}^1
(\overline{\X}_{n,n,\zeta,\varpi_n},\overline{\mathbb{Q}}_{\ell}).  
\end{equation} 
By \eqref{ya} and \eqref{index}, the both sides of \eqref{aa}
have the same dimension. Hence, the required assertion follows. 
\end{proof}
We identify $U_{F_2}^{n-1}/U_{F_2}^n$
with $\mathbb{F}_{q^2}$ by $1+\varpi^{n-1}x \mapsto
\bar{x}$ for $x \in \mathcal{O}_{F_2}$. 
Let $\mathcal{C}$ be as in \S \ref{As}. 
We set 
\[
I=\left\{
\chi_0 \in 
(U^0_{F_2}/U_{F_2}^{n})^{\vee}\   
 \big|\ 
{\chi_0}|_{U_{F_2}^{n-1}/U_{F_2}^n}
\in \mathcal{C}
\right\}. 
\]
Let $\chi_0 \in I$. 
We set $\psi={\chi_0}|_{U_{F_2}^{n-1}/U_{F_2}^n} \in 
\mathcal{C}$. 
Let  
$(F_2/F, \chi)$ be the minimal admissible pair 
such that 
\[
l(\chi)=n-1, \quad  
\chi|_{U_{F_2}^0/U_{F_2}^n}=\chi_0, \quad  
\chi(\varpi)=1.
\] 
Let $\La_{\chi^{\vee}}$ and $\La'_{\chi}$
denote the irreducible representations of 
$J_{1,n}$ and $J_{2,n}$ in \S \ref{U}
respectively. 
Note that 
\begin{align*}
U_{F_2}^0 H_{\zeta}^n &=U_{F_2}^0 U_{\mathfrak{M}}^{[\frac{n}{2}]} \subset F_2^{\times}U_{\mathfrak{M}}^{[\frac{n}{2}]}=J_{1,n}, \\  
U_{F_2}^0 H_{\zeta,D}^n &
=U_{F_2}^0 U_D^{n-1} \subset F_2^{\times} U_D^{n-1}
=J_{2,n}, 
\end{align*}
because $U_{F_2}^0 U_D^{2k+1}=
U_{F_2}^0 U_D^{2k}$ for any integer $k \geq 0$.  
We set 
\begin{equation}\label{12a}
\tau_{\zeta,\chi_0}=\La_{\chi^{\vee}}|_{U^0_{F_2} H_{\zeta}^n},
\quad 
\tau^D_{\zeta,\chi_0}=\La'_{\chi}|_{U^0_{F_2} H_{\zeta,D}^n}. 
\end{equation}
Note that one of \eqref{12a} is one-dimensional, 
and the other is $q$-dimensional.  
Let $\xi'_{\zeta,\chi_0}$ 
denote the character of $W^{(n)}_{F_2}$
defined by 
\[
\xi'_{\zeta,\chi_0}\left(a^{(n)}(\sigma)\right)
=(-1)^{n_{\sigma}}
\chi_0\left(\bom{a}^0_{F_2,\varpi}(\sigma)\right)\  
\textrm{for $\sigma \in W_{F_2}$}. 
\]
We put 
\[
\bom{\pi}_{\zeta, \chi_0}
=
\tau_{\zeta, \chi_0} 
\otimes \tau^D_{\zeta, \chi_0} 
\otimes
\xi'_{\zeta, \chi_0}, 
\] 
which can be regarded as a $\mathcal{G}_{n,\zeta}$-representation with trivial action of $\Delta_{\zeta}(\varpi)$ (cf.\ \eqref{trig}). 
\begin{remark}
Let $\chi$ be a character of $F_2^{\times}$
such that $l(\chi)=n-1$ and 
$\chi|_{U_{F_2}^0/U_{F_2}^n} \in I$. 
Then, $(F_2/F,\chi)$ is a minimal 
admissible pair by \cite[Proposition 2 ii) in V\S2]{Se}. 
\end{remark}
\begin{proposition}\label{ppp1}
We have an isomorphism 
\begin{equation}\label{I}
H_{\rm c}^1(\overline{\X}_{n,\zeta},\overline{\mathbb{Q}}_{\ell})\left(\frac{1}{2}\right) 
\simeq  \bigoplus_{\chi_0 \in I}
\bom{\pi}_{\zeta,\chi_0}
\end{equation}
as $\mathcal{G}_{n,\zeta}$-representations.  
\end{proposition} 
\begin{proof}
We set $\psi={\chi_0}|_{U_{F_2}^{n-1}/U_{F_2}^n}
\in \mathcal{C}$. 
We will check
$
\bom{\pi}_{\zeta,\chi_0}|_{\mathcal{H}_{n,\zeta}}
\simeq \tau'_{\zeta, \psi}$.
As mentioned in \S \ref{U}, the restrictions 
$\La_{\chi^{\vee}}|_{U_{F_2}^1U_{\mathfrak{M}}^{[\frac{m+1}{2}]}}$ and $\La'_{\chi}|_{U_{F_2}^1U_D^n}$ are multiples of characters. 
Note that 
\[
H_{\zeta}^{n+1} \subset U_{\mathfrak{M}}^{[\frac{n+1}{2}]}, \quad 
H_{\zeta,D}^{n+1} \subset U_D^n.  
\]
By \S \ref{U}, we can check that 
the restriction $(\tau_{\zeta,\chi_0} \otimes 
\tau^D_{\zeta,\chi_0})|_{\mathcal{L}_{n,\zeta}}$ 
is trivial, and  
the representation 
$(\tau_{\zeta,\chi_0} \otimes 
\tau^D_{\zeta,\chi_0})|_{\cK_{n,\zeta}}$
satisfies \eqref{cop}. Hence, by Lemma \ref{fcf}, 
 we have an isomorphism 
 $\tau_{\zeta,\psi} \simeq (\tau_{\zeta,\chi_0} \otimes 
\tau^D_{\zeta,\chi_0})|_{\cK_{n,\zeta}}$
as $\cK_{n,\zeta}$-representations. 
The restriction 
$\bom{\pi}_{\zeta,\chi_0}|_{I^{(n)}_{F_2}}$
is trivial. Clearly, 
we have 
$\xi'_{\zeta,\chi_0}|_{W^{(n)}_{F_{2,n-1}}}=\psi'_{\zeta}$. Hence, we have $
\bom{\pi}_{\zeta,\chi_0}|_{\mathcal{H}_{n,\zeta}}
\simeq \tau'_{\zeta, \psi}$.

 The representations $\{\bom{\pi}_{\zeta,\chi_0}\}_{\chi_0 \in I}$ are different from 
 each other. 
Hence, by Proposition \ref{ppp0}, 
Lemma \ref{ppp01} and Frobenius reciprocity, 
we have a 
$\mathcal{G}_{n,\zeta}$-equivariant injection 
\begin{equation}\label{tou}
\bigoplus_{\chi_0 \in I}\bom{\pi}_{\zeta,\chi_0}
\hookrightarrow 
H_{\rm c}^1(\overline{\X}_{n,\zeta},\overline{\mathbb{Q}}_{\ell})\left(\frac{1}{2}\right). 
\end{equation}
We have $|I|=q^{2n-3}(q-1)(q^2-1)$ and 
$\dim_{\overline{\mathbb{Q}}_{\ell}}\bom{\pi}_{\zeta,\chi_0}=q$. 
Since the both sides of \eqref{tou} 
are $q^{2(n-1)}(q-1)(q^2-1)$-dimensional by Proposition \ref{rx}, \eqref{ya} and 
Lemma \ref{hei}.1, 
the required assertion follows. 
\end{proof}
We consider the affinoid 
$
\X_{n,\zeta}
\subset 
\X^{(0)}(\mathfrak{p}^n) \subset \mathrm{LT}(\mathfrak{p}^n)$, 
for which we write 
$\X^{(0)}_{n,\zeta}$.  
Let $\varpi \colon \mathrm{LT}(\mathfrak{p}^n)
\to 
 \mathrm{LT}(\mathfrak{p}^n)$ be the 
 automorphism induced by the action of  
$\varpi \in F^{\times} \subset D^{\times}$. 
We set  
\begin{align*}
\bom{J}_n & =(1,\varpi^{\mathbb{Z}},1) \mathcal{G}_{n,\zeta}
=J_{1,n} \times J_{2,n} \times W_{F_2}
\subset \left(F^{\times}\mathrm{GL}_2(\mathcal{O}_F)
\right) \times D^{\times} \times W_F. 
\end{align*}
Then, we consider 
$\coprod_{i \in \mathbb{Z}} \varpi^i \X^{(0)}_{n,\zeta,\C}$, 
on which $\bom{J}_n$ acts. 
Then, we have a $\bom{J}_n$-equivariant injection 
\begin{equation}\label{tor}
\X_{n,\zeta,\C} \simeq 
\left(\coprod_{i \in \mathbb{Z}} \varpi^i \X^{(0)}_{n,\zeta,\C}\right)\Big/\varpi^{\mathbb{Z}} \hookrightarrow 
\left(\mathrm{LT}(\mathfrak{p}^n)/\varpi^{\mathbb{Z}}\right)_{\C},  
\end{equation}
where $\X_{n,\zeta, \C}$ admits the trivial 
action of the element 
$\varpi \in F^{\times} \subset D^{\times}$. 
The right hand side of \eqref{tor} 
is non-canonically isomorphic to a disjoint union of 
two copies of $\X(\mathfrak{p}^n)$.

\begin{proposition}\label{torne}
Let $(F_2/F,\chi)$ be a minimal 
admissible pair 
such that $l(\chi) \geq 1$ and $\chi(\varpi)=1$. 
Then we have a $G$-equivariant 
injection 
\[
\pi_{\chi^{\vee}} \otimes \rho_{\chi} \otimes \tau_{\chi}
\hookrightarrow 
\overline{\mathcal{U}}_{\mathrm{c}}.  
\]
\end{proposition}
\begin{proof}
Let $l(\chi)=n-1$ with $n \geq 2$. 
By Proposition \ref{rx}, Lemmas \ref{top}.2, \eqref{hei},  
\eqref{dat} and \eqref{tor}, we have a $\bom{J}_n$-equivariant injection
\begin{equation}\label{torn}
H_{\mathrm{c}}^1(\overline{\X}_{n,\zeta},\overline{\mathbb{Q}}_{\ell})\left(\frac{1}{2}\right) \hookrightarrow 
H_{\mathrm{c}}^1((\mathrm{LT}(\mathfrak{p}^n)/\varpi^{\mathbb{Z}})_{\C},\overline{\mathbb{Q}}_{\ell})\left(\frac{1}{2}\right)  \subset 
\overline{\mathcal{U}}_{\mathrm{c}}.
\end{equation}
We put $\chi_0=\chi|_{U_{F_2}^0/U_{F_2}^{n+1}}$. 
We have $\chi_0 \in I$. 
We regard $\bom{\pi}_{\zeta,\chi_0}$ 
as a $\bom{J}_n$-representation with 
trivial $(1,\varpi,1)$-action. 
Let 
\[
\bom{\pi}_{\zeta,\chi_0} \simeq 
\La_{\chi^{\vee}} \otimes
\La'_{\chi} \otimes \Delta_0 
\chi
\]
be as in \S \ref{U}, for which we write 
$\bom{\La}_{\chi}$. 
It is clear that 
\begin{equation}\label{ge1}
\bom{\La}_{\chi}(1,\varpi,1)=1, \quad 
\bom{\La}_{\chi}|_{U_{F_2}^0H_{\zeta}^n \times 
U_{F_2}^0 H_{\zeta,D}^n} \simeq \tau_{\zeta,\chi_0}
\otimes \tau^D_{\zeta,\chi_0}. 
\end{equation}
For $\sigma \in W_{F_2}$, we have 
\begin{equation}\label{ge2}
\bom{\La}_{\chi}\left(a^{(n)}(\sigma)\right)
=\Delta_0(\sigma) \chi(\sigma)
=(-1)^{n_{\sigma}} \chi(\sigma)=
(-1)^{n_{\sigma}} 
\chi_0\left(\bom{a}^0_{F_2,\varpi}(\sigma)\right)
=\xi'_{\zeta,\chi_0}\left(a^{(n)}(\sigma)\right). 
\end{equation}
By \eqref{ge1}
 and \eqref{ge2}, we have an isomorphism 
$
\bom{\pi}_{\zeta,\chi_0} \simeq 
\bom{\La}_{\chi}
$
as $\bom{J}_n$-representations.  
Hence, 
by Proposition \ref{ppp1} and  
\eqref{torn}, we obtain 
a $\bom{J}_n$-equivariant injection 
$
\bom{\La}_{\chi}
 \hookrightarrow 
\overline{\mathcal{U}}_{\mathrm{c}}$. 
The required assertion 
follows from Frobenius reciprocity. 
\end{proof}
\subsubsection{Ramified case}
Let $E$ be a totally ramified 
quadratic extension of $F$.
We  take a uniformizer 
$\varpi_E$ of $E$ such that  
$\varpi_E^2 \in F$. We write $\varpi$ 
for $\varpi_E^2$.  
Let $n$ be an odd positive integer. 
Let $H_E^n$ and $H_{E,D}^n$ be as in 
\eqref{dece1} and \eqref{deccde1} respectively. 
We regard $E^{\times}$ as a subgroup of 
$G_D$
via $\Delta_E$ in \eqref{ebe}. 
We set 
\[
\mathcal{P}_{E,n}=
F^{\times}U_E^1\left(H_{E}^n \times H_{E,D}^n\right) \subset 
\mathcal{K}_{E,n}=E^{\times} \mathcal{P}_{E,n} \subset 
G_D. 
\]
We consider 
the projections 
$s_{E,1} \colon \mathrm{M}_2(F) \to E$ and 
$s_{E,2} \colon D \to E$ which are induced by 
\eqref{dec1} and \eqref{decd1} respectively.  
We set 
$s_E \colon \mathrm{M}_2(F) \times D \to E;\ 
(x,y) \mapsto s_{E,2}(y)-s_{E,1}(x)$. 
Let $\psi \in \mathbb{F}_q^{\vee} \setminus 
\{1\}$. 
We set 
\begin{equation}\label{vol}
\widetilde{\psi}\left(x(1+y)\right)=\psi\left(2 \overline{\varpi_E^{-n}s_E(y)}\right)\ \textrm{for $x \in F^{\times} U_E^1$
and 
$1+y \in H_{E}^n \times H_{E,D}^n$}.  
\end{equation}
This is well-defined and determines 
a character of $\mathcal{P}_{E,n}$. 
Furthermore, $E^{\times}$ normalizes the character 
$\widetilde{\psi}$. 
This character extends 
uniquely to a character $\widetilde{\psi}'$ of 
$\mathcal{K}_{E,n}$ 
such that $\widetilde{\psi}'(\varpi_E)=-1$. 

Let 
\[
W'_{E_n}=\{(1,\varphi^{-n_{\sigma}},\sigma) \in 
G \mid 
\sigma \in W_{E_n}\} \subset W'_E=\{(1,\varphi^{-n_{\sigma}},\sigma) \in 
G \mid 
\sigma \in W_E\}.
\]
Let $I''_E$ be as in \eqref{if2}.  
Note that $W'_E$ normalizes $I''_E$. 
We set $\widetilde{W}_{E_n}=W'_{E_n} I''_{E} \subset \widetilde{W}_E=W'_EI''_E$. 
Let $\kappa \colon \{\pm 1\} \to \overline{\mathbb{Q}}_{\ell}^{\times}$ be the non-trivial 
character. 
Then, we have $G_{1,2,1}(\kappa,\psi)=-\tau(\kappa_{E/F},\psi_F)$ in the notation of \eqref{mnc}. 
We define a character $\psi'$  of 
$\widetilde{W}_{E_n}$
by
\begin{itemize}
\item $\psi'(1,\varphi^{-n_{\sigma}},\sigma) =
\left\{-\bigl(\frac{-1}{\mathbb{F}_q}\bigr)^{m-1}
\lambda_{E/F}(\psi_F)\right\}^{n_{\sigma}} 
\psi(2 \bom{\pi}_{E,n}(\sigma))$ 
for $\sigma \in W_{E_n}$, and 
\item $\psi'(1,d,\sigma) =\bigl(\frac{\bar{d}}{\mathbb{F}_q}\bigr)$ for $(1,d,\sigma) \in I''_E$. 
\end{itemize}
We set $\mathcal{H}_{E,n}=
\mathcal{K}_{E,n} \widetilde{W}_{E_n}$. 
Note that 
\begin{align*}
\mathcal{K}_{E,n} \cap \widetilde{W}_{E_n}
&=\{(1,d,1) \mid d \in U_E^n\}, \\
\widetilde{\psi}'|_{\mathcal{K}_{E,n} \cap \widetilde{W}_{E_n}}&=\psi'|_{\mathcal{K}_{E,n} \cap \widetilde{W}_{E_n}}. 
\end{align*}
Hence, $\widetilde{\psi}'$ and $\psi'$ 
determines the character $\bom{\Phi}_{\psi}$ of 
$\mathcal{H}_{E,n}$. 

Let $\varpi_{E,n+1} \in 
\mathscr{G}[\mathfrak{p}_E^{n+1}]_{\mathrm{prim}}$. 
Then, $\mathcal{H}_{E,n}$ stabilizes 
the affinoid $\Z_{n,n,\varpi_{E,n+1}}$, and 
the action on the reduction $\overline{\Z}_{n,n,\varpi_{E,n+1}}$
is described in \S\ref{4}. 
\begin{proposition}\label{rami}
We have an isomorphism 
\[
H_{\rm c}^1(\overline{\Z}_{n,n,\varpi_{E,n+1}},\overline{\mathbb{Q}}_{\ell})\left(\frac{1}{2}\right)
\simeq 
\bigoplus_{\psi \in \mathbb{F}_q^{\vee} \setminus \{1\}}
\bom{\Phi}_{\psi}
\]
as $\mathcal{H}_{E,n}$-representations. 
\end{proposition}
\begin{proof}
By Proposition \ref{rz}, Lemma \ref{ral}, 
all the results in \S \ref{GL} and \S \ref{D}, 
Lemma \ref{ll11}.1, Corollary \ref{lc1}.1 and Lemma \ref{sign}, we have the claim.
\end{proof}
We set 
\begin{align*}
\mathcal{G}_{E,n}=
\left(\left(U^0_E \times U^0_E\right) \mathcal{K}_{E,n}\right)\widetilde{W}_E 
= \left(E^{\times}\left(U^0_EH_E^n \times U^0_E H_{E,D}^n\right)\right)  
W'_E
 \subset G. 
\end{align*}
\begin{lemma}\label{zzz01}
Let $n$ be an odd positive integer. 
Then, we have an isomorphism 
\[
H_{\rm c}^1(\overline{\Z}_{\varpi_E,n},
\overline{\mathbb{Q}}_{\ell})
\simeq \Ind_{\mathcal{H}_{E,n}}^{\mathcal{G}_{E,n}} 
H_{\rm c}^1(\overline{\Z}_{n,n, \varpi_{E,n+1}},
\overline{\mathbb{Q}}_{\ell})
\] 
as $\mathcal{G}_{E,n}$-representations. 
\end{lemma}
\begin{proof}
We can check that 
\begin{equation}\label{index2}
[\mathcal{G}_{E,n}:\mathcal{H}_{E,n}]=
\left|U_E^0/U_E^n\right|. 
\end{equation}
We have a $\mathcal{G}_{E,n}$-equivariant injection
\begin{equation}\label{zaz}
H_{\rm c}^1(\overline{\Z}_{\varpi_E,n},
\overline{\mathbb{Q}}_{\ell})
\hookrightarrow \Ind_{\mathcal{H}_{E,n}}^{\mathcal{G}_{E,n}} 
H_{\rm c}^1(\overline{\Z}_{n,n, \varpi_{E,n+1}},
\overline{\mathbb{Q}}_{\ell}).
\end{equation}
Since the both sides of \eqref{zaz} have the same dimension by \eqref{za} and \eqref{index2}, the claim follows. 
\end{proof}
We identify 
$U_E^n/U_E^{n+1}$ with $\mathbb{F}_q$
by $1+\varpi_E^n x \mapsto \bar{x}$ for 
$x \in \mathcal{O}_E$. 
We set 
\[
I^E=\left\{\chi_0 \in 
\left(U^0_E/U_E^{n+1}\right)^{\vee}\ \Big|\  {\chi_0}|_{U_E^n/U_E^{n+1}} \in \mathbb{F}_q^{\vee} \setminus \{1\}\right\}. 
\]
Let $\chi_0 \in I^E$. 
For a character 
$\phi \in \mathbb{F}_q^{\vee} \setminus \{1\}$, let $\phi_2$ denote the character 
defined by $x \mapsto 
\phi(2x)$ for $x \in \mathbb{F}_q$. 
Let $\psi \in \mathbb{F}_q^{\vee} \setminus \{1\}$
be the character 
such that 
$\psi_2={\chi_0}|_{U_E^n/U_E^{n+1}}$. 
Recall that the character 
$\widetilde{\psi}$ of 
$\mathcal{P}_{E,n}$ 
is normalized by $E^{\times}$. 
We define characters $\psi_{\chi_0}$ of 
$U^0_E H_E^n$ 
and 
$\psi_{\chi_0}^D$ of $U^0_E H_{E,D}^n$ by 
\begin{align*}
\psi_{\chi_0}(xg)&=\chi_0^{\vee}(x)
\widetilde{\psi}(g,1)\ \textrm{for $x \in U^0_E$ and $g \in H_E^n$}, \\  
\psi^D_{\chi_0}(xd)&=\chi_0(x) 
\widetilde{\psi}(1,d)
\ \textrm{for $x \in U^0_E$ and $d \in H_{E,D}^n$}
\end{align*}
respectively (cf. \eqref{vol}). 
Note that 
\[
\left(E^{\times}(U_E^0 H_E^n \times U_E^0 H_{E,D}^n)\right)
\cap W'_E=\{1\}. 
\]
We define a character 
$\bom{\Phi}_{E, \chi_0}$ of $\mathcal{G}_{E,n}$
by 
\begin{itemize}
\item $\bom{\Phi}_{E, \chi_0}|_{U^0_E H_{E}^n \times U^0_E H_{E,D}^n} =\psi_{\chi_0} 
\otimes \psi_{\chi_0}^D$, 
\item $\bom{\Phi}_{E,\chi_0}(\varpi_E)=-1$, and 
\item $\bom{\Phi}_{E,\chi_0}(1,\varphi^{-n_{\sigma}},\sigma)=\left\{-\bigl(\frac{-1}{\mathbb{F}_q}\bigr)^{m-1}
\lambda_{E/F}(\psi_F)\right\}^{n_{\sigma}}
\left(\frac{\overline{\bom{a}^0_{E, \varpi_E}(\sigma)}}{\mathbb{F}_q}\right)\chi_0\bigl(\bom{a}^0_{E, \varpi_E}(\sigma)\bigr)$ for $\sigma \in W_E$.
\end{itemize}
We can directly check that 
\[
\bom{\Phi}_{E,\chi_0}|_{\cK_{E,n}}=\widetilde{\psi}', \quad 
\bom{\Phi}_{E,\chi_0}|_{\widetilde{W}_{E_n}}=\psi'. 
\]
Hence, 
we have 
\begin{equation}\label{ph}
\bom{\Phi}_{E,\chi_0}|_{\mathcal{H}_{E,n}}=\bom{\Phi}_{\psi}. 
\end{equation}
\begin{remark}
Let $\chi$ be a character of $E^{\times}$
such that $l(\chi)=n$ and 
$\chi|_{U_E^0/U_E^{n+1}} \in I^E$.
Then, $(E/F,\chi)$ is a minimal admissible pair, 
because any character of $E^{\times}$, which 
factors through the norm map $\Nr_{E/F}$,
has an even level by 
\cite[Corollary 3 in V\S3]{Se}. 
\end{remark}
\begin{proposition}\label{zzz1}
We have an isomorphism 
\begin{equation}\label{ze}
H_{\rm c}^1(\overline{\Z}_{\varpi_E,n},\overline{\mathbb{Q}}_{\ell})\left(\frac{1}{2}\right)
\simeq 
\bigoplus_{\chi_0 \in I^E} 
\bom{\Phi}_{E, \chi_0}
\end{equation}
as $\mathcal{G}_{E,n}$-representations. 
\end{proposition}
\begin{proof}
Let $\chi_0 \in I^E$. Let $\psi \in \mathbb{F}_q^{\vee} \setminus \{1\}$
be the character 
such that 
$\psi_2={\chi_0}|_{U_E^n/U_E^{n+1}}$. 
By Proposition \ref{rami}, Lemma \ref{zzz01}, \eqref{ph}
and Frobenius reciprocity, 
we have a $\mathcal{G}_{E,n}$-equivariant injection
\begin{equation}\label{tou2}
\bigoplus_{\chi_0 \in I^E}
\bom{\Phi}_{E,\chi_0} \hookrightarrow 
H_{\rm c}^1(\overline{\Z}_{\varpi_E,n},\overline{\mathbb{Q}}_{\ell})\left(\frac{1}{2}\right). 
\end{equation}
Since the both sides of \eqref{tou2} are 
$q^{n-1}(q-1)^2$-dimensional by Proposition \ref{rz}, \eqref{za} and Lemma \ref{lc1}.1, 
we obtain the claim. 
\end{proof}
We set
\[
\mathrm{LT}'(\mathfrak{p}^n)=
\mathrm{LT}(\mathfrak{p}^{m+1})/U_{\mathfrak{I}}^{n+1}, \quad 
\Y^{(h)}(\mathfrak{p}^n)=\X^{(h)}(\mathfrak{p}^n)/U_{\mathfrak{I}}^{n+1}. 
\]
Then, we have 
\[
\mathrm{LT}'(\mathfrak{p}^n)=\coprod_{h \in \mathbb{Z}} \Y^{(h)}(\mathfrak{p}^n). 
\]
Note that $U_{\mathfrak{I}}^{n+1}$ is a normal 
subgroup of 
the standard Iwahori subgroup $\mathfrak{I}^{\times}$, 
and $\varpi_E^{-1} U_{\mathfrak{I}}^{n+1}
\varpi_E=U_{\mathfrak{I}}^{n+1}$.   
Thereby, the product group 
$\bigl(\varpi_E^{\mathbb{Z}}\mathfrak{I}^{\times}\bigr) 
\times 
D^{\times} \times W_F$ acts on 
$\mathrm{LT}'(\mathfrak{p}^n)_{\C}$
 (cf.\ \S \ref{got}).
Let $\varpi \colon \mathrm{LT}'(\mathfrak{p}^n) \to
\mathrm{LT}'(\mathfrak{p}^n)$ be the automorphism 
induced by the action of 
$\varpi \in D^{\times}$. 
We consider the affinoid 
$\Z_{\varpi_E,n} \subset 
\Y^{(0)}(\mathfrak{p}^n) \subset 
\mathrm{LT}'(\mathfrak{p}^n)$, 
for which we write 
$\Z^{(0)}_{\varpi_E, n}$. 
We have 
\[
(1,\varphi^{\mathbb{Z}},1)\mathcal{G}_{E,n} 
\subset 
\left(\varpi_E^{\mathbb{Z}}\mathfrak{I}^{\times}\right) 
\times 
D^{\times} \times W_F. 
\]
Hence, we have a $(1,\varphi^{\mathbb{Z}},1)\mathcal{G}_{E,n}$-equivariant injection 
\begin{equation}\label{torz} 
\Z'_{\varpi_E,n}=
\left(\coprod_{i \in \mathbb{Z}} \varphi^i \Z^{(0)}_{\varpi_E,n,\C}\right)\Big/ \varpi^{\mathbb{Z}}\hookrightarrow
\left(\mathrm{LT}'(\mathfrak{p}^n)/\varpi^{\mathbb{Z}}\right)_{\C}.  
\end{equation}
We set 
\begin{align*}
\bom{J}'_{E,n}=(1,\varpi^{\mathbb{Z}},1) \mathcal{G}_{E,n} \subset 
\bom{J}_{E,n}=(1,\varphi^{\mathbb{Z}},1) \mathcal{G}_{E,n}=J_{E,1,n} \times J_{E,2,n} \times W_E.  
\end{align*} 
We have $[\bom{J}_{E,n} : \bom{J}'_{E,n}]=2$. 
\begin{proposition}\label{tone}
Let $(E/F,\chi)$ be a minimal admissible 
pair such that $\chi(\varpi)=1$. 
Then, we have a $G$-equivariant injection 
\[
\pi_{\chi^{\vee}} \otimes \rho_{\chi} \otimes 
\tau_{\chi} \hookrightarrow \overline{\mathcal{U}}_{\mathrm{c}}. 
\]
\end{proposition}
\begin{proof}
Let $n$ be the level of $\chi$. 
Note that $\Z'_{\varpi_E,n}$ is isomorphic to
a disjoint union of two copies of 
$\Z_{\varpi_E,n,\C}$. 
By Proposition \ref{rz}, Lemma \ref{top}.2, Corollary \ref{lc1}.2, \eqref{dat} and \eqref{torn},  
we have an injective $\bom{J}_{E,n}$-equivariant 
homomorphism 
\begin{gather}\label{hog}
\begin{aligned}
H_{\rm c}^1(\overline{\Z}'_{\varpi_E,n},\overline{\mathbb{Q}}_{\ell})\left(\frac{1}{2}\right) \hookrightarrow
H_{\rm c}^1((\mathrm{LT}'(\mathfrak{p}^n)/\varpi^{\mathbb{Z}})_{\C},\overline{\mathbb{Q}}_{\ell})\left(\frac{1}{2}\right)
\subset & \ 
\overline{\mathcal{U}}_{\mathrm{c}}. 
\end{aligned}
\end{gather}
We have an isomorphism
\begin{equation}\label{2i}
H_{\rm c}^1(\overline{\Z}'_{\varpi_E,n},\overline{\mathbb{Q}}_{\ell})\left(\frac{1}{2}\right) 
\simeq \Ind_{\bom{J}'_{E,n}}^{\bom{J}_{E,n}}
H_{\rm c}^1(\overline{\Z}_{\varpi_E,n},\overline{\mathbb{Q}}_{\ell})\left(\frac{1}{2}\right)
\end{equation}
as $\bom{J}_{E,n}$-representations. 
We set $\chi_0=\chi|_{U_E^0/U_E^n}$. 
We have $\chi_0 \in I^E$. 
For $\iota \in \{\pm 1\}$, 
let $\bom{\Phi}^{\iota}_{E,\chi_0}$ be the character 
of $\bom{J}_{E,n}$ 
such that 
$\bom{\Phi}^{\iota}_{E,\chi_0}|_{\mathcal{G}_{E,n}}=\bom{\Phi}_{E,\chi_0}$ and 
$\bom{\Phi}^{\iota}_{E,\chi_0}(1,\varphi,1)
=\iota$. 
Then, by Proposition \ref{zzz1}, 
\eqref{hog} and \eqref{2i}, 
we have 
\begin{equation}\label{2ii}
\bom{\Phi}^{\iota}_{E,\chi_0} \subset 
H_{\rm c}^1(\overline{\Z}'_{\varpi_E,n},\overline{\mathbb{Q}}_{\ell})\left(\frac{1}{2}\right)
\subset \overline{\mathcal{U}}_{\rm c}. 
\end{equation}
Assume that $\chi_{\iota}(\varpi_E)=\iota$. 
Let 
\[
\La_{E,\chi_{\iota}^{\vee}} \otimes \La'_{E,\chi_{\iota}} \otimes 
\Delta_{E,\chi_{\iota}} \chi_{\iota}  
\] 
be the $\bom{J}_{E,n}$-representation 
defined in \S \ref{R}. 
We simply write $\bom{\La}_{\chi_{\iota}}$ for it. 
Then, 
we easily check that 
\begin{equation}\label{ga1}
\bom{\La}_{\chi_{\iota}}(1,\varphi,1)=-\iota, \quad 
\bom{\La}_{\chi_{\iota}}|_{U_E^0H_E^n \times U_E^0 H_{E,D}^n}=\psi_{\chi_0} \otimes \psi^D_{\chi_0}, \quad 
\bom{\La}_{\chi_{\iota}}(\varpi_E,\varpi_E,1)=-1. 
\end{equation}
The element  
$\zeta(\alpha,\chi_{\iota})$ in the notation of 
\S \ref{R} equals $1$. 
Hence, for $\sigma \in W_E$,   
we have 
\begin{gather}\label{ga2}
\begin{aligned}
\bom{\La}_{\chi_{\iota}}(1,\varphi^{-n_{\sigma}},\sigma)
&=(-\iota)^{n_{\sigma}}\Delta_{E,\chi}(\sigma)\chi(\sigma) \\
&=  
\left(-\lambda_{E/F}(\psi_F)^{n}\right)^{n_{\sigma}} 
\Delta_{E,\chi}\left(\bom{a}^0_{E,\varpi_E}(\sigma)\right) 
\chi_0\left(\bom{a}^0_{E,\varpi_E}(\sigma)\right)\\
&=\bom{\Phi}^{\iota}_{E,\chi_0}(1,\varphi^{-n_{\sigma}},\sigma), 
\end{aligned}
\end{gather}
where we use the definition of 
$\Delta_{E,\chi_{\iota}}$ and \eqref{lam} 
at the second and 
the last equalities. 
By \eqref{ga1} and \eqref{ga2}, as $\bom{J}_{E,n}$-representations, 
$\bom{\Phi}^{\iota}_{E,\chi_0}$ is isomorphic to 
$\bom{\La}_{\chi_{\iota}}$. 
Therefore, by \eqref{2ii} and 
Frobenius reciprocity, 
we obtain the claim. 
\end{proof}
\subsubsection{Conclusion}
As a result of the analysis in this section, 
we obtain Proposition \ref{Main}. 
\begin{corollary}
Proposition \ref{Main} holds. 
\end{corollary}
\begin{proof}
The required assertion follows from 
Propositions \ref{level1-1}, \ref{torne} and \ref{tone}. 
\end{proof}

\noindent
Takahiro Tsushima\\ 
Department of Mathematics and Informatics, 
Faculty of Science, Chiba University
1-33 Yayoi-cho, Inage, 
Chiba, 263-8522, Japan \\
tsushima@math.s.chiba-u.ac.jp\\ 
\end{document}